\documentclass[10pt]{article}

\usepackage{bm}
\usepackage{geometry}
\usepackage{amsmath}
\usepackage{amsfonts}
\usepackage{amssymb}
\usepackage{amsthm}
\usepackage{mathrsfs,mathtools}
\usepackage{authblk}
\usepackage{hyperref}
\usepackage{microtype}
\usepackage{color,xcolor}
\usepackage[cyr]{aeguill}
\usepackage{enumerate}

\colorlet{darkblue}{blue!50!black}
\hypersetup{
	colorlinks,%
	citecolor=blue,%
	filecolor=red,%
	linkcolor=darkblue,%
	urlcolor=blue,%
	pdfnewwindow=true,%
	pdfstartview={FitH}
}

\numberwithin{equation}{section}

\theoremstyle{plain}
\newtheorem{theorem}{Theorem}[section]
\newtheorem{lemma}[theorem]{Lemma}
\newtheorem{prop}[theorem]{Proposition}
\newtheorem{coro}[theorem]{Corollary}
\newtheorem{assumption}[theorem]{Assumption}
\theoremstyle{definition}
\newtheorem{defn}[theorem]{Definition}
\newtheorem{remark*}{Remark}
\theoremstyle{remark}
\newtheorem{remark}[theorem]{Remark}

\numberwithin{equation}{section}

\DeclareMathOperator{\dive}{div}

\DeclareMathOperator{\loc}{loc}
\DeclareMathOperator{\supp}{supp}

\DeclareMathOperator{\erf}{Erf}
\DeclareMathOperator{\co}{co}
\DeclareMathOperator{\sgn}{sgn}

\newcommand{\PP}{{\mathbb P}}
\newcommand{\R}{{\mathbb R}}
\newcommand{\F}{{\mathbb F}}
\newcommand{\ddd}{{\mathrm d}}
\newcommand{\N}{{\mathbb N}}
\newcommand{\X}{{\mathbb X}}
\newcommand{\NN}{{\mathcal N}}
\newcommand{\MM}{{\mathcal M}}
\newcommand{\UU}{{\mathcal U}}
\newcommand{\TT}{{\mathcal T}}
\newcommand{\XX}{{\mathcal X}}

\newcommand{\LL}{{\mathcal L}}
\newcommand{\JJ}{{\mathcal J}}

\begin{document}
    \title{Local and global existence for the stochastic Prandtl equation with multiplicative noises in two and three dimensions}
    
    \author{Ya-Guang Wang\footnote{email address: ygwang@sjtu.edu.cn}}
    \affil{School of Mathematical Sciences, Center for Applied Mathematics, MOE-LSC and SHL-MAC, Shanghai Jiao Tong University, 200240 Shanghai, China}
    \author{Meng Zhao\footnote{corresponding author, email address:  mzhao1009@gmail.com}}
    \affil{Department of Mathematics, CY Cergy Paris University, CNRS UMR 8088, 2 avenue Adolphe Chauvin, 95302 Cergy-Pontoise, France}

    \date{}
	\maketitle
	\fontsize{12}{15}
	\selectfont
    \begin{abstract}
        We study the analytic and Gevrey solutions of the stochastic Prandtl equation in two and three dimensions with suitable nonlinear multiplicative noises, which describes the velocity field within the boundary layer arising in the inviscid limit of the stochastic Navier--Stokes equation with no-slip boundary conditions. In the stochastic setting, the energy functional (a key quantity for characterizing the analytic radius of solutions in the existing literature of the deterministic setting) fails to admit a pathwise control, creating a crucial challenge. To address this, we derive higher-order estimates in conormal Sobolev spaces, eliminating the dependence of the analytic radius on the unknown. A local wellposedness result is obtained in a tangentially analytic and normally Sobolev-type space, by introducing three approximation schemes of the equation. Furthermore, for the stochastic Prandtl equation with tangential random diffusion, we observe that the noise exhibits a regularizing effect, in the sense that the two and three dimensional stochastic Prandtl equations admit  global Gevrey-2 solutions with high probability and a linearly growing (in time) Gevrey radius, and the linearized two-dimensional equation around any shear flow admits a Gevrey-4 solution. 

    \end{abstract}

{\textbf{\scriptsize{2020 Mathematics Subject Classification:}} \scriptsize{35Q30, 35R60, 60H15, 76D10}}

\textbf{\scriptsize{Keywords:}} {\scriptsize{Stochastic Prandtl equation, analytic and Gevrey solutions, regularizing effect of noises}}

    \tableofcontents
    \section{Introduction}\label{Intro}
    In this paper, we study the stochastic Prandtl equation in two and three dimensions:
    \begin{align}\label{Intro1}
    \begin{cases}
    \mathrm{d} u+(u\cdot\nabla_{x}  u+v\partial_y u-\partial_y^2 u+\nabla_x P)\mathrm{d}t=\F(x,y,u)\mathrm{d}W,\qquad (x,y)\in\R^{d}_+,\\
    \nabla_x \cdot u+\partial_y v=0,\\
u|_{y=0}=v|_{y=0}=0,\qquad \lim\limits_{y\to\infty}u=U,\\
u|_{t=0}=u_0,
    \end{cases}
    \end{align}
    where $\R^{d}_+:=\R^{d-1}\times\R_+$ with $d=2$ or $3$. For $d=2$, $u:= u$ is a scalar and $\nabla_x :=\partial_x$; while for $d=3$, $u:=(u^1,u^2), \nabla_x:=(\partial_{x_1},\partial_{x_2})$. The unknowns $(u,v)$ represent respectively the tangential and normal velocities of fluid  in the boundary layer, while $(U, P)$ denote respectively the traces at $\{y=0\}$ of the tangential velocity and the pressure of the corresponding stochastic Euler flow, which is given as a priori and obeys the following stochastic Bernoulli law:
    \begin{align}\label{Intro2}\mathrm{d} U+(U\cdot\nabla_x U+\nabla_x P)\mathrm{d}t=\overline{\F}(x,U)\mathrm{d}W,\qquad (t,x)\in\R_+\times \R^{d-1}.\end{align}
    The driving process $W$ is a cylindrical Wiener process defined on some filtered probability space $(\Omega,\mathscr{F},\{\mathscr{F}_t\}_{t\ge0},\PP )$ that satisfies the usual condition, cf. Definition 2.25 in \cite{KS91}, and $\F$ is a nonlinear map satisfying certain conditions, see Assumption \ref{assumption}, with $\overline{\F}(x,\cdot)$ denoting the limit of $\mathbb{\F}(x,y,\cdot)$, as $y\to\infty$.
    
    The deterministic Prandtl equation (corresponding to the case $\F\equiv 0$ in \eqref{Intro1}) was first introduced by Ludwig Prandtl in his seminal work \cite{P1904}. This equation models the boundary layer—a thin viscous region that develops near solid boundaries in the inviscid limit of the Navier--Stokes (NS, for short) equation under no-slip boundary conditions. The wellposedness of the Prandtl equation is one of the key steps to rigorously justify the Prandtl theory. However, the lack of tangential diffusion in \eqref{Intro1}, together with the nonlinear convection term $v\partial_y u=-\partial_y^{-1}(\nabla_x\cdot u)\partial_y u$, results in a loss of one tangential derivative in the process of energy estimate, posing a significant challenge when establishing the wellposedness theory in general Sobolev spaces. Mainly, the wellposedness of the deterministic Prandtl equation holds for the following three kinds of initial data.

     \begin{itemize}
        \item Under a monotonic assumption of the tangential velocity in the normal variable, Oleinik and her collaborators \cite{Oleinik63, OS1999} obtained the first local existence and uniqueness result for classical solutions of the two-dimensional Prandtl equation. With an additional favorable condition on the pressure $P$ of the outflow, the global existence of weak solutions was obtained in \cite{XZ2004, XinZhangZhao}. The main technique used in these works is the Crocco transformation, which reduces the two-dimensional  Prandtl equation into a scalar degenerate parabolic equation. More recently, the local wellposedness theory in finite order Sobolev spaces under the monotonic assumption has been re-addressed independently in \cite{AWXY2015} and \cite{MW2015} through direct energy method. A wellposedness result of classical solutions for the three-dimensional Prandtl equation was obtained in \cite{LWY2017} under a special structural assumption. 
        
        \item For initial data with analytic regularity in both tangential and normal variables, Sammartino and Caflisch \cite{SC1998} proved local wellposedness for the Prandtl equation by applying the abstract Cauchy--Kowalewskaya theorem. The analyticity in the normal variable was later removed in \cite{LCS2003}. Recently, direct energy-based method was introduced to establish both local and global wellposedness for tangentially analytic data; see \cite{KuVicol13,ZZ2016, IV2016, PZ2021}.
       
        \item For initial data with Gevrey-$\frac{7}{4}$ regularity in tangential variables, a local wellposedness result was obtained in \cite{GeMasmoudi15} for the two-dimensional Prandtl equation. This result was later improved to the optimal Gevrey-2 class by Li and Yang \cite{LiYang20}, under the assumption that the tangential velocity possesses a single non-degenerate critical point. This regularity requirement is indeed sharp, as indicated by the linear instability result obtained in \cite{GD2009}. Subsequent work by Dietert and Gérard-Varet \cite{DG2018} removed the non-degenerate critical point assumption; while Wang, Wang, and Zhang \cite{WWZ2021} obtained
        global existence for small data. More recently, for the three-dimensional Prandtl equation with Gevrey-2 initial data, the local existence of solutions was established in \cite{LiMYang22}; while an almost global existence result for small data was obtained in \cite{PanXu}.
        
    \end{itemize}

    However, the stochastic counterpart of the Prandtl equation and the boundary layer problem remains completely unexplored. To the best of our knowledge, this is the first result conducting analysis on the Prandtl equation in the stochastic regime. Considering the wide usage of randomness in describing numerical, empirical, and physical uncertainties in fluid mechanics, it is important to strengthen the mathematical foundations of the stochastic boundary layers, which is one of the main concerns addressed in this paper.
    
    The first part of the paper is devoted to the local wellposedness theory of the stochastic Prandtl equation \eqref{Intro1}. It is noteworthy that the noise term does not involve the normal velocity $v$, and consequently the addition of the noise does not deteriorate the structure of the equation, as $v$ is the bad term that causes the loss of derivatives. We emphasize that equation \eqref{Intro1} is physically derived by a multi-scale analysis of the stochastic NS equation, aligning with Prandtl's original idea; see Appendix \ref{derivation} for details. The result on the local wellposedness theory in the  tangentially analytic space is presented in Section \ref{mainresult1}. From a probabilistic point of view, this solution is a pathwise one, where the driving noise and the associated filtration are given as a priori; while in terms of PDEs, we consider solutions which evolve continuously in time with values in a tangentially analytic and normally Sobolev-type space. 

     Compared with its deterministic counterpart, in order to obtain this local wellposedness, we have to address the following new problems arising from the stochastic nature.

     \begin{itemize}
         \item The first problem is to characterize the time evolution of the analytic radius of the solution. In  the deterministic literature, cf. \cite{ZZ2016,PZ2021}, the analytic radius $\sigma(t)$ is set to be a functional of the energy $\mathcal{E}_u(t)$. Therefore, one can establish not only a uniform bound for the approximate solutions $u^{\epsilon}$, but also a uniform lower bound for their analytic radius $\sigma^{\epsilon}(t)$:
         \[\sigma_0(t)\le \sigma^{\epsilon}(t)\]
         through direct energy estimate. Consequently, a uniform lower bound is valid for the lifespan of $u^\epsilon$: 
    \[\tau^{\epsilon}:=\inf\{t\ge 0| \sigma^{\epsilon}(t)=0\}\ge\tau_0:=\inf\{t\ge 0| \sigma_0(t)=0\},\]
    so that one may pass the limit $\epsilon\to 0^+$ on the interval $[0,\tau_0)$. However, in the stochastic case, the energy functional $\mathcal{E}_u(t)$ admits only moment estimates, which is insufficient to derive neither a pathwise lower bound $\sigma_0(t)$ nor a uniform lifespan $\tau_0$ for the approximate solution $u^{\epsilon}$. We resolve this difficulty by introducing another characterization of the analytic radius which is independent of the unknown, together with normally higher-order Sobolev spaces, in order to close the energy estimate.

    \item However, when establishing normally higher-order estimates in the stochastic case, mild singularity may appear at the boundary, as observed in \cite{Kry03} in the case of a linear parabolic SPDE. This is because the boundary conditions for the normal derivatives $\partial_y^ju$ may involve the noise term, resulting in a mild blow-up of order $O(y^{-\alpha})$ at the boundary $\{y=0\}$. To remedy the boundary singularity, one could introduce ``compatibility conditions" on the noise, as proposed in \cite{Flan1990,D2020}, to prevent the noise from factoring into the boundary conditions. However, this would lead to a series of unnecessary and complicated conditions, because of the nonlinear multiplicative noises. Therefore, inspired by \cite{Kry1994}, we adopt another strategy: incorporate the boundary singularity by using the conormal derivatives $Z^j:=y^j\partial_y^j$.
     \end{itemize}
     
     In the second part, we prove probabilistic global existence for the stochastic Prandtl equation \eqref{Intro1} driven by noises with a special structure. For general nonlinear multiplicative noises, the result seems to be inaccessible. In fact, establishing global existence for the Prandtl equation needs to control both potential finite-time blow-up and more subtle ``blow-down" phenomenon, where the latter means the analytic radius must remain strictly positive, as $t\to\infty$. Consequently, existing probabilistic global existence results for general nonlinear noises, e.g. \cite{KXZ2023,LLW2024}, which are mainly based on the prevention of blow-up, do not fit into our situation. On the other hand, the ``blow-down" phenomenon is excluded in the deterministic case, cf. \cite{PZ2021,WWZ2021}, through a quantitative control of the analytic radius via the energy $\mathcal{E}_u$. However, as discussed previously, this strategy fails in the stochastic case, due to the lack of pathwise estimate of $\mathcal{E}_u$.
     
     Inspired by the recent works \cite{GV2014,BNSW2020,RS2023}, we study the stochastic Prandtl equation \eqref{Intro1} with a tangential random diffusion
     \[\F(u):=(\alpha_1+\alpha|\nabla_{x}|^{a
    })u,\]
    and a uniform outflow
    \[U:=\UU\vec{e}, \quad \nabla_{x}P:=-\beta U,\]
    where $\beta,a\ge 0$ and $\alpha_1,\alpha\in\R$ are constants, $\UU$ is independent of the tangential variables, and $\vec{e}\in\R^{d-1}$ is a constant vector. Specifically, we investigate the following stochastic Prandtl equation:
\begin{align}\label{Intro3}
    \begin{cases}
    \mathrm{d} u+(u\cdot\nabla_{x}  u+v\partial_y u-\partial_y^2 u-\beta \UU\vec{e})\mathrm{d}t=(\alpha_1+\alpha|\nabla_{x}|^{a})u\mathrm{d}B_t,\\
    \nabla_{x}\cdot  u+\partial_y v=0,\\
u|_{y=0}=v|_{y=0}=0,\qquad \lim\limits_{y\to\infty}u=\UU\vec{e},\\
u|_{t=0}=u_0,
    \end{cases}
    \end{align}
    where $\{B_t\}_{t\ge 0}$ is a standard Brownian motion, and $\UU$ satisfies
\begin{align}\label{1.4}
    \begin{cases}
    \mathrm{d}\UU = \beta \UU\mathrm{d}t+\alpha_1 \UU \mathrm{d}B_t,\\
    \UU|_{t=0}=\UU_0.
    \end{cases}
    \end{align}
Let us emphasize that the stochastic term $\alpha|\nabla_x|^au\mathrm{d}B_t$ does not have a sign, and thus does not a priori regularize the equation.
    
    Our second main result is presented in Section \ref{mainresult2}, which consists of two parts. First, in the three-dimensional case, we set $a=\frac{1}{2}$ and establish the global existence, with high probability, of Gevrey-2 solutions with a linearly growing (in time) Gevrey radius, which reveals a regularizing effect of the noise. Second, in the two-dimensional case, we further investigate this regularization phenomenon by considering the initial data beyond the optimal Gevrey regularity of the corresponding deterministic problem. More precisely, we take the parameters $\beta=\alpha_1=0,a=\frac{1}{4}$ and linearize the resulting equation around any shear flow $(u^s,0)$:
    \begin{align}\label{Intro3-1}
    \begin{cases}
    \mathrm{d} u+(u^s\partial_x u+v\partial_yu^s-\partial_y^2u)\mathrm{d}t=\alpha|\partial_{x}|^{\frac{1}{4}
    }u\mathrm{d}B_t,\\
    \partial_x  u+\partial_y v=0,\\
    u|_{y=0}=v|_{y=0}=0, \qquad u|_{t=0}=u_0. 
    \end{cases}
    \end{align}
    We shall prove that, with high probability, the above linearized equation has a global Gevrey-4 solution, with the Gevrey radius growing linearly in time. This represents a significant complement to the deterministic theory \cite{GD2009,DG2018}, where the optimal regularity is Gevrey-2 when $u^s$ is non-monotonic, thereby demonstrating that stochastic effects can partially mitigate the inherent singularities of the Prandtl equation.
    
    To illustrate the regularization mechanism, let us apply the random multiplier
    \begin{align}\label{Intro3_1}
        \Gamma_t:=\UU^{-1}(t)e^{(\sigma_0+\lambda t-\alpha B_t)|\nabla_{x}|^{a}},\qquad u_{\Gamma}:=\Gamma_tu
    \end{align}
    on both sides of \eqref{Intro3}. Utilizing the It\^o formula, we arrive at the following deterministic PDE with random parameter
    \begin{align}\label{Intro4}
    \begin{cases}
    \partial_t u_{\Gamma}+\beta (u_{\Gamma}-\vec{e})+\left((\alpha\alpha_1-\lambda)|\nabla_{x}|^{a}+\frac{\alpha^2}{2}|\nabla_{x}|^{2a}\right)u_{\Gamma}\\
    \hspace{1.23in}\quad
    +\left(u\cdot\nabla_{x} u\right)_{\Gamma}-\left(\partial_y^{-1}(\nabla_{x}\cdot u)\partial_y u\right)_{\Gamma}=\partial_y^2u_{\Gamma},\\
u_{\Gamma}|_{y=0}=0,\qquad \lim\limits_{y\to\infty}u_{\Gamma}=\vec{e},\\
u_{\Gamma}(0)=\UU_0^{-1}e^{\sigma_0|\nabla_{x}|^{a}}u_0.
    \end{cases}
    \end{align}
    There are two terms worthy noting in the above equation: the damping term 
 $\beta(u_{\Gamma}-\vec{e})$ and the dissipation term $\frac{\alpha^2}{2}|\nabla_{x}|^{2a}u_{\Gamma}$. The former arises from the outflow $\UU$ and prevents the solution from blow-up, while the latter stems from the random force and controls the radius from ``blow-down". Notably, the dissipation term $\frac{\alpha^2}{2}|\nabla_{x}|^{2a}u_{\Gamma}$ is enhanced by the quadratic variation of the noise, allowing us to characterize the radius through the stochastic process $\sigma(t):= \sigma_0+\lambda t-\alpha B_t$ with $\lambda>0$ and $\alpha\neq 0$. This stands in contrast to the deterministic setting, where only weaker dissipation $|\nabla_{x}|^a u_{\Gamma}$ can be obtained from the radius characterization. The combination of these two terms forms the foundation for establishing our global existence results.

    The paper is organized as follows. In Section \ref{pre}, we review the deterministic and stochastic background needed throughout the paper, and present the conditions posed on the nonlinearity $\F$ and related estimates. Section \ref{mainresults} is devoted to the precise statements of the main results. Then, we perform some a priori estimates in Section \ref{AP}. In Section \ref{LWP}, we combine these estimates with approximation schemes to establish local wellposedness for the stochastic Prandtl equation \eqref{Intro1}. Section \ref{GWP} is devoted to the construction of global solutions and the phenomenon of regularization-by-noise. Finally, in appendices, we present the derivation of the stochastic Prandtl equation, the proof of auxiliary estimates given in Lemma \ref{product_estimate1} and Proposition \ref{Proposition_struc1}, and the existence of analytic solutions of the stochastic Bernoulli law \eqref{Intro2}.

    \section{Preliminaries}\label{pre} 
    In this section, we introduce the function space of the solutions and related product estimates in this space, together with the stochastic framework and estimates of the nonlinear multiplicative noise. These results will be used throughout the paper.
    
    \subsection{Functional framework and product estimates}
    For $k\ge 1$, we introduce the conormal derivative with respect to the boundary $\{y=0\}$:
    \[Z^k:=y^k\partial_y^k.\]
    
    \begin{defn}
    For $s\ge 1$ and $\gamma>0$, we define the weighted conormal Sobolev spaces
    \[H^s_{\gamma,\co}:=\left\{u:\R^d_+\to \R\bigg|\|u\|^2_{H^s_{\gamma,\co}}:=\sum_{|k|+j\le s}\|e^{\gamma y^2} Z^j\partial_x^k u\|^2_{L^2}<\infty\right\},\]
    and
    \[X^s_{\gamma}:=\left\{u:\R^d_+\to \R\bigg| \|u\|_{X^s_{\gamma}}^2:=\|u\|^2_{H^s_{\gamma,\co}}+\|\partial_y u\|^2_{H^{s-1}_{\gamma,\co}}<\infty\right\}.\]
    For $s_1,s_2\in\mathbb{N}$ and $\gamma>0$, the anisotropic weighted conormal Sobolev space is defined as
        \[H^{s_1,s_2}_{\gamma,\co}:=\left\{u:\R^d_+\to \R\bigg|\|u\|^2_{H^{s_1,s_2}_{\gamma,\co}}:=\sum_{|k|\le s_1, j\le s_2}\|e^{\gamma y^2} Z^j\partial_x^k u\|^2_{L^2}<\infty\right\}.\]
    \end{defn}

    It is clear to see that
    \begin{align}
        \label{relation}\sum_{s_1+s_2=s}\|u\|_{H^{s_1,s_2}_{\gamma,\co}}^2\lesssim_s\|u\|_{H^s_{\gamma,\co}}^2\lesssim_s\sum_{s_1+s_2=s}\|u\|_{H^{s_1,s_2}_{\gamma,\co}}^2,
    \end{align} 
    where the notation $A\lesssim_s B$ means that $A\le C_s B$ for a constant $C_s$ depending only on~$s$.
    
    \begin{defn}
    For $\sigma>0$, we write
    \[u_{\sigma}:=e^{\sigma|\nabla_x|}u=\mathcal{F}_x^{-1}(e^{\sigma|\xi|}\mathcal{F}_x u),\]
    where $\mathcal{F}_x$ denotes the Fourier transform in $x$-variables. For $s\ge1$ and $\sigma,\gamma>0$, we define the analytic-Sobolev space as
    \[X^{s}_{\gamma,\sigma}:=\left\{u:\R^d_+\to \R| u_{\sigma}\in X^{s}_{\gamma}\right\}\]
    endowed with the norm $\|u\|_{X^{s}_{\gamma,\sigma}}:=\|u_{\sigma}\|_{X^{s}_{\gamma}}$.
    \end{defn}
    
    The following product estimates in $X^{s}_{\gamma,\sigma}$ will be used later in the energy estimates. 
    \begin{lemma}\label{product_estimate1}
    For any $s\ge 4$ and $\sigma,\gamma>0$, the following estimates hold:
    \begin{enumerate}
        \item $\|uv\|_{X^s_{\gamma,\sigma}}\lesssim_{s} \|u \|_{X^s_{\gamma,\sigma}}\|v \|_{X^{s-1}_{\gamma,\sigma}}+\|u \|_{X^{s-1}_{\gamma,\sigma}}\|v \|_{X^{s}_{\gamma,\sigma}}$,
        \item $\|\partial_y^{-1}v\partial_yu\|_{X^s_{\gamma,\sigma}}\lesssim_{s,\gamma}\|\partial_yu\|_{X^{s-1}_{\gamma,\sigma}}\|v\|_{X^s_{\gamma,\sigma}}+\|\partial_yu\|_{X^{s}_{\gamma,\sigma}}\|v\|_{X^{s-1}_{\gamma,\sigma}}$,
        \item $\|\langle \nabla_x\rangle^{-\frac{1}{2}}(u\cdot \nabla_x v)\|_{X^s_{\gamma,\sigma}}\lesssim_{s} \|u \|_{X^s_{\gamma,\sigma}}\|\langle \nabla_x\rangle^{\frac{1}{2}} v \|_{X^s_{\gamma,\sigma}}$,
        \item $\|\langle \nabla_x\rangle^{-\frac{1}{2}}[(\partial_y^{-1}\nabla_x \cdot u) \partial_y v]\|_{X^s_{\gamma,\sigma}}\lesssim_{s,\gamma} \|u\|_{X^s_{\gamma,\sigma}}\| \partial_y v\|_{X^s_{\gamma,\sigma}}+\|v\|_{X^s_{\gamma,\sigma}}\|\langle\nabla_x\rangle^{\frac{1}{2}}u\|_{X^s_{\gamma,\sigma}}$,
    \end{enumerate}
    where $\langle \nabla_x \rangle:=(1+|\nabla_x|^2)^{\frac{1}{2}}$ and $\partial_y^{-1}u(x,y):=\int_0^yu(x,\tilde{y})\mathrm{d}\tilde{y}.$
    \end{lemma}
    As the state of $u$, as $y\to\infty$, is inhomogeneous, we need the following estimates to control the terms arising from the state at far field. 
    \begin{lemma}\label{product_estimate2}
        For any $s\ge 4$ and $\sigma,\gamma>0$,  we have 
        \begin{enumerate}
            \item $\|U v\|_{X^{s}_{\gamma,\sigma}}\lesssim_{s}\|U\|_{H^s_{x,\sigma}}\|v\|_{X^{s}_{\gamma,\sigma}}$,
            \item $\|\langle\nabla_x\rangle^{-\frac{1}{2}}(U\nabla_x v)\|_{X^s_{\gamma,\sigma}}\lesssim_s \|U \|_{H^s_{x,\sigma}}\|\langle \nabla_x\rangle^{\frac{1}{2}} v \|_{X^s_{\gamma,\sigma}},$
        \end{enumerate}
        where 
        \[H_{x,\sigma}^s:=\{U: \R^{d-1}\to \R| U_{\sigma}\in H_x^s\}\]
    with the norm $\|U\|_{H_{x,\sigma}^s}:=\|U_{\sigma}\|_{H_x^s}$, and $H^s_x$ denotes the usual Sobolev space in $x$-variables.
    \end{lemma}
    For reader's convenience, we gather the proof of Lemma \ref{product_estimate1} in Appendix \ref{proof_product_estimate}, and omit the proof of Lemma \ref{product_estimate2}, as it follows exactly in the same way.

    \subsection{Stochastic framework}
    The driving noise $W$ is a cylindrical Wiener process in an auxiliary separable Hilbert space $\mathcal{H}$. Formally, we have the following expansion 
    \[W(t)=\sum_{i=1}^{\infty}e_iB_i(t),\]
    where $\{B_i\}_{i=1}^{\infty}$ is a sequence of mutually independent standard Brownian motions and $\{e_i\}_{i=1}^{\infty}$ is an orthonormal basis of $\mathcal{H}$. 

    Let us briefly recall the construction of stochastic integrals in function spaces. We mainly adopt the framework established in \cite{K1999} and \cite{MR2001}, which is equivalent to the classical method introduced in \cite{DZ1992}, but more suitable for our analysis. For any $s\ge 1$ and $\sigma,\gamma>0$, we introduce the $\mathcal{H}$-valued weighted conormal Sobolev space
    \begin{align}\label{pre2}
    \mathbb{X}^{s}_{\gamma}:=\left\{\F: \R_+^d\to \mathcal{H}\bigg|\|\F\|^2_{\X^s_\gamma}:=\sum_{l=0}^1\sum_{k+|j|+l\le s}\|e^{\gamma y^2} Z^j\partial_x^k \partial_y^l\F\|^2_{L^2(\R^{d}_+;\mathcal{H})}<\infty\right\}\end{align}
    and the $\mathcal{H}$-valued analytic-Sobolev space
    \[\mathbb{X}^s_{\gamma,\sigma}:=\{\F: \R_+^d\to \mathcal{H} |\F_{\sigma}\in \mathbb{X}^s_{\gamma}\}\]
    endowed with the norm $\|\F\|_{\mathbb{X}^s_{\gamma,\sigma}}:=\|\F_{\sigma}\|_{\mathbb{X}^s_{\gamma}}$. Notice that for $\F\in \mathbb{X}^{s}_{\gamma,\sigma}$, 
    \begin{align*}
        \mathcal{L}_{\F}:\mathcal{H}&\to X^{s}_{\gamma,\sigma},\\
        h&\mapsto \langle \F, h\rangle_{\mathcal{H}}
    \end{align*}
    is a Hilbert--Schmidt operator. In fact, from the Parseval identity, we have 
    \begin{align}
        \|\mathcal{L}_{\F}\|_{HS}^2&:=\sum_{i=1}^{\infty}\|\langle \F,e_i\rangle_{\mathcal{H}}\|^2_{X^s_{\gamma,\sigma}}=\sum_{l=0}^{1}\sum_{|k|+j+l\le s}\sum_{i=1}^{\infty}\|\langle e^{\gamma y^2} Z^j\partial_x^k \partial_y^l\F_{\sigma},e_i\rangle_{\mathcal{H}}\|^2_{L^2}=\|\F\|^2_{\mathbb{X}^s_{\gamma,\sigma}}.\label{pre3}
    \end{align}
    Therefore, for any predictable process $\F\in L^2(\Omega;L^2_{\loc}(\R_+;\mathbb{X}^{s}_{\gamma,\sigma}))$, by the classical theory for Hilbert--Schmidt operator-valued integrands, cf. \cite{DZ1992}, one could define the It\^o integral
    \begin{align}\label{pre4}I(t):=\int_0^t\F(r)\mathrm{d}W=\sum_{i=1}^{\infty}\int_0^t \F_i(r)\mathrm{d}B_i,\qquad\F_i:=\langle \F,e_i\rangle_{\mathcal{H}},\end{align}
    which is an $X^{s}_{\gamma,\sigma}$-valued square-integrable martingale satisfying the It\^o isometry 
    \begin{align}\label{pre5}\mathbb{E}\left\|\int_0^t \F(r)\mathrm{d}W\right\|_{X^{s}_{\gamma,\sigma}}^2=\mathbb{E}\int_0^t\|\F(r)\|^2_{\mathbb{X}^{s}_{\gamma,\sigma}}\mathrm{d}r.\end{align}
    Through an appropriate localization procedure, this stochastic integral can be extended to all predictable processes belonging almost surely to the space $L^2_{\loc}(\R_+;\mathbb{X}^{s}_{\gamma,\sigma})$. Let us recall the Burkholder--Davis--Gundy inequality satisfied by this stochastic integral:
    \begin{align}\label{pre6}
    \mathbb{E}\sup_{t\in[0,\tau]}\left|\int_0^t\left\langle u,\F(r)\mathrm{d}W\right\rangle_{X^{s}_{\gamma,\sigma}}\right|\le C\mathbb{E}\left(\int_{0}^{\tau}\|u\|^2_{X^s_{\gamma,\sigma}}\|\F(r)\|^2_{\mathbb{X}^{s}_{\gamma,\sigma}}\mathrm{d}r\right)^{\frac{1}{2}}
    \end{align}
    for any stopping time $\tau\ge 0$.

    For any predictable process $\overline{\F}\in L^2(\Omega; L^2_{\loc}(\R_+; \mathbb{H}_{x,\sigma}^{s}) )$ with
    \begin{align*}
    \mathbb{H}_{x,\sigma}^{s}:=\left\{\overline{\F}: \R^{d-1}\to \mathcal{H}\bigg|\|\overline{\F}\|_{\mathbb{H}_{x,\sigma}
    ^{s}}^2:=\sum_{|k|\le s}\| \partial_x^k\overline{\F}_{\sigma}\|^2_{L^2(\R^{d-1};\mathcal{H})}<\infty\right\},\end{align*}
    we can similarly define the It\^o integral as in \eqref{pre4}. The resulting stochastic integral is a square-integrable martingale taking values in $H^s_{x,\sigma}$ and satisfies \eqref{pre5} and \eqref{pre6}.
    
    \subsection{Structure of the multiplicative noise and related nonlinear estimates}\label{noises}  
    Before proceeding into the assumptions on the nonlinearity $\F$, let us notice that even in the deterministic setting, in order to have a semi-linear PDE wellposed in the analytic class, one has to impose analytic regularity on the force $\F$, cf. \cite{BRS2003} and references therein. Therefore, we consider the following analytic nonlinearity:
    \begin{align}
        \label{struc0_1}\F(x,y,u):=\sum_{n\in\mathbb{N}^{d-1}}a_n(x,y)u^n,
    \end{align}
    where $u^n:=\Pi_{i=1}^{d-1}u_i^{n_i}$ with $u=(u_1,\ldots, u_{d-1})$ and $n=(n_1,\ldots, n_{d-1})\in \mathbb{N}^{d-1}$, and the coefficients $a_n(x,y): \R^d_+\to \mathcal{H}^{d-1}$ are analytic functions with possibly non-vanishing state at far field $\{y=\infty\}$, which satisfy
    \begin{align}\label{struc0_2}
        a_n(x,y)=a^0_n(x,y)+\bar{a}_n(x),\qquad a^0_n\in\mathbb{X}^s_{\gamma_0,2\sigma_0},\qquad\bar{a}_n\in \mathbb{H}^s_{x,2\sigma_0}.
    \end{align}
    To ensure the convergence in \eqref{struc0_1}, we assume the following summability condition 
    \begin{align}\label{struc0_3}
        \sum_{n\in\mathbb{N}^{d-1}}(\|a^0_n\|_{\mathbb{X}^s_{\gamma_0,2\sigma_0}}+\|\bar{a}_n\|_{\mathbb{H}^s_{x,2\sigma_0}})\Pi_{i=1}^{d-1}|
        u_i|^{n_i}<\infty,\qquad \forall u\in\R^{d-1}.
    \end{align}
    Moreover, we require the condition 
    \begin{align}
        \label{struc0_3_1}
        a_{\mathbf{0}}(x,0)\equiv 0,
    \end{align}
    where $\mathbf{0}$ denotes the zero mode in $\N^{d-1}$. This, together with the boundary condition of $u$, implies $\F(x,y,u)\ddd W|_{y=0}=0$.

    On the other hand, equation \eqref{Intro1}$_1$ might be intuitively viewed as a hyperbolic equation in the tangential direction, while for stochastic hypobolic equations, the multiplicative noise is usually independent of the derivatives of the unknown, cf. \cite{GV2014,CDK2012,JN2008}; otherwise, there would be a loss of derivatives in the energy estimate. However, working within the tangentially analytic class enables a gain of a half-order tangential derivative. Therefore, we are motivated to consider following generalized linear multiplicative noise:
    \begin{align}\label{struc0_4}
        \F(x,y,u):=\mathbb{G}(x,y)|\nabla_x|^{a}u,
    \end{align}
    where $a\in [0,\frac{1}{2}]$ and
    \begin{align}\label{struc0_5}
        \mathbb{G}(x,y)=\mathbb{G}_0(x,y)+\bar{\mathbb{G}}(x),\qquad\mathbb{G}_0\in \mathbb{X}^s_{\gamma_0,2\sigma_0},\qquad\bar{\mathbb{G}}\in \mathbb{H}^s_{x,2\sigma_0}.
    \end{align}
    In summary, the precise assumption of the nonlinearity $\F$ is stated as follows.
    \begin{assumption}\label{assumption} The nonlinearity $\F$ is given by
    \[\F(x,y,u):=\F_1(x,y,u)+\mathbb{G}(x,y)|\nabla_x|^{a}u,\]
    with $\F_1$ satisfying \eqref{struc0_1}--\eqref{struc0_3_1}, $\mathbb{G}$ satisfying \eqref{struc0_5}, and $a\in [0,\frac{1}{2}]$.
    \end{assumption}
    The following estimates are used to control the stochastic terms in the energy estimate and their proofs are gathered in Appendix \ref{proof_nonlinear_estimate}.
    \begin{prop}\label{Proposition_struc1} 
    Let the nonlinearity $\F$ satisfy Assumption \ref{assumption}. For any $s\ge 4$, $\gamma\in (0,\gamma_0]$, and $\sigma\in(0,\sigma_0]$, the following estimates hold.
        \begin{enumerate}
            \item There exists a positive non-decreasing function $\mathcal{K}_s$ depending only on $s$ such that
        \begin{equation}\label{estimate_noise1}
            \|\F(u)-\F(\tilde{u})\|_{\mathbb{X}^s_{\gamma,\sigma}}\le \mathcal{K}_s(\|v\|_{ X^s_{\gamma,\sigma}},\|\tilde{v}\|_{ X^s_{\gamma,\sigma}},\|U\|_{H^s_{x,\sigma}})\|\langle\nabla_x\rangle^{a}(u-\tilde{u})\|_{X^s_{\gamma,\sigma}}
        \end{equation}
        for any functions $u:=v+U$ and $\tilde{u}:=\tilde{v}+U$ satisfying $\langle\nabla_x\rangle^{a}v,\langle\nabla_x\rangle^{a}\tilde{v}\in X^s_{\gamma,\sigma}$, and $U\in H^{s}_{x,\sigma}$.
        \item There exists a positive non-decreasing function $\kappa_s$ depending only on $s$ such that
        \begin{align}\label{estimate_noise2}
            \|\F(u)-\overline{\F}(U)\|_{\mathbb{X}^{s}_{\gamma,\sigma}}&\le \kappa_s(\|v\|_{X^s_{\gamma,\sigma}},\|U\|_{H^s_{x,\sigma}})\left(\|\langle\nabla_x\rangle^{a}v\|_{X^s_{\gamma,\sigma}}+1\right)
        \end{align}
        for any function $u:= v+U$ satisfying $\langle\nabla_x\rangle^{a}v\in X^s_{\gamma,\sigma}$ and $U\in H^{s}_{x,\sigma}$, where $\overline{\F}(U)=\overline{\F}(x,U):=\lim_{y\to\infty}\F(x,y,u)$.
        
        \item There exists a positive non-decreasing function $\kappa_s$ depending only on $s$ such that
        \begin{align}\label{estimate_noise3}
            \|\nabla_x(\F(u)-\overline{\F}(U))\|_{\mathbb{X}^{s}_{\gamma,\sigma}}&\le \kappa_s(\|v\|_{X^s_{\gamma,\sigma}},\|U\|_{H^{s+1}_{x,\sigma}})\left(\|\langle\nabla_x\rangle^{a+1}v\|_{X^s_{\gamma,\sigma}}+1\right)
        \end{align}
        for any function $u:= v+U$ satisfying $\langle\nabla_x\rangle^{a+1}v\in X^s_{\gamma,\sigma}$ and $U\in H^{s+1}_{x,\sigma}$.
        \end{enumerate}
    \end{prop}

    \section{Main results}\label{mainresults}
    In this section, we make precise statements of our main results. First, we establish local wellposedness for the stochastic Prandtl equation \eqref{Intro1} with generic nonlinear multiplicative noises, in a tangentially analytic and normally Sobolev-type space. Next, we investigate the probabilistic global existence of solutions for the stochastic Prandtl equation driven by tangential random diffusion:
    \begin{itemize}
    \item In the two-dimensional case, we show that the linearized equation around any shear flow admits a Gevrey-4 solution.
    \item In the three-dimensional case, there is a global Gevrey-2 solution with its Gevrey radius growing linearly in time.
    \end{itemize}
    
    \subsection{Main result I: local wellposedness}\label{mainresult1}
    Let us present the notion of local pathwise solutions of the stochastic Prandtl equation~\eqref{Intro1}. Throughout the rest of the paper, we shall always assume that $(U, P)$ is a given local pathwise solution of the stochastic Bernoulli law \eqref{Intro2} with analytic regularity. The existence of such solutions can be established via, for example, first taking $P$ as a stochastic process valued in the analytic class $H^s_{x,\sigma}$ and then solving the resulting stochastic Burgers equation, see Appendix \ref{bernoulli} for details. 
    \begin{defn}\label{DEFLOCAL}
        Given a stochastic basis $(\Omega,\mathscr{F},\{\mathscr{F}_t\}_{t\ge0},\PP,W)$, a pair $(u,\tau)$ is said to be a local pathwise solution of the stochastic Prandtl equation \eqref{Intro1}, if 
        \begin{enumerate}
            \item $\tau$ is a strictly positive stopping time; 
            \item $u$ is a progressively measurable process valued in 
            \begin{align}
                \label{DEFLOCAL0}E:=\{u\in L_y^\infty H_x^1|\partial_y u\in L^2(\R^{d}_+),u|_{y=0}=0\},
            \end{align}
            where $L_y^\infty H_x^1:=L^\infty([0, +\infty);H^1(\mathbb R^{d-1}))$;
            \item for any test function $\varphi\in C_0^{\infty}(\R^{d}_+)$, it follows that
            \begin{align}\label{DEFLOCAL1}
                \langle\varphi,u(t\wedge\tau)\rangle&+\int_0^{t\wedge\tau}\langle \varphi,\left(u\cdot\nabla_{x}  u-\partial_y^{-1}(\nabla_x \cdot u)\partial_y u-\partial_y^2 u+\nabla_x P\right)\rangle\mathrm{d}r\notag\\&\qquad\qquad\qquad\qquad\qquad=\langle\varphi,u(0)\rangle+\sum_{i=1}^{\infty}\int_0^{t\wedge\tau}\langle \varphi, \F_i(u)\rangle \mathrm{d}B_i
            \end{align}
            for any $t>0$.
        \end{enumerate}
        A local pathwise solution $(u,\tau)$ is said to be unique, if for any solution $(\tilde{u},\tilde{\tau})$, 
        \[\PP \{u(t)=\tilde{u}(t),\forall t\le \tau\wedge\tilde{\tau}\}=1.\]
    \end{defn}
    To present our first result, we introduce the functions \begin{align}\label{MR1}\gamma(t):=\frac{\gamma_0}{\langle t\rangle^{\delta}},\qquad \sigma(t):=\sigma_0-\lambda t, \qquad \psi(t,y):=\erf\left(\frac{y}{2\sqrt{t+(8\gamma_0)^{-1}}}\right),    
    \end{align}
    where $\delta,\lambda>0$ are parameters, $\langle t\rangle:=1+t$, and $\erf(y):=\frac{2}{\sqrt{\pi}}\int_0^ye^{-\tilde{y}^2}\mathrm{d} \tilde{y}.$ It is easy to see that $\psi(t,y)$ solves the following heat equation:
   \[\begin{cases}
    \partial_t\psi-\partial_y^2\psi=0,\\
    \psi|_{t=0}=\erf(\sqrt{2\gamma_0}y),\qquad
    \lim\limits_{y\to\infty}\psi(t,y)=1,\qquad \psi|_{y=0}=0.
    \end{cases}\]
    
    \begin{theorem} \label{MT1}{\rm (Local wellposedness of the stochastic Prandtl equation).} Let $s\ge 4$ and $\gamma_0,\sigma_0>0$. Suppose that the nonlinearity $\F$ satisfies Assumption \ref{assumption} and $(U, P)$ is a local pathwise solution of the stochastic Bernoulli law \eqref{Intro2} with spatial regularity $H^{s+2}_{x,\sigma_0}$. Then, for any given initial data
    \[u_0=w_0+\erf(\sqrt{2\gamma_0}y) U_0,\]
    with $w_0$ being an $\mathscr{F}_0$-measurable random variable valued in $X^{s}_{\gamma_0,2\sigma_0}$, there exist $\mathscr{F}_0$-measurable random parameters $\delta,\lambda>0$ and a unique local pathwise solution $(u,\tau)$ of the stochastic Prandtl equation \eqref{Intro1} such that
    \[u-\psi U\in L^{\infty}\left(0,\tau; X^s_{\gamma(\cdot),\sigma(\cdot)}\right),\qquad \PP\mbox{-a.s.}\]
    with $\gamma(\cdot),\sigma(\cdot)$ given in \eqref{MR1}.
    \end{theorem}

    \begin{remark}\label{remarkloss} \rm There is an instant loss of analytic radius in the above statement: the initial radius $2\sigma_0$ immediately decays as $\sigma(t):=\sigma_0-\lambda t$. This phenomenon of radius depletion also occurs in deterministic settings, cf. \cite{IV2016}. In our analysis, it is mainly used to ensure the higher-order tangential estimate, cf. Proposition \ref{AP_proposition2}, which plays an important role in the convergence of approximate solutions. However, it is not necessary to lose half of the initial radius. Indeed, from the proof, we will see that Theorem \ref{MT1} remains valid, if the initial radius $2\sigma_0$ is replaced by any fixed $\tilde\sigma_0>\sigma_0$.

    \end{remark}
    
    \subsection{Main result \texorpdfstring{I\!I}{II}: global existence}\label{mainresult2}
    Next, we investigate the stochastic Prandtl equation \eqref{Intro3} with a tangential random diffusion and $a=\frac{1}{2}$. The notion of global solutions with high probability is stated as follows.
    \begin{defn}\label{DEFGLOBAL}
        Given a stochastic basis $(\Omega,\mathscr{F},\{\mathscr{F}_t\}_{t\ge0},\PP,W)$ and $\epsilon>0$, a triplet $(u,\tau,\Omega_{\epsilon})$ is said to be a global solution of the stochastic Prandtl equation \eqref{Intro3} with probability $1-\epsilon$, if $\PP (\Omega_{\epsilon})\ge 1-\epsilon$ and
        \begin{enumerate}
            \item $\tau$ is a stopping time satisfying $\tau=\infty$ on $\Omega_{\epsilon}$;
            \item $u$ is progressively measurable, and takes values in the space $E$ defined in \eqref{DEFLOCAL0};
            \item for any test function $\varphi\in C^{\infty}_0(\R_+^{d})$, it follows that
            \begin{align*}
                \langle\varphi,u(t\wedge \tau)\rangle&+\int_0^{t\wedge \tau}\langle \varphi,\left(u\cdot\nabla_{x}  u-\partial_y^{-1}(\nabla_x\cdot u)\partial_y u-\partial_y^2 u-\beta \UU\vec{e}\right)\rangle\mathrm{d}r\\&\qquad\qquad\qquad\qquad\qquad=\langle\varphi,u(0)\rangle+\int_0^{t\wedge \tau}\langle \varphi, (\alpha_1+\alpha|\nabla_x|^{\frac{1}{2}})u\rangle \mathrm{d}B_{r}
                \end{align*}
                for any $t\ge 0$.
        \end{enumerate}
    \end{defn}
    To formulate the result, we define the weighted anisotropic Sobolev spaces
    \[H^{s,0}_\gamma:=\left\{u:\R^d_+\to \R\Big|\|u\|^2_{H^{s,0}_\gamma}:=\|\langle y\rangle^\gamma\langle\nabla_x\rangle^{s}u\|^2_{L^2}<\infty\right\},\qquad \langle y\rangle:=1+y,\]
    and
    \begin{align}\label{MR3_2}
    \XX^{s,a}_{\gamma}:=\left\{u:\R^d_+\to \R\Big|\|u\|^2_{\XX^{s,a}_{\gamma}}:=\|u\|_{H^{s,0}_{\gamma}}^2+\|\partial_y u\|^2_{H^{s-1/a,0}_{\gamma}}<\infty\right\},\end{align}
    and the Gevrey--Sobolev space
    \[\XX^{s,a}_{\gamma,\sigma}:=\left\{u:\R^d_+\to \R\Big|\|u\|_{\XX^{s,a}_{\gamma,\sigma}}:=\left\|e^{\sigma |\nabla_x|^{1/a}}u\right\|_{\XX^{s,a}_{\gamma}}<\infty\right\}.\]
    We introduce the corrector
    \begin{align}\label{MR3_1}\Psi(y):=1-e^{-\sqrt{\beta}y},\end{align}
    and the following characterization of the Gevrey radius:
    \begin{align}\label{MR4}
        \sigma(t):=\sigma_0+\lambda t-\alpha B_t.
    \end{align}
    Our next result is presented below.
    \begin{theorem}\label{MT2}{\rm (Probabilistic global existence of Gevrey-2 solutions).} Suppose that $a=\frac{1}{2}$ in \eqref{Intro3}. Let $s> \frac{3}{2}, \gamma>1$, and $\UU_0,\sigma_0>0$. Then, for any $\epsilon\in (0,1)$ and any initial data $u_0$ satisfying 
    \begin{align}
        \label{MR5_1}
        u_0-\Psi\UU_0\vec{e}\in L^\infty\left(\Omega;{\XX^{s,2}_{\gamma,\sigma_0}}\right),
    \end{align}
     there exist parameters $\lambda,\alpha,\alpha_1,\beta$ such that the stochastic Prandtl equation \eqref{Intro3} admits a global solution $(u,\tau,\Omega_{\epsilon})$ with probability $1-\epsilon$ in the sense of Definition \ref{DEFGLOBAL}. Moreover, 
     \[u-\Psi\UU(t)\vec{e}\in L^{\infty}\left(0,\tau;\XX^{s,2}_{\gamma,\sigma(\cdot)}\right),\qquad \PP\mbox{-a.s.,}\]
     where the outflow $\UU(t)$ is given in \eqref{1.4}, and the Gevrey radius $\sigma(\cdot)$ given by \eqref{MR4} satisfies
    \begin{align}\label{MR5}
        \sigma(t)\ge \frac{\sigma_0+\lambda t}{2},\qquad \forall t\ge0
    \end{align}
    on $\Omega_\epsilon$.
    \end{theorem}
    \begin{remark}\rm
    The lower bound estimate \eqref{MR5} is not sharp near $t=0$, as there is an instant loss of the initial radius. This is due to the growing mechanism of the Brownian path: with high probability, $B_t\le c_1+c_2t$ for some $c_1,c_2>0$. Similarly, it is not necessary to lose half of the initial radius, and one can replace $\frac{\sigma_0}{2}$ with $c\sigma_0$ for any $c\in(0,1)$.
    \end{remark}
    Finally, we study the linearized stochastic Prandtl equation \eqref{Intro3-1} driven by tangential random diffusion with Gevrey-$4$ initial data.
    \begin{theorem}
        {\rm (Probabilistic global existence of Gevrey-4 solutions in 2D).}\label{MT3} Let $s\ge \frac{1}{4}, \gamma>1$, and $\sigma_0>0$. Suppose that the shear flow $u^s(t,y)$ satisfies
        \[\sup_{t\ge0} \left(\|\langle y\rangle^\gamma\partial_yu^s\|_{L^2_y}+\|\langle y\rangle^\gamma\partial^2_yu^s\|_{L^2_y}\right)<\infty.\]
        Then, for any $\epsilon\in (0,1)$, there exist parameters $\lambda$ and $\alpha$ such that for any 
        \begin{align}\label{MR6}
            u_0\in\XX^{s+3/4,4}_{\gamma,\sigma_0}\qquad \PP\mbox{-a.s.},
        \end{align}
        the linearized stochastic Prandtl equation \eqref{Intro3-1} admits a global Gevrey-4 solution $(u,\tau,\Omega_{\epsilon})$ with probability $1-\epsilon$. Moreover,
        \begin{align}
            \label{MR7}
            u\in L^{\infty}\left(0,\tau;\XX^{s,4}_{\gamma,\sigma(\cdot)}\right),\qquad\PP\mbox{-a.s.,}
        \end{align}
    where the Gevrey radius $\sigma(\cdot)$ given by \eqref{MR4} satisfies \eqref{MR5} on $\Omega_\epsilon$.
    \end{theorem}
    \begin{remark}\begin{itemize}
        \item As indicated in \eqref{MR6} and \eqref{MR7}, there is an instant loss of tangential regularity. This loss is caused by the estimates of the good unknown function, see \eqref{2DGWP17}, and is consistent with the deterministic case, see Remark 1.1 in \cite{WWZ2021}.
        \item We choose to study the 2D case for the sake of simplicity. By combining the strategy in \cite{LiMYang22}, similar result can be obtained for the 3D case. 
    \end{itemize}
    \end{remark}

    \section{A priori estimates}\label{AP}
    In this section, we derive a priori estimates of the homogenized and truncated stochastic Prandtl equation \eqref{AP1}. The bounds will be frequently used throughout the paper. 
    
    \subsection{Homogenized and truncated stochastic Prandtl equation}
    To homogenize the far-field state of $u$ in \eqref{Intro1}, we introduce 
    \[u^s(t,x,y):=\psi(t,y)U(t,x),\qquad w(t,x,y):=u(t,x,y)-u^s(t,x,y),\]
    where $U(t,x):=\lim_{y\to\infty}u(t,x,y)$, and $\psi(t,y)$ is given in \eqref{MR1}. After truncating the nonlinearities in the resulting homogenized equation for $w$, we obtain
    \begin{align}\label{AP1}
    \begin{cases}
        \mathrm{d}w-\partial_y^2w\mathrm{d}t+\chi^2(\mathcal{E}_w(t))\left[\mathcal{B}_1(w,w)+\mathcal{B}_1(w,u^s)+\mathcal{B}_1(u^s,w)+\mathcal{B}_2(w+u^s,w+u^s)\right]\mathrm{d} t\\\qquad\qquad= \left[\mathcal{B}_1(U-u^s,u^s)-(1-\psi)\nabla_x P\right]\mathrm{d} t+\chi^2(\mathcal{E}_w(t))[\F(w+u^s)-\psi \overline{\F}(U)]\mathrm{d}W,\\
        w|_{y=0}=\lim_{y\to\infty} w=0,\\
        w(0)=w_0:=u_0-U_0\erf\left(\sqrt{2\gamma_0}y\right),
    \end{cases}
    \end{align}
    where 
    \begin{align}
        \label{bilinear_term}
        \mathcal{B}_1(u,\tilde{u}):=(u\cdot\nabla_x) \tilde{u},\qquad \mathcal{B}_2(u,\tilde{u}):=-\partial_y^{-1}(\nabla_x \cdot u)\partial_y \tilde{u},
    \end{align}
    and $\chi:\R_+\to [0,1]$ is a smooth cut-off function satisfying
    \begin{align}\label{chi}
        \chi(z):=
    \begin{cases}
    1\qquad z\le 2\MM, \\
    0\qquad z\ge 3\MM 
    \end{cases}
    \end{align}
    for some constant $\MM >0$, and
    \begin{align}\label{parabolic_energy}
        \mathcal{E}_w(t):=\left(\sup_{0\le r\le t}\|w_{\sigma}(r)\|^2_{X^s_{\gamma}}+\int_0^{t}\left(\||\nabla_x|^{\frac{1}{2}} w_{\sigma}(r)\|^2_{X^s_{\gamma}}+\|\partial_y w_{\sigma}(r)\|^2_{X^s_{\gamma}}\right)\mathrm{d}r\right)^{\frac{1}{2}}
    \end{align}
    denotes the parabolic energy of $w_{\sigma}$ with $\sigma=\sigma(t)$, $\gamma=\gamma(t)$ given in \eqref{MR1}. Here and throughout the paper, to simplify the notation, we usually omit the time dependence of the parameters $\gamma(t)$ and $\sigma(t)$. This will not cause any ambiguity. 
    
    \subsection{A priori estimates of the homogenized and truncated equation}\label{AP_for_solution}
    First, we assume the following boundedness of the outflow:
    \begin{align}\label{APassumption}
        \|U_0\|_{H^{s+2}_{x,\sigma_0}}\le \NN ,\qquad \|\nabla_x P|_{t=0}\|_{H^{s+1}_{x,\sigma_0}}\le \NN,
    \end{align}
    where $\NN >0$ is a fixed constant. This restriction will be removed in Section \ref{proof_LWP}. We introduce the positive stopping time 
    \[\tau_{\rm out}:=\inf\{t\ge0| \langle t\rangle^{\delta}\|U_{\sigma_0}(t)\|_{H^{s+2}_{x}}\ge 2\NN \}\wedge \inf\{t\ge0| \|\nabla_xP_{\sigma_0}(t)\|_{H^{s+1}_{x}}\ge 2\NN \}.\]
    The main result of this subsection is the following a priori estimate on $[0,T_*]$, where
    \begin{align}\label{AP1_1}
    T_*:=\frac{1}{16\gamma_0}\wedge\frac{\sigma_0}{2\lambda}\wedge1\wedge\tau_{\rm out}.
    \end{align}
    \begin{prop}\label{AP_proposition1}
        Let $s\ge 4$, $\gamma_0,\sigma_0>0$, and let assumption \eqref{APassumption} be satisfied. Then, there exist positive constants $\lambda,\delta$ such that for any smooth solution $w$ of the equation \eqref{AP1} and any stopping time $\tau\le T_*$,
        \begin{align}\label{AP53}
        \mathbb{E}\sup_{0\le t\le \tau}\|w_{\sigma}\|^2_{X^s_{\gamma}}+\mathbb{E}\int_0^{\tau}&\left(\||\nabla_x|^{\frac{1}{2}} w_{\sigma}\|^2_{X^s_{\gamma}}+\|\partial_y w_{\sigma}\|^2_{X^s_{\gamma}}\right)\mathrm{d}t\lesssim_{s,\gamma_0,\MM ,\NN } \mathbb{E}\|w_{0}\|^2_{X^s_{\gamma_0,\sigma_0}}+1,
    \end{align}
    where the parameters $\sigma=\sigma(t)$ and $\gamma=\gamma(t)$ are given in \eqref{MR1}.
    \end{prop}
    The proof of Proposition \ref{AP_proposition1} is based on the following lemmas.
    \begin{lemma}\label{APlemma1}{\rm(Estimate for the profile $u^s$).} Under the assumption of Proposition \ref{AP_proposition1}, we have
    \begin{align}
    \label{AP8}
        \|(u^s-U)_{\sigma}\|_{X^s_{\gamma}}&\lesssim_{s,\gamma_0}\|U_{\sigma}\|_{H^s_{x}},\\
        \label{AP9}
        \|\langle\nabla_x\rangle^2(u^s-U)_{\sigma}\|_{X^s_{\gamma}}&\lesssim_{s,\gamma_0}\|U_{\sigma}\|_{H^{s+2}_{x}},\\
        \label{AP10}
        \|\partial_y\langle\nabla_x\rangle u^s_{\sigma}\|_{X^s_{\gamma}}&\lesssim_{s,\gamma_0} \|U_{\sigma}\|_{H^{s+1}_{x}}
    \end{align}
    for $t\le T_*$.
    \end{lemma}
    \begin{proof} Notice that
    \begin{align}\label{AP4}
    \|(u^s-U)_{\sigma}\|^2_{X^s_{\gamma}}=\sum_{l=0}^1 \sum_{|k|+j+l\le s}\|e^{\gamma y^2}Z^j\partial_y^l(\psi-1)\|^2_{L^2_y}\|\partial_x^k U_{\sigma}\|^2_{L^2_x}.
    \end{align}
    Therefore, it suffices to bound the weighted norms of $\psi-1$. By the definition of $\psi$ in \eqref{MR1},
    \[\psi-1=-\frac{2}{\sqrt{\pi}}\int_{\frac{y}{2\sqrt{t+(8\gamma_0)^{-1}}}}^{\infty}e^{-\tilde{y}^2}\mathrm{d}\tilde{y}.\]
   Applying the following tail estimate of the Gaussian integral
    \begin{align}\label{GaussianTail}
     \int_a^{\infty}e^{-y^2}\mathrm{d}y\le Ce^{-a^2},\qquad\forall a\ge 0,
    \end{align}
    where $C>0$ is a universal constant, we get 
    \begin{align}
        \label{AP5}\|e^{\gamma y^2}(\psi-1)\|_{L^2_y}\lesssim \left(\int_0^{\infty}e^{\frac{2\gamma_0 y^2}{\langle t\rangle^{\delta}}-\frac{y^2}{2t+(4\gamma_0)^{-1}}}\mathrm{d}y\right)^{\frac{1}{2}}\lesssim \left(\int_0^{\infty}e^{-\frac{2\gamma_0 y^2}{3}}\mathrm{d}y\right)^{\frac{1}{2}}<\infty
    \end{align}
    for any $t\le T_*$. As for $j+l\ge1$, by applying the Fa\`a di Bruno formula, we derive
    \begin{align*}
        |e^{\gamma y^2} Z^j\partial_y^l (\psi-1)|&=|e^{\gamma y^2} y^j\partial_y^{l+j} \psi|\\&\lesssim e^{\gamma_0 y^2}e^{-\frac{y^2}{4t+(2\gamma_0)^{-1}}}P_{2j+l-1}(y) \lesssim e^{-\frac{\gamma_0y^2}{3}}P_{2j+l-1}(y),\qquad t\le T_*, 
    \end{align*}
    where $P_{2j+l-1}(y)$ stands for a polynomial of order $2j+l-1$ with non-negative coefficients. This, together with \eqref{GaussianTail}, yields
    \begin{align}\label{AP6}
        \|e^{\gamma y^2}Z^j\partial_y^l(\psi-1)\|_{L^2_y}\lesssim \left(\int_0^{\infty}e^{-\frac{2\gamma_0y^2}{3}}P^2_{2j+l-1}(y)\mathrm{d}y\right)^{\frac{1}{2}}<\infty,\qquad t\le T_*.
    \end{align}
    Combining the estimates \eqref{AP4}--\eqref{AP6}, we obtain \eqref{AP8}. The remaining estimates \eqref{AP9} and \eqref{AP10} can be verified in a similar way, and their proofs are thus omitted. 
    \end{proof}
    \begin{lemma}\label{APlemma2}
        {\rm(Estimate for the convection terms).} Under the assumption of Proposition \ref{AP_proposition1}, we have
        \begin{align}
        \label{AP11}
        \langle w_{\sigma},\mathcal{B}_1(w,w)_{\sigma}\rangle_{X^s_{\gamma}}&\lesssim_s\|\langle\nabla_x\rangle^{\frac{1}{2}} w_{\sigma}\|^2_{X^s_{\gamma}}\|w_{\sigma}\|_{X^s_{\gamma}},\\
        \label{AP12}
        \langle w_{\sigma},\mathcal{B}_1(u^s,w)_{\sigma}\rangle_{X^s_{\gamma}}&\lesssim_{s,\gamma_0}\|\langle\nabla_x\rangle^{\frac{1}{2}} w_{\sigma}\|^2_{X^s_{\gamma}}\|U_{\sigma}\|_{H^{s+1}_{x}},\\\label{AP13}
        \langle w_{\sigma},\mathcal{B}_1(w,u^s)_{\sigma}\rangle_{X^s_{\gamma}}&\lesssim_{s,\gamma_0}\| w_{\sigma}\|^2_{X^s_{\gamma}}\|U_{\sigma}\|_{H^{s+1}_{x}},
        \end{align}
        and
        \begin{align}
        \label{AP13_1}
        \langle w_{\sigma},\mathcal{B}_2(w+u^s,w+u^s)_{\sigma}\rangle_{X^s_{\gamma}}&\lesssim_{s,\gamma_0}C_{\delta}\|\langle\nabla_x\rangle^{\frac{1}{2}}w_{\sigma}\|_{X^s_{\gamma}}\|U_{\sigma}\|_{H^{s+1}_{x}}^2\notag\\&\quad+C_{\delta}\|\langle\nabla_x\rangle^{\frac{1}{2}}w_{\sigma}\|^2_{X^s_{\gamma}}(\|w_{\sigma}\|_{X^s_{\gamma}}+\|U_{\sigma}\|_{H^{s+1}_{x}})\notag\\&\quad+C_{\delta}\|\langle\nabla_x\rangle^{\frac{1}{2}}w_{\sigma}\|_{X^s_{\gamma}}\|\partial_y w_{\sigma}\|_{X^s_{\gamma}}(\|w_{\sigma}\|_{X^s_{\gamma}}+\|U_{\sigma}\|_{H^{s+1}_{x}})\notag\\&\quad+\|w_{\sigma}\|_{\widetilde{X}^s_{\gamma}}\|U_{\sigma}\|_{H^{s+1}_{x}}(\|\partial_y w_{\sigma}\|_{X^s_{\gamma}}+\|U_{\sigma}\|_{H^{s+1}_{x}})
        \end{align}
        for $t\le T_*$, where 
    \begin{align}
        \label{AP13_2}\|w_{\sigma}\|^2_{\widetilde{X}^s_{\gamma}}:=\sum_{l=0}^1\sum_{|k|+j+l\le s}\|ye^{\gamma y^2}Z^j\partial^k_x \partial_y^lw_{\sigma}\|^2_{L^2}.
    \end{align}
    \end{lemma}
    \begin{proof} The estimates \eqref{AP11}--\eqref{AP13} follow directly from Lemmas \ref{product_estimate1}, \ref{product_estimate2}, and \ref{APlemma1}. Now, we turn to the estimate \eqref{AP13_1}. To treat the operator $\partial_y^{-1}$ in $\mathcal{B}_2$, we decompose
    \begin{align*}
    \langle w_{\sigma},\mathcal{B}_2(w+u^s,w+u^s)_{\sigma}\rangle_{X^s_{\gamma}}&=\langle w_{\sigma},\mathcal{B}_2(w+u^s-U,w+u^s-U)_{\sigma}\rangle_{X^s_{\gamma}}\\&\quad+\langle w_{\sigma},\mathcal{B}_2(U,w+u^s-U)_{\sigma}\rangle_{X^s_{\gamma}}\\&=: I_1+I_2. 
    \end{align*}
    Notice that by the definition of $\gamma(t)$ in \eqref{MR1}, we have 
    \[C_{\gamma_0,\delta}\le \gamma(t)\le \gamma_0,\qquad t\le T_*.\]
    Then, by applying Lemmas \ref{product_estimate1} and \ref{APlemma1}, we derive
    \begin{align}
        \label{AP14}
        I_1&\lesssim_{s,\gamma_0,\delta}\|\langle\nabla_x\rangle^{\frac{1}{2}} w_{\sigma}\|_{X^s_{\gamma}}\| (w+u^s-U)_{\sigma}\|_{X^s_{\gamma}}\|\partial_y (w+u^s)_{\sigma}\|_{X^s_{\gamma}}\notag\\&\quad+\|\langle\nabla_x\rangle^{\frac{1}{2}} w_{\sigma}\|_{X^s_{\gamma}}\| \langle\nabla_x\rangle^{\frac{1}{2}} (w+u^s-U)_{\sigma}\|_{X^s_{\gamma}}\|(w+u^s-U)_{\sigma}\|_{X^s_{\gamma}}\notag\\&\lesssim_{s,\gamma_0,\delta}\|\langle\nabla_x\rangle^{\frac{1}{2}}w_{\sigma}\|_{X^s_{\gamma}}(\|w_{\sigma}\|_{X^s_{\gamma}}+\|U_{\sigma}\|_{H^{s+1}_{x}})(\|\langle\nabla_x\rangle^{\frac{1}{2}}w_{\sigma}\|_{X^s_{\gamma}}+\|\partial_y w_{\sigma}\|_{X^s_{\gamma}})\notag\\&\quad+\|\langle\nabla_x\rangle^{\frac{1}{2}} w_{\sigma}\|_{X^s_{\gamma}}\|U_{\sigma}\|_{H^{s+1}_x}^2.
    \end{align}
    For $I_2$, since $U$ is independent of $y$, we have
    \begin{align*}
        I_2=\langle w_{\sigma},\mathcal{B}_2(U,w+u^s-U)_{\sigma}\rangle_{X^s_{\gamma}}=-\langle w_{\sigma},\left((y\nabla_x\cdot U)\partial_y (w+u^s)\right)_{\sigma}\rangle_{X^s_{\gamma}}.
    \end{align*}
    To control the linearly increasing term $y\nabla_x\cdot U$, we use the following commutator:
    \[[Z^j,y]u:=Z^j(yu)-yZ^ju=jyZ^{j-1}u,\qquad \forall j\ge 1.\]
    Therefore, we derive
    \begin{align}
        -I_2&=\sum_{l=0}^1\sum_{|k|+j+l\le s}\langle e^{\gamma y^2}Z^j\partial_x^k\partial_y^l w_{\sigma}, e^{\gamma y^2}Z^j\partial_x^k\left((y\nabla_x\cdot U)\partial_y^{1+l} (w+u^s)\right)_{\sigma}\rangle_{L^2}\notag\\&\quad+\sum_{|k|+j\le s-1}\langle e^{\gamma y^2}Z^j\partial_x^k\partial_y w_{\sigma}, e^{\gamma y^2}Z^j\partial_x^k\left((\nabla_x\cdot U)\partial_y (w+u^s)\right)_{\sigma}\rangle_{L^2}\notag\\&=\sum_{l=0}^1\sum_{|k|+j+l\le s}\langle ye^{\gamma y^2}Z^j\partial_x^k\partial_y^l w_{\sigma}, e^{\gamma y^2}Z^j\partial_x^k\partial_y^l\left((\nabla_x\cdot U)\partial_y(w+u^s)\right)_{\sigma}\rangle_{L^2}\notag\\&\quad+\sum_{l=0}^1\sum_{|k|+j+l\le s}j\langle ye^{\gamma y^2}Z^j\partial_x^k\partial_y^l w_{\sigma}, e^{\gamma y^2}Z^{j-1}\partial_x^k\partial_y^l\left((\nabla_x\cdot U)\partial_y(w+u^s)\right)_{\sigma}\rangle_{L^2}\notag\\&\quad+\langle\partial_y w_{\sigma},\left((\nabla_x\cdot U)\partial_y (w+u^s-U)\right)_{\sigma}\rangle_{H^{s-1}_{\gamma,\co}},\notag
    \end{align}
    which, together with Lemmas \ref{product_estimate2} and \ref{APlemma1}, implies
    \begin{align}
        \label{AP17}
        I_2&\lesssim_s \|w_{\sigma}\|_{\tilde{X}^s_{\gamma}}\|\left((\nabla_x\cdot U)\partial_y(w+u^s)\right)_{\sigma}\|_{X^s_{\gamma}}+\|w_{\sigma}\|_{X^{s}_{\gamma}}\|\left((\nabla_x\cdot U)(w+u^s-U)\right)_{\sigma}\|_{X^{s}_{\gamma}}\notag\\ &\lesssim_s \|w_{\sigma}\|_{\tilde{X}^s_{\gamma}}\|\nabla_x\cdot U_{\sigma}\|_{H^s_{x}}\|\partial_y(w+u^s)_{\sigma}\|_{X^s_{\gamma}}+\|w_{\sigma}\|_{X^s_{\gamma}}\|\nabla_x\cdot U_{\sigma}\|_{H^{s}_{x}}\|(w+u^s-U)_{\sigma}\|_{X^{s}_{\gamma}}\notag\\ &\lesssim_{s,\gamma_0} \|U_{\sigma}\|_{H^{s+1}_{x}}^2(\|w_{\sigma}\|_{\tilde{X}^s_{\gamma}}+\|w_{\sigma}\|_{X^s_{\gamma}})+\|U_{\sigma}\|_{H^{s+1}_{x}}(\|w_{\sigma}\|_{\tilde{X}^s_{\gamma}}\|\partial_yw_{\sigma}\|_{X^s_{\gamma}}+\|w_{\sigma}\|_{X^s_{\gamma}}^2).
    \end{align}
    Combining this with \eqref{AP14}, we obtain \eqref{AP13_1}.
    \end{proof}
    \begin{lemma}\label{APlemma3}
        {\rm(Estimate for the force terms).} Under the assumption of Proposition \ref{AP_proposition1}, we have
        \begin{align}\label{AP22}
        \langle w_{\sigma}, \mathcal{B}_1(U-u^s,u^s)_{\sigma}-(1-\psi)\nabla_x P_{\sigma} \rangle_{X^s_{\gamma}}\lesssim_{s,\gamma_0}\|w_{\sigma}\|_{X^s_{\gamma}}(\|U_{\sigma}\|^2_{H^{s+1}_x}+\|\nabla_x P_{\sigma}\|_{H^s_{x}}),
        \end{align}
        and
        \begin{align}\label{AP21}
        \|[\F(w+u^s)-\psi\overline{\F}(U)]_{\sigma}\|_{\mathbb{X}^s_{\gamma}}\lesssim_{\gamma_0}\mathcal{K}_s(\|w_{\sigma}\|_{X^s_{\gamma}},\|U_{\sigma}\|_{H^{s}_{x}})(\|\langle\nabla_x\rangle^{\frac{1}{2}}w_{\sigma}\|_{X^s_{\gamma}}+\|U_{\sigma}\|_{H^{s+1}_{x}})
        \end{align}
        for $t\le T_*$, where $\mathcal{K}_s$ is a positive non-decreasing function depending only on $s$.
    \end{lemma}
    \begin{proof}
        The estimate \eqref{AP22} follows from a direct application of Lemmas \ref{product_estimate1}, \ref{product_estimate2}, and \ref{APlemma1}. For the stochastic term, we apply Proposition \ref{Proposition_struc1} to derive 
        \begin{align}\label{AP19}
        \|[\F(w+u^s)&-\psi\overline{\F}(U)]_{\sigma}\|_{\mathbb{X}^s_{\gamma}}\notag\\&\lesssim_{s,\gamma_0}\kappa_s(\|w_{\sigma}\|_{X^s_{\gamma}},\|U_{\sigma}\|_{H^s_{x}})\left(\|\langle\nabla_x\rangle^{a}w_{\sigma}\|_{X^s_{\gamma}}+\|U_{\sigma}\|_{H^{s+a}_{x}}\right)+\|\overline{\F}_{\sigma}(U)\|_{\mathbb{H}^s_{x}}.
        \end{align}
        On the other hand, proceeding as in the proof of \eqref{estimate_noise2}, we have
        \begin{align*}
        \|\overline{\F}_{\sigma}(U)\|_{\mathbb{H}^s_{x}}\le \mathcal{K}_s(\|U_{\sigma}\|_{H^{s}_{x}})\|U_{\sigma}\|_{H^{s}_{x}}
        \end{align*}
        for some positive non-decreasing function $\mathcal{K}_s$. Combining this with \eqref{AP19}, the desired estimate \eqref{AP21} holds.
    \end{proof}
    Now we are in a position to prove Proposition \ref{AP_proposition1}.
    \begin{proof}[Proof of Proposition \ref{AP_proposition1}]
        By applying the Fourier multiplier $e^{\sigma|\nabla_x|}$ on both sides of \eqref{AP1}, we obtain
    \begin{align}
        \label{AP2}
    &\mathrm{d}w_{\sigma}+(\lambda|\nabla_x|w_{\sigma}-\partial_y^2w_{\sigma})\mathrm{d}t\notag\\&\qquad+\chi^2(\mathcal{E}_w(t))\left[\mathcal{B}_1(w,w)_{\sigma}+\mathcal{B}_1(w,u^s)_{\sigma}+\mathcal{B}_1(u^s,w)_{\sigma}+\mathcal{B}_2(w+u^s,w+u^s)_{\sigma}\right]\mathrm{d} t\notag\\&= \left[\mathcal{B}_1(U-u^s,u^s)_{\sigma}-(1-\psi)\nabla_x P_{\sigma}\right]\mathrm{d} t+\chi^2(\mathcal{E}_w(t))[\F(w+u^s)-\psi \overline{\F}(U)]_{\sigma}\mathrm{d}W.
    \end{align}

    \noindent {\textbf{Step I. A priori estimate in $H^s_{\gamma,\co}$.}} By the It\^o formula, we have 
    \begin{align}
        \label{AP26}
        \mathrm{d}\|w_{\sigma}\|^2_{H^s_{\gamma,\co}}+&2\left(\lambda\||\nabla_x|^{\frac{1}{2}}w_{\sigma}\|^2_{H^s_{\gamma,\co}}+\frac{\gamma_0\delta}{\langle t\rangle^{\delta+1}}\|w_{\sigma}\|^2_{\widetilde{H}^s_{\gamma,\co}}-\langle w_{\sigma},\partial_y^2 w_{\sigma}\rangle_{H^s_{\gamma,\co}}\right)\mathrm{d}t\notag\\&\qquad\qquad+2\chi^2(\mathcal{E}_w(t))\left(\langle w_{\sigma},\mathcal{B}_1(w,w)_{\sigma}+\mathcal{B}_1(w,u^s)_{\sigma}+\mathcal{B}_1(u^s,w)_{\sigma}\rangle_{H^s_{\gamma,\co}}\right)\mathrm{d}t\notag\\&\qquad\qquad+2\chi^2(\mathcal{E}_w(t))\langle w_{\sigma},\mathcal{B}_2(w+u^s,w+u^s)_{\sigma}\rangle_{H^s_{\gamma,\co}}\mathrm{d}t\notag\\&=\left(2\langle w_{\sigma}, \mathcal{B}_1(U-u^s,u^s)_{\sigma}-(1-\psi)\nabla_x P_{\sigma}\rangle_{H^s_{\gamma,\co}}\right)\mathrm{d} t\notag\\&\qquad\qquad +\chi^4(\mathcal{E}_w(t))\|[\F(w+u^s)-\psi\overline{\F}(U)]_{\sigma}\|^2_{\mathbb{H}^{s}_{\gamma,\co}}\mathrm{d} t\notag\\&\qquad\qquad+2\chi^2(\mathcal{E}_w(t))\langle w_{\sigma}, [\F(w+u^s)-\psi\overline{\F}(U)]_{\sigma}\mathrm{d} W\rangle_{H^s_{\gamma,\co}}, 
    \end{align}
    where $\|w_{\sigma}\|^2_{\widetilde{H}^s_{\gamma,\co}}:=\sum_{|k|+j\le s}\|ye^{\gamma y^2}Z^j\partial^k_x w_{\sigma}\|^2_{L^2}.$ Integrating by parts yields
    \begin{align}\label{AP27}
    -\langle w_{\sigma},\partial_y^2 w_{\sigma}\rangle_{H^s_{\gamma,\co}}&=-\sum_{j+|k|\le s}\langle e^{\gamma y^2}Z^j\partial_x^kw_{\sigma},  e^{\gamma y^2}Z^j\partial_x^k\partial_y^2 w_{\sigma}\rangle_{L^2}\notag\\&=\|\partial_y w_{\sigma}\|^2_{H^s_{\gamma,\co}}+4\gamma\sum_{j+|k|\le s}\langle ye^{\gamma y^2}Z^j\partial_x^kw_{\sigma},  e^{\gamma y^2}Z^j\partial_x^k\partial_y w_{\sigma}\rangle_{L^2}\notag\\&\quad+2\sum_{j+|k|\le s}j\langle e^{\gamma y^2}Z^{j-1}\partial_x^k\partial_yw_{\sigma},  e^{\gamma y^2}Z^j\partial_x^k\partial_y w_{\sigma}\rangle_{L^2},
    \end{align}
    where we have
     \begin{align}\label{AP28}
        4\gamma\sum_{j+|k|\le s}\langle ye^{\gamma y^2}Z^j\partial_x^kw_{\sigma},  e^{\gamma y^2}Z^j\partial_x^k\partial_y w_{\sigma}\rangle_{L^2}\le \frac{4\gamma_0}{\langle t\rangle^{\delta}}\|w_{\sigma}\|_{\widetilde{H}^s_{\gamma,\co}}\|\partial_yw_{\sigma}\|_{H^s_{\gamma,\co}},
    \end{align}
    and 
    \begin{align}\label{AP29}
        &2\sum_{j+|k|\le s}j\langle e^{\gamma y^2}Z^{j-1}\partial_x^k\partial_yw_{\sigma},  e^{\gamma y^2}Z^j\partial_x^k\partial_y w_{\sigma}\rangle_{L^2}\notag\\&\qquad\le 2s \sum_{\substack{j+|k|\le s\\ j\ge 1}}\langle e^{\gamma y^2}Z^{j-1}\partial_x^k\partial_yw_{\sigma},  e^{\gamma y^2}Z^j\partial_x^k\partial_y w_{\sigma}\rangle_{L^2}
        \le 2s\|\partial_y w_{\sigma}\|_{H^{s-1}_{\gamma,\co}}\|\partial_yw_{\sigma}\|_{H^s_{\gamma,\co}}.
    \end{align}
    Combining the estimates \eqref{AP27}--\eqref{AP29}, we obtain
    \begin{align*}
        -\langle w_{\sigma},\partial_y^2 w_{\sigma}\rangle_{H^s_{\gamma,\co}}\ge \frac{1}{2}\|\partial_y w_{\sigma}\|^2_{H^s_{\gamma,\co}}-C_{s,\gamma_0}\left(\frac{1}{\langle t\rangle^{2\delta}}\|w_{\sigma}\|^2_{\widetilde{X}^s_{\gamma}}+\|w_{\sigma}\|^2_{X^s_{\gamma}}\right),
    \end{align*}
    where we recall the definition of $\|w_{\sigma}\|_{\widetilde{X}^s_{\gamma}}$ given in \eqref{AP13_2}. This, together with \eqref{AP26}, Lemmas \ref{APlemma1}--\ref{APlemma3}, and the Burkholder--Davis--Gundy inequality \eqref{pre6}, implies
    \begin{align}
        \label{AP31}
        &\mathbb{E}\sup_{t\le \tau}\|w_{\sigma}\|^2_{H^s_{\gamma,\co}}+2\mathbb{E}\int_0^{\tau}\left(\lambda\||\nabla_x|^{\frac{1}{2}} w_{\sigma}\|^2_{H^s_{\gamma,\co}}+\frac{1}{2}\|\partial_y w_{\sigma}\|^2_{H^s_{\gamma,\co}}+\frac{\gamma_0\delta}{\langle t\rangle^{\delta+1}}\|w_{\sigma}\|^2_{\tilde{H}^s_{\gamma,\co}}\right)\mathrm{d}t\notag\\&\qquad\qquad\qquad\qquad\qquad\qquad\qquad\qquad\qquad\qquad\qquad\lesssim_{s,\gamma_0}\|w_{\sigma}(0)\|^2_{H^s_{\gamma,\co}}+ \Lambda(\tau),
    \end{align}
    where 
    \begin{align}
    \label{AP31_1}\Lambda(\tau)&:=C_{\delta}\mathbb{E}\int_0^{\tau}\chi^2(\mathcal{E}_w(t))\|\langle\nabla_x\rangle^{\frac{1}{2}}w_{\sigma}\|^2_{X^s_{\gamma}}\left(\|w_{\sigma}\|_{X^s_{\gamma}}+\|U_{\sigma}\|_{H^{s+1}_x}\right)\mathrm{d}t\notag\\& \quad+C_{\delta}\mathbb{E}\int_0^{\tau}\chi^2(\mathcal{E}_w(t))\|\langle\nabla_x\rangle^{\frac{1}{2}}w_{\sigma}\|_{X^s_{\gamma}}\|\partial_yw_{\sigma}\|_{X^s_{\gamma}}\left(\|w_{\sigma}\|_{X^s_{\gamma}}+\|U_{\sigma}\|_{H^{s+1}_x}\right)\mathrm{d}t\notag\\& \quad+C_{\delta}\mathbb{E}\int_0^{\tau}\|\langle\nabla_x\rangle^{\frac{1}{2}}w_{\sigma}\|_{X^s_{\gamma}}\|U_{\sigma}\|^2_{H^{s+1}_x}\mathrm{d}t+\mathbb{E}\int_0^{\tau}(\|w_{\sigma}\|^2_{X^s_{\gamma}}+\|\nabla_x P_{\sigma}\|^2_{H^{s}_x})\mathrm{d}t\notag\\&\quad+\mathbb{E}\int_0^{\tau}\left[\|w_{\sigma}\|_{\widetilde{X}^s_{\gamma}}\|U_{\sigma}\|_{H^{s+1}_x}\left(\|\partial_y w_{\sigma}\|_{X^{s}_{\gamma}}+\|U_{\sigma}\|_{H^{s+1}_x}\right)+\frac{1}{\langle t\rangle^{2\delta}}\|w_{\sigma}\|^2_{\widetilde{X}^s_{\gamma}}\right]\mathrm{d}t\notag\\&\quad+\mathbb{E}\left(\int_0^{\tau}\chi^4(\mathcal{E}_w(t))\|w_{\sigma}\|^2_{X^s_{\gamma}}\|[\F(w+u^s)-\psi\overline{\F}(U)]_{\sigma}\|^2_{\mathbb{X}^s_{\gamma}}\mathrm{d} t\right)^{\frac{1}{2}}\notag\\&\quad+\mathbb{E}\int_0^{\tau}\chi^4(\mathcal{E}_w(t))\|[\F(w+u^s)-\psi\overline{\F}(U)]_{\sigma}\|^2_{\mathbb{X}^s_{\gamma}}\mathrm{d} t.
    \end{align}
    \noindent {\textbf{Step I\!I. A priori estimate for $\partial_yw_{\sigma}$ in $H^{s-1,0}_{\gamma,\co}$.}} Let us take $\partial_y$ on both sides of \eqref{AP2} and apply the It\^o formula to derive
    \begin{align}
        \label{AP32}
        &\mathrm{d}\|\partial_y w_{\sigma}\|^2_{H^{s-1,0}_{\gamma,\co}}+2\left(\lambda\||\nabla_x|^{\frac{1}{2}}\partial_yw_{\sigma}\|^2_{H^{s-1,0}_{\gamma,\co}}+\frac{\gamma_0\delta}{\langle t\rangle^{\delta+1}}\|\partial_yw_{\sigma}\|^2_{\widetilde{H}^{s-1,0}_{\gamma,\co}}-\langle \partial_yw_{\sigma},\partial_y^3 w_{\sigma}\rangle_{H^{s-1,0}_{\gamma,\co}}\right)\mathrm{d}t\notag\\&\qquad\qquad +2\chi^2(\mathcal{E}_w(t))\langle \partial_yw_{\sigma},\partial_y\mathcal{B}_1(w,w)_{\sigma}+\partial_y\mathcal{B}_1(w,u^s)_{\sigma}+\partial_y\mathcal{B}_1(u^s,w)_{\sigma}\rangle_{H^{s-1,0}_{\gamma,\co}}\mathrm{d}t\notag\\&\qquad\qquad  +2\chi^2(\mathcal{E}_w(t))\langle \partial_yw_{\sigma},\partial_y\mathcal{B}_2(w+u^s,w+u^s)_{\sigma}\rangle_{H^{s-1,0}_{\gamma,\co}}\mathrm{d}t\notag\\&=2\langle \partial_yw_{\sigma}, \partial_y\mathcal{B}_1(U-u^s,u^s)_{\sigma}-\partial_y(1-\psi)\nabla_x P_{\sigma}\rangle_{H^{s-1,0}_{\gamma,\co}}\mathrm{d} t\notag\\&\qquad\qquad+2\chi^2(\mathcal{E}_w(t))\langle \partial_yw_{\sigma}, \partial_y[\F(w+u^s)-\psi\overline{\F}(U)]_{\sigma}\mathrm{d} W\rangle_{H^{s-1,0}_{\gamma,\co}}\notag\\&\qquad\qquad+\chi^4(\mathcal{E}_w(t))\|\partial_y[\F(w+u^s)-\psi\overline{\F}(U)]_{\sigma}\|_{\mathbb{H}^{s-1,0}_{\gamma,\co}}^2\mathrm{d}t, 
    \end{align}
    where for each $i$,
    \begin{align}
        \label{AP32_1}
        \|\partial_yw_{\sigma}\|^2_{\widetilde{H}^{s-i-1,i}_{\gamma,\co}}:=\sum_{j\le i}\sum_{|k|\le s-i-1}\|ye^{\gamma y^2}Z^j\partial^k_x \partial_yw_{\sigma}\|^2_{L^2}.
    \end{align}
    Using equation \eqref{AP2} and the condition \eqref{struc0_3_1}, we have 
    \[\partial_y^2w_{\sigma}|_{y=0}=\nabla_x P_{\sigma}.\]
    Therefore, we get 
    \begin{align}\label{AP33}
        -\langle \partial_yw_{\sigma},\partial_y^3 w_{\sigma}\rangle_{H^{s-1,0}_{\gamma,\co}}&=-\sum_{|k|\le s-1}\langle e^{\gamma y^2}\partial_x^k\partial_yw_{\sigma},e^{\gamma y^2}\partial_x^k\partial^3_yw_{\sigma}\rangle_{L^2}\notag\\&=\|\partial_y^2 w_{\sigma}\|^2_{H^{s-1,0}_{\gamma,\co}}+4\gamma\sum_{|k|\le s-1}\langle ye^{\gamma y^2}\partial_x^k\partial_yw_{\sigma},e^{\gamma y^2}\partial_x^k\partial^2_yw_{\sigma}\rangle_{L^2}\notag\\&\quad-\sum_{|k|\le s-1}\langle \partial_x^k\partial_yw_{\sigma}|_{y=0},\partial_x^k\nabla_xP_{\sigma}\rangle_{L^2},
    \end{align}
    where 
    \begin{align}\label{AP34}
        4\gamma\sum_{|k|\le s-1}\langle ye^{\gamma y^2}\partial_x^k\partial_yw_{\sigma},e^{\gamma y^2}\partial_x^k\partial^2_yw_{\sigma}\rangle_{L^2}\le \frac{4\gamma_0}{\langle t\rangle^{\delta}}\|\partial_yw_{\sigma}\|_{\tilde{H}^{s-1,0}_{\gamma,\co}}\|\partial_y^2 w_{\sigma}\|_{H^{s-1,0}_{\gamma,\co}}.
    \end{align}
    To control the boundary integral, we use the estimate
    \[\|f|_{y=0}\|_{L^2_x}\le \sqrt{2}\|f\|^{\frac{1}{2}}_{L^2}\|\partial_yf\|^{\frac{1}{2}}_{L^2}\]
    to obtain
    \begin{align}\label{AP35}
    \sum_{|k|\le s-1}\langle \partial_x^k\partial_yw_{\sigma}|_{y=0},\partial_x^k\nabla_xP_{\sigma}\rangle_{L^2}\lesssim_s \|\nabla_x P_{\sigma}\|_{H^s_x}\|\partial_yw_{\sigma}\|_{H^{s-1,0}_{\gamma,\co}}^{\frac{1}{2}}\|\partial_y^2w_{\sigma}\|_{H^{s-1,0}_{\gamma,\co}}^{\frac{1}{2}}.
    \end{align}
    Combining the estimates \eqref{AP33}--\eqref{AP35}, we have
    \begin{align*}
        -\langle \partial_yw_{\sigma},\partial_y^3 w_{\sigma}\rangle_{H^{s-1,0}_{\gamma,\co}}\ge \frac{1}{2}\|\partial_y^2 w_{\sigma}\|_{H^{s-1,0}_{\gamma,\co}}^2-C_{s,\gamma_0}\left(\|\nabla_x P_{\sigma}\|_{H^s_x}^2+\|w_{\sigma}\|^2_{X^s_{\gamma}}+\frac{1}{\langle t\rangle^{2\delta}}\|w_{\sigma}\|^2_{\widetilde{X}^s_{\gamma}}\right).
    \end{align*}
    This, together with the identity \eqref{AP32}, Lemmas \ref{APlemma1}--\ref{APlemma3}, and the Burkholder--Davis--Gundy inequality \eqref{pre6}, implies
    \begin{align}
        \label{AP37}
        &\mathbb{E}\sup_{t\le \tau}\|\partial_yw_{\sigma}\|^2_{H^{s-1,0}_{\gamma,\co}}+2\mathbb{E}\int_0^{\tau}\left(\lambda\||\nabla_x|^{\frac{1}{2}} \partial_yw_{\sigma}\|^2_{H^{s-1,0}_{\gamma,\co}}+\frac{1}{2}\|\partial^2_y w_{\sigma}\|^2_{H^{s-1,0}_{\gamma,\co}}\right)\mathrm{d}t\notag\\&\qquad\qquad\qquad +2\mathbb{E}\int_0^{\tau}\frac{\gamma_0\delta}{\langle t\rangle^{\delta+1}}\|\partial_yw_{\sigma}\|^2_{\widetilde{H}^{s-1,0}_{\gamma,\co}}\mathrm{d}t\lesssim_{s,\gamma_0} \|\partial_yw_{\sigma}(0)\|^2_{H^{s-1,0}_{\gamma,\co}}+\Lambda(\tau),
    \end{align}
    where $\Lambda(\tau)$ is given in \eqref{AP31_1}.

    \noindent {\textbf{Step I\!I\!I. A priori estimate for $\partial_y w_{\sigma}$ in $H^{s-1}_{\gamma,\co}$.}} Due to the relation \eqref{relation}, we establish an estimate for $\|\partial_y w_\sigma\|_{H^{s-1}_{\gamma, \co}}$ by induction on the index $k$ of $H_{\gamma, \co}^{s-1-k, k}$. Assume that for some $i\le s-1$, we have
    \begin{align}
        \label{AP38}
        &\mathbb{E}\sup_{t\le \tau}\|\partial_yw_{\sigma}\|^2_{H^{s-i,i-1}_{\gamma,\co}}\notag\\&\qquad\qquad+2\mathbb{E}\int_0^{\tau}\left(\lambda\||\nabla_x|^{\frac{1}{2}} \partial_yw_{\sigma}\|^2_{H^{s-i,i-1}_{\gamma,\co}}+\frac{1}{2}\|\partial^2_y w_{\sigma}\|^2_{H^{s-i,i-1}_{\gamma,\co}}+\frac{\gamma_0\delta}{\langle t\rangle^{\delta+1}}\|\partial_yw_{\sigma}\|^2_{\widetilde{H}^{s-i,i-1}_{\gamma,\co}}\right)\mathrm{d}t\notag\\&\qquad\lesssim_{s,\gamma_0}\|\partial_yw_{\sigma}(0)\|^2_{H^{s-1}_{\gamma,\co}}+ \Lambda(\tau),
    \end{align}
    where $\Lambda(\tau)$ is given in \eqref{AP31_1}, and $\|\partial_yw_{\sigma}\|_{\widetilde{H}^{s-i,i-1}_{\gamma,\co}}$ is defined in \eqref{AP32_1}. Then, by the It\^o formula, we get
    \begin{align}
        \label{AP39}
        &\mathrm{d}\|\partial_y w_{\sigma}\|^2_{H^{s-i-1,i}_{\gamma,\co}}+2\left(\lambda\||\nabla_x|^{\frac{1}{2}}\partial_yw_{\sigma}\|^2_{H^{s-i-1,i}_{\gamma,\co}}-\langle \partial_yw_{\sigma},\partial_y^3 w_{\sigma}\rangle_{H^{s-i-1,i}_{\gamma,\co}}\right)\mathrm{d}t\notag\\&\qquad\quad+2\chi^2(\mathcal{E}_w(t))\left(\langle \partial_yw_{\sigma},\partial_y\mathcal{B}_1(w,w)_{\sigma}+\partial_y\mathcal{B}_1(w,u^s)_{\sigma}+\partial_y\mathcal{B}_1(u^s,w)_{\sigma}\rangle_{H^{s-i-1,i}_{\gamma,\co}}\right)\mathrm{d}t\notag\\&\qquad\quad+2\chi^2(\mathcal{E}_w(t))\left(\langle \partial_yw_{\sigma},\partial_y\mathcal{B}_2(w+u^s,w+u^s)_{\sigma}\rangle_{H^{s-i-1,i}_{\gamma,\co}}\right)\mathrm{d}t+\frac{2\gamma_0\delta}{\langle t\rangle^{\delta+1}}\|\partial_yw_{\sigma}\|^2_{\tilde{H}^{s-i-1,i}_{\gamma,\co}}\mathrm{d}t\notag\\&=2\langle \partial_yw_{\sigma}, \partial_y\mathcal{B}_1(U-u^s,u^s)_{\sigma}-\partial_y(1-\psi)\nabla_x P_{\sigma}\rangle_{H^{s-i-1,i}_{\gamma,\co}}\mathrm{d} t\notag\\&\qquad\quad+2\chi^2(\mathcal{E}_w(t))\langle \partial_yw_{\sigma}, \partial_y[\F(w+u^s)-\psi\overline{\F}(U)]_{\sigma}\mathrm{d} W\rangle_{H^{s-i-1,i}_{\gamma,\co}}\notag\\&\qquad\quad+\chi^4(\mathcal{E}_w(t))\|\partial_y[\F(w+u^s)-\psi\overline{\F}(U)]_{\sigma}\|^2_{\mathbb{H}^{s-i-1,i}_{\gamma,\co}}\mathrm{d} t. 
    \end{align}
    For the third term on the left hand-side of \eqref{AP39}, we have
    \begin{align}\label{AP40}
    -\langle \partial_y w_{\sigma},\partial_y^{3} w_{\sigma}\rangle_{H^{s-i-1,i}_{\gamma,\co}}&=-\sum_{j\le i}\sum_{|k|\le s-i-1}\langle e^{\gamma y^2}Z^j\partial_x^k\partial_yw_{\sigma}, e^{\gamma y^2}Z^j\partial_x^k\partial_y^3w_{\sigma}\rangle_{L^2}\notag\\&=\|\partial_y^2 w_{\sigma}\|^2_{H^{s-i-1,i}_{\gamma,\co}}-\sum_{|k|\le s-i-1}\langle \partial_x^k\partial_yw_{\sigma}|_{y=0}, \partial_x^k\nabla_x P_{\sigma}\rangle_{L^2}\notag\\&\quad+\sum_{j\le i}\sum_{|k|\le s-i-1}2j\langle e^{\gamma y^2}Z^{j-1}\partial_x^k\partial_y^2w_{\sigma}, e^{\gamma y^2}Z^j\partial_x^k\partial_y^2w_{\sigma}\rangle_{L^2}\notag\\&\quad+4\gamma\sum_{j\le i}\sum_{|k|\le s-i-1}\langle ye^{\gamma y^2}Z^{j}\partial_x^k\partial_yw_{\sigma}, e^{\gamma y^2}Z^j\partial_x^k\partial_y^2w_{\sigma}\rangle_{L^2},
    \end{align}
    where
    \begin{align}
        \label{AP41}
        \sum_{j\le i}\sum_{|k|\le s-i-1}2j\langle e^{\gamma y^2}Z^{j-1}\partial_x^k\partial_y^2w_{\sigma}, e^{\gamma y^2}Z^j\partial_x^k\partial_y^2w_{\sigma}\rangle_{L^2}\lesssim_s\|\partial_y^2w_{\sigma}\|_{H^{s-i-1,i-1}_{\gamma,\co}}\|\partial_y^2 w_{\sigma}\|_{H^{s-i-1,i}_{\gamma,\co}},
    \end{align}
    and 
    \begin{align}
        \label{AP42}
        4\gamma\sum_{j\le i}\sum_{|k|\le s-i-1}\langle ye^{\gamma y^2}Z^{j}\partial_x^k\partial_yw_{\sigma}, e^{\gamma y^2}Z^j\partial_x^k\partial_y^2w_{\sigma}\rangle_{L^2}\le \frac{4\gamma_0}{\langle t\rangle^{\delta}}\|w_{\sigma}\|_{\widetilde{X}^{s}_{\gamma}}\|\partial_y^2 w_{\sigma}\|_{H^{s-i-1,i}_{\gamma,\co}}.
    \end{align}
    Similar to \eqref{AP35}, we have
    \begin{align}
        \label{AP43}
        \sum_{|k|\le s-i-1}\langle \partial_x^k\partial_yw_{\sigma}|_{y=0}, \partial_x^k\nabla_x P_{\sigma}\rangle_{L^2}&\lesssim_s \|\nabla_x P_{\sigma}\|_{H^s_x}\|\partial_yw_{\sigma}\|_{H^{s-i-1,0}_{\gamma,\co}}^{\frac{1}{2}}\|\partial_y^2w_{\sigma}\|_{H^{s-i-1,0}_{\gamma,\co}}^{\frac{1}{2}}\notag\\&\lesssim_s \|\nabla_x P_{\sigma}\|_{H^s_x}\|w_{\sigma}\|_{X^s_{\gamma}}^{\frac{1}{2}}\|\partial_y^2w_{\sigma}\|_{H^{s-i-1,i}_{\gamma,\co}}^{\frac{1}{2}}.
    \end{align}
    Combining the estimates \eqref{AP40}--\eqref{AP43}, we get
    \begin{align*}
        -\langle \partial_y w_{\sigma},\partial_y^{3} w_{\sigma}\rangle_{H^{s-i-1,i}_{\gamma,\co}}&\ge \frac{1}{2}\|\partial_y^2 w_{\sigma}\|_{H^{s-i-1,i}_{\gamma,\co}}^2-C_s\|\partial_y^2w_{\sigma}\|_{H^{s-i-1,i-1}_{\gamma,\co}}^2\notag\\&\qquad-C_{s,\gamma_0}\left(\|\nabla_x P_{\sigma}\|_{H^s_x}^2+\|w_{\sigma}\|^2_{X^s_{\gamma}}+\frac{1}{\langle t\rangle^{2\delta}}\|w_{\sigma}\|^2_{\widetilde{X}^s_{\gamma}}\right).
    \end{align*}
    This, together with the identity \eqref{AP39}, Lemmas \ref{APlemma1}--\ref{APlemma3}, and the Burkholder--Davis--Gundy inequality \eqref{pre6}, yileds
    \begin{align*}
        &\mathbb{E}\sup_{t\le \tau}\|\partial_yw_{\sigma}\|^2_{H^{s-i-1,i}_{\gamma,\co}}\notag\\&\qquad\qquad+2\mathbb{E}\int_0^{\tau}\left(\lambda\||\nabla_x|^{\frac{1}{2}} \partial_yw_{\sigma}\|^2_{H^{s-i-1,i}_{\gamma,\co}}+\frac{1}{2}\|\partial^2_y w_{\sigma}\|^2_{H^{s-i-1,i}_{\gamma,\co}}+\frac{\gamma_0\delta}{\langle t\rangle^{\delta+1}}\|\partial_yw_{\sigma}\|^2_{\widetilde{H}^{s-i-1,i}_{\gamma,\co}}\right)\mathrm{d}t\notag\\&\qquad\lesssim_{s,\gamma_0} \|\partial_yw_{\sigma}(0)\|^2_{H^{s-i-1,i}_{\gamma,\co}}+\Lambda(\tau)+\mathbb{E}\int_0^{\tau}\|\partial_y^2w_{\sigma}\|_{H^{s-i-1,i-1}_{\gamma,\co}}^2\mathrm{d}t.
    \end{align*}
    By the induction hypothesis \eqref{AP38}, we have    \begin{align*}\mathbb{E}\int_0^{\tau}\|\partial_y^2w_{\sigma}\|_{H^{s-i-1,i-1}_{\gamma,\co}}^2\mathrm{d}t\le \mathbb{E}\int_0^{\tau}\|\partial_y^2w_{\sigma}\|_{H^{s-i,i-1}_{\gamma,\co}}^2\mathrm{d}t\lesssim_{s,\gamma_0} \|\partial_yw_{\sigma}(0)\|^2_{H^{s-1}_{\gamma,\co}}+\Lambda(\tau),
    \end{align*}
    which leads to
    \begin{align*}
        &\mathbb{E}\sup_{t\le \tau}\|\partial_yw_{\sigma}\|^2_{H^{s-i-1,i}_{\gamma,\co}}\notag\\&\qquad\qquad+2\mathbb{E}\int_0^{\tau}\left(\lambda\||\nabla_x|^{\frac{1}{2}} \partial_yw_{\sigma}\|^2_{H^{s-i-1,i}_{\gamma,\co}}+\frac{1}{2}\|\partial^2_y w_{\sigma}\|^2_{H^{s-i-1,i}_{\gamma,\co}}+\frac{\gamma_0\delta}{\langle t\rangle^{\delta+1}}\|\partial_yw_{\sigma}\|^2_{\widetilde{H}^{s-i-1,i}_{\gamma,\co}}\right)\mathrm{d}t\notag\\&\qquad\lesssim_{s,\gamma_0} \|\partial_yw_{\sigma}(0)\|^2_{H^{s-1}_{\gamma,\co}}+\Lambda(\tau).
    \end{align*}
    Therefore, by induction on $i$, the estimate \eqref{AP38} holds for all $0\le i\le s-1$, which implies
 \begin{align}
        \label{AP47_1}
        \mathbb{E}&\sup_{t\le \tau}\|\partial_yw_{\sigma}\|^2_{H^{s-1}_{\gamma,\co}}+2\mathbb{E}\int_0^{\tau}\left(\lambda\||\nabla_x|^{\frac{1}{2}} \partial_yw_{\sigma}\|^2_{H^{s-1}_{\gamma,\co}}+\frac{1}{2}\|\partial^2_y w_{\sigma}\|^2_{H^{s-1}_{\gamma,\co}}+\frac{\gamma_0\delta}{\langle t\rangle^{\delta+1}}\|\partial_yw_{\sigma}\|^2_{\widetilde{H}^{s-1}_{\gamma,\co}}\right)\mathrm{d}t\notag\\&\qquad\qquad\qquad\qquad\qquad\qquad\qquad\qquad\qquad\lesssim_{s,\gamma_0} \|\partial_yw_{\sigma}(0)\|^2_{H^{s-1}_{\gamma,\co}}+\Lambda(\tau).
    \end{align}

  \noindent {\textbf{Step I\!V. Proof of the estimate \eqref{AP53}.}}
    Combining \eqref{AP31} and \eqref{AP47_1}, we obtain
    \begin{align}\label{AP48}
        \mathbb{E}\sup_{t\le \tau}\|w_{\sigma}\|^2_{X^s_{\gamma}}+2\mathbb{E}\int_0^{\tau}&\left(\lambda\||\nabla_x|^{\frac{1}{2}} w_{\sigma}\|^2_{X^s_{\gamma}}+\frac{1}{2}\|\partial_y w_{\sigma}\|^2_{X^s_{\gamma}}+\frac{\gamma_0\delta}{\langle t\rangle^{\delta+1}}\|w_{\sigma}\|^2_{\widetilde{X}^s_{\gamma}}\right)\mathrm{d}t\notag\\&\qquad\qquad\qquad\qquad\qquad\lesssim_{s,\gamma_0}\|w_{\sigma}(0)\|^2_{X^s_{\gamma}}+\Lambda(\tau),
    \end{align}
    where $\Lambda(\tau)$ is given in \eqref{AP31_1}. To control $\Lambda(\tau)$, let us choose the parameters $\delta,\lambda$ in a proper way. Using the assumption \eqref{APassumption}, we establish the following estimate for the fourth term on the right hand-side of \eqref{AP31_1}:
    \begin{align*}
        C_{s,\gamma_0}\mathbb{E}\int_0^{\tau}&\left[\|w_{\sigma}\|_{\widetilde{X}^s_{\gamma}}\|U_{\sigma}\|_{H^{s+1}_x}\left(\|\partial_y w_{\sigma}\|_{X^{s}_{\gamma}}+\|U_{\sigma}\|_{H^{s+1}_x}\right)+\frac{1}{\langle t\rangle^{2\delta}}\|w_{\sigma}\|^2_{\widetilde{X}^s_{\gamma}}\right]\mathrm{d}t\notag\\&\qquad\qquad\le  C_{s,\gamma_0}\mathbb{E}\int_0^{\tau}\left[\frac{\NN }{\langle t\rangle^{\delta}}\|w_{\sigma}\|_{\widetilde{X}^s_{\gamma}}\left(\|\partial_y w_{\sigma}\|_{X^{s}_{\gamma}}+\|U_{\sigma}\|_{H^{s+1}_x}\right)+\frac{1}{\langle t\rangle^{2\delta}}\|w_{\sigma}\|^2_{\widetilde{X}^s_{\gamma}}\right]\mathrm{d}t\notag\\&\qquad\qquad\le C_{s,\gamma_0,\NN }\mathbb{E}\int_0^{\tau}\frac{1}{\langle t\rangle^{2\delta}}\|w_{\sigma}\|^2_{\widetilde{X}^s_{\gamma}}\mathrm{d}t+\frac{1}{4}\mathbb{E}\int_0^{\tau}\left(\|\partial_y w_{\sigma}\|_{X^{s}_{\gamma}}^2+\|U_{\sigma}\|^2_{H^{s+1}_x}\right)\mathrm{d}t.
    \end{align*}
    By setting
    \begin{align}\label{AP48_1}
        \delta>1\vee\frac{C_{s,\gamma_0,\NN }}{2\gamma_0},
    \end{align} 
    we have
    \[\frac{2\gamma_0\delta}{\langle t\rangle^{\delta+1}}>\frac{C_{s,\gamma_0,\NN }}{\langle t\rangle^{\delta+1}}>\frac{C_{s,\gamma_0,\NN }}{\langle t\rangle^{2\delta}},\]
    which, together with \eqref{AP48}, implies
    \begin{align}\label{AP49}
        \mathbb{E}\sup_{t\le \tau}\|w_{\sigma}\|^2_{X^s_{\gamma}}+2\mathbb{E}\int_0^{\tau}\left(\lambda\||\nabla_x|^{\frac{1}{2}} w_{\sigma}\|^2_{X^s_{\gamma}}+\frac{3}{8}\|\partial_y w_{\sigma}\|^2_{X^s_{\gamma}}\right)\mathrm{d}t\lesssim_{s,\gamma_0,\NN } \|w_{\sigma}(0)\|^2_{X^s_{\gamma}}+\Lambda_1(\tau),
    \end{align}
    where
    \begin{align}
    \label{AP50}\Lambda_1(\tau):&=\mathbb{E}\int_0^{\tau}\chi^2(\mathcal{E}_w(t))\|\langle\nabla_x\rangle^{\frac{1}{2}}w_{\sigma}\|^2_{X^s_{\gamma}}\left(\|w_{\sigma}\|_{X^s_{\gamma}}+\|U_{\sigma}\|_{H^{s+1}_x}\right)\mathrm{d}t\notag\\&\quad +\mathbb{E}\int_0^{\tau}\chi^2(\mathcal{E}_w(t))\|\langle\nabla_x\rangle^{\frac{1}{2}}w_{\sigma}\|_{X^s_{\gamma}}\|\partial_yw_{\sigma}\|_{X^s_{\gamma}}\left(\|w_{\sigma}\|_{X^s_{\gamma}}+\|U_{\sigma}\|_{H^{s+1}_x}\right)\mathrm{d}t\notag\\&\quad +\mathbb{E}\int_0^{\tau}\left(\|\langle\nabla_x\rangle^{\frac{1}{2}}w_{\sigma}\|_{X^s_{\gamma}}+1\right)\|U_{\sigma}\|^2_{H^{s+1}_x}\mathrm{d}t+\mathbb{E}\int_0^{\tau}(\|w_{\sigma}\|^2_{X^s_{\gamma}}+\|\nabla_x P_{\sigma}\|^2_{H^{s}_x})\mathrm{d}t\notag\\&\quad +\mathbb{E}\left(\int_0^{\tau}\chi^4(\mathcal{E}_w(t))\|w_{\sigma}\|^2_{X^s_{\gamma}}\|[\F(w+u^s)-\psi\overline{\F}(U)]_{\sigma}\|^2_{\mathbb{X}^s_{\gamma}}\mathrm{d} t\right)^{\frac{1}{2}}\notag\\&\quad +\mathbb{E}\int_0^{\tau}\chi^4(\mathcal{E}_w(t))\|[\F(w+u^s)-\psi\overline{\F}(U)]_{\sigma}\|^2_{\mathbb{X}^s_{\gamma}}\mathrm{d} t.
    \end{align}
    To bound $\Lambda_1$, we define
    \begin{align}\label{AP50_1}
    \tau_{3\MM }:=\inf\{t\ge0| \mathcal{E}_w(t)\ge 3\MM \},
    \end{align}
    where $\MM$ is the constant given in \eqref{chi}. As $\mathcal{E}_w(t)$ is non-decreasing in $t$, by Lemma \ref{APlemma3}, we have
    \begin{align}\label{AP51}
    \mathbb{E}&\left(\int_0^{\tau}\chi^4(\mathcal{E}_w(t))\|w_{\sigma}\|^2_{X^s_{\gamma}}\|[\F(w+u^s)-\psi\overline{\F}(U)]_{\sigma}\|^2_{\mathbb{X}^s_{\gamma}}\mathrm{d} t\right)^{\frac{1}{2}}\notag\\&\qquad\qquad\lesssim_{s,\MM ,\NN }\mathbb{E}\left(\int_0^{\tau\wedge \tau_{3\MM }}\left(\|\langle\nabla_x\rangle^{a} w_{\sigma}\|_{X^s_{\gamma}}^2+1\right)\mathrm{d}t\right)^{\frac{1}{2}}\notag\\&\qquad\qquad\lesssim_{s,\gamma_0,\MM ,\NN }\mathbb{E}\left(1+\mathcal{E}_w^2(\tau_{3\MM })\right)^{\frac{1}{2}}\le C_{s,\gamma_0,\MM ,\NN }.
    \end{align}
    Similarly, 
    \begin{align}
        \label{AP51_1}
        \mathbb{E}\int_0^{\tau}\chi^4(\mathcal{E}_w(t))\|[\F(w+u^s)-\psi\overline{\F}(U)]_{\sigma}\|^2_{\mathbb{X}^s_{\gamma}}\mathrm{d} t\le C_{s,\gamma_0,\MM ,\NN }.
    \end{align}
    As for the remaining terms on the right hand-side of \eqref{AP50}, we have
    \begin{align}
        \label{AP51_2}
        \mathbb{E}\int_0^{\tau}\chi^2&(\mathcal{E}_w(t))\|\langle\nabla_x\rangle^{\frac{1}{2}}w_{\sigma}\|^2_{X^s_{\gamma}}\left(\|w_{\sigma}\|_{X^s_{\gamma}}+\|U_{\sigma}\|_{H^{s+1}_x}\right)\mathrm{d}t\notag\\&+\mathbb{E}\int_0^{\tau}\chi^2(\mathcal{E}_w(t))\|\langle\nabla_x\rangle^{\frac{1}{2}}w_{\sigma}\|_{X^s_{\gamma}}\|\partial_yw_{\sigma}\|_{X^s_{\gamma}}\left(\|w_{\sigma}\|_{X^s_{\gamma}}+\|U_{\sigma}\|_{H^{s+1}_x}\right)\mathrm{d}t\notag\\\le& C_{s,\MM ,\NN }\mathbb{E}\int_0^{\tau\wedge \tau_{3\MM }}\||\nabla_x|^{\frac{1}{2}} w_{\sigma}\|^2_{X^s_{\gamma}}\mathrm{d}t+\frac{1}{4}\mathbb{E}\int_0^{\tau\wedge \tau_{3\MM }}\|\partial_yw_{\sigma}\|^2_{X^s_{\gamma}}\mathrm{d}t,
    \end{align}
    and
    \begin{align}
        \label{AP51_3}
        \mathbb{E}\int_0^{\tau}&\left(\|\langle\nabla_x\rangle^{\frac{1}{2}}w_{\sigma}\|_{X^s_{\gamma}}+1\right)\|U_{\sigma}\|^2_{H^{s+1}_x}\mathrm{d}t+\mathbb{E}\int_0^{\tau}(\|w_{\sigma}\|^2_{X^s_{\gamma}}+\|\nabla_x P_{\sigma}\|^2_{H^{s}_x})\mathrm{d}t\notag\\&\qquad\qquad\qquad\le C_{s,\gamma_0,\MM ,\NN }+C_{s,\gamma_0,\MM ,\NN }\mathbb{E}\int_0^{\tau}\left(\||\nabla_x|^{\frac{1}{2}} w_{\sigma}\|^2_{X^s_{\gamma}}+\|w_{\sigma}\|^2_{X^s_{\gamma}}\right)\mathrm{d}t.
    \end{align}
    By taking
    \[2\lambda\ge C_{s,\gamma_0,\MM ,\NN }+C_{s,\MM ,\NN }+1,\]
    and combining the estimates \eqref{AP51}--\eqref{AP51_3} with \eqref{AP49}, we conclude that
    \begin{align}\label{AP52}
        \mathbb{E}\sup_{t\le \tau}\|w_{\sigma}\|^2_{X^s_{\gamma}}+\mathbb{E}\int_0^{\tau}&\left(\||\nabla_x|^{\frac{1}{2}} w_{\sigma}\|^2_{X^s_{\gamma}}+\|\partial_y w_{\sigma}\|^2_{X^s_{\gamma}}\right)\mathrm{d}t\notag\\&\qquad\lesssim_{s,\gamma_0,\MM ,\NN } \mathbb{E}\int_0^{\tau}\|w_{\sigma}\|^2_{X^s_{\gamma}}\mathrm{d}t+\mathbb{E}\|w_{0}\|^2_{X^s_{\gamma_0,\sigma_0}}+1,
    \end{align}
   which implies  the desired estimate \eqref{AP53} by using the Gronwall inequality. 
    \end{proof}
    
    \subsection{Higher-order tangential estimates}
    Next, we turn to the higher-order tangential estimates, which will be used later to prove the convergence of approximate solutions.
    \begin{prop}\label{AP_proposition2}
        Let $s\ge 4$, $\gamma_0,\sigma_0>0$, and let assumption \eqref{APassumption} be satisfied. Then, there exist positive constants $\lambda,\delta$ such that for any smooth solution $w$ of the equation \eqref{AP1} and stopping time $\tau\le T_*$, 
        \begin{align}\label{AP68}
        \mathbb{E}\sup_{t\le \tau}\|\nabla_xw_{\sigma}\|^2_{X^s_{\gamma}}+\mathbb{E}\int_0^{\tau}&\left(\||\nabla_x|^{\frac{3}{2}} w_{\sigma}\|^2_{X^s_{\gamma}}+\|\partial_y\nabla_x w_{\sigma}\|^2_{X^s_{\gamma}}\right)\mathrm{d}t\notag\\&\qquad\qquad\qquad\lesssim_{s,\gamma_0,\MM ,\NN } \mathbb{E}\|\nabla_xw_{0}\|^2_{X^s_{\gamma_0,\sigma_0}}+1,
    \end{align}
    where the parameters $\sigma=\sigma(t)$ and $\gamma=\gamma(t)$ are given in \eqref{MR1}, and $T_*$ is defined in \eqref{AP1_1}.
    \end{prop}
    
    Similar to the proof of Proposition \ref{AP_proposition1}, the estimate \eqref{AP68} follows directly through an induction-based argument, if the following Lemmas \ref{APlemma4} and \ref{APlemma5} are proved.
    
    \begin{lemma}\label{APlemma4}
        {\rm(Higher-order estimate for the convection terms).} Under the assumption of Proposition \ref{AP_proposition2}, we have
        \begin{align}
        \label{AP60}
        \langle \nabla_x w_{\sigma},\nabla_x\mathcal{B}_1(w,w)_{\sigma}\rangle_{X^s_{\gamma}}&\lesssim_s\| w_{\sigma}\|_{X^s_{\gamma}}(\|\nabla_x w_{\sigma}\|_{X^s_{\gamma}}^2+\|\langle\nabla_x \rangle^{\frac{3}{2}}w_{\sigma}\|_{X^s_{\gamma}}^2),\\
        \label{AP61}
        \langle \nabla_x w_{\sigma},\nabla_x\mathcal{B}_1(w,u^s)_{\sigma}\rangle_{X^s_{\gamma}}&\lesssim_{s,\gamma_0} \|U_{\sigma}\|_{H^{s+2}_{x}} (\|w_{\sigma}\|_{X^{s}_{\gamma}}^2+\|\nabla_x w_{\sigma}\|_{X^{s}_{\gamma}}^2),\\\label{AP62}
        \langle \nabla_x w_{\sigma},\nabla_x\mathcal{B}_1(u^s,w)_{\sigma}\rangle_{X^s_{\gamma}}&\lesssim_{s,\gamma_0} \|U_{\sigma}\|_{H^{s+2}_{x}} (\|\nabla_x w_{\sigma}\|_{X^s_{\gamma}}^2+\|\langle\nabla_x \rangle^{\frac{3}{2}}w_{\sigma}\|_{X^s_{\gamma}}^2),
        \end{align}
        and
        \begin{align}\label{AP62_1}
        &\langle \nabla_x w_{\sigma},\nabla_x\mathcal{B}_2(w+u^s,w+u^s)_{\sigma}\rangle_{X^s_{\gamma}}\notag\\&\qquad\qquad\qquad\lesssim_{s,\gamma_0}C_{\delta}\|\langle\nabla_x\rangle^{\frac{3}{2}} w_{\sigma}\|_{X^s_{\gamma}}(\|\langle\nabla_x\rangle^{\frac{3}{2}} w_{\sigma}\|_{X^s_{\gamma}}+\|U_{\sigma}\|_{H^{s+2}_x})(\|w_{\sigma}\|_{X^s_{\gamma}}+\|U_{\sigma}\|_{H^{s+2}_x})\notag\\&\qquad\qquad\qquad\qquad\quad+C_{\delta}\|\langle\nabla_x\rangle^{\frac{3}{2}} w_{\sigma}\|_{X^s_{\gamma}}\|\partial_y w_{\sigma}\|_{X^s_{\gamma}}(\|\nabla_x w_{\sigma}\|_{X^s_{\gamma}}+\|U_{\sigma}\|_{H^{s+2}_x})\notag\\&\qquad\qquad\qquad\qquad\quad+C_{\delta}\|\nabla_x w_{\sigma}\|_{X^s_{\gamma}}\|\partial_y\langle\nabla_x \rangle w_{\sigma}\|_{X^s_{\gamma}}(\|w_{\sigma}\|_{X^s_{\gamma}}+\|U_{\sigma}\|_{H^{s+2}_x})\notag\\&\qquad\qquad\qquad\qquad\quad+\|U_{\sigma}\|_{H^{s+2}_x}\|\nabla_x w_{\sigma}\|_{\tilde{X}^{s}_{\gamma}}(\|\partial_y\langle\nabla_x \rangle w_{\sigma}\|_{X^s_{\gamma}}+\|U_{\sigma}\|_{H^{s+2}_x})
        \end{align}
        for $t\le T_*$.
    \end{lemma}

    \begin{proof}
        The estimates \eqref{AP60}--\eqref{AP62} follow from a direct application of Lemmas \ref{product_estimate1}, \ref{product_estimate2}, and \ref{APlemma1}. To prove \eqref{AP62_1}, we decompose
        \begin{align*}
        \langle \nabla_x w_{\sigma}&,\nabla_x\mathcal{B}_2(w+u^s,w+u^s)_{\sigma}\rangle_{X^s_{\gamma}}\\&=\langle \nabla_x w_{\sigma},\nabla_x\mathcal{B}_2(w+u^s-U,w+u^s-U)_{\sigma}\rangle_{X^s_{\gamma}}+\langle \nabla_x w_{\sigma},\nabla_x\mathcal{B}_2(U,w+u^s-U)_{\sigma}\rangle_{X^s_{\gamma}}\\&=:J_1+J_2.
    \end{align*}
    For $J_1$, we have
    \begin{align*}
        J_1&\le \|\langle\nabla_x\rangle^{\frac{3}{2}} w_{\sigma}\|_{X^s_{\gamma}}\|\langle\nabla_x\rangle^{-\frac{1}{2}} [\partial_y^{-1}\nabla_x^2(w+u^s-U)\partial_y(w+u^s-U)]_{\sigma}\|_{X^s_{\gamma}}\\&\qquad\qquad\qquad+\|\nabla_xw_{\sigma}\|_{X^s_{\gamma}}\| [\partial_y^{-1}\nabla_x(w+u^s-U)\partial_y\nabla_x(w+u^s-U)]_{\sigma}\|_{X^s_{\gamma}}.
    \end{align*}
    On the other hand, by applying Lemma \ref{product_estimate1}, we obtain 
    \begin{align*}
        &\|\langle\nabla_x\rangle^{-\frac{1}{2}} [\partial_y^{-1}\nabla_x^2(w+u^s-U)\partial_y(w+u^s-U)]_{\sigma}\|_{X^s_{\gamma}}\\&\qquad\qquad\qquad\qquad\lesssim_{s,\delta}\|\langle\nabla_x\rangle^{\frac{3}{2}}(w+u^s-U)_{\sigma}\|_{X^s_{\gamma}}\|(w+u^s-U)_{\sigma}\|_{X^s_{\gamma}}\notag\\&\qquad\qquad\qquad\qquad\quad+\|\nabla_x(w+u^s-U)_{\sigma}\|_{X^s_{\gamma}}\|\partial_y(w+u^s)_{\sigma}\|_{X^s_{\gamma}},
    \end{align*}
    and 
    \begin{align*}
        \|&[\partial_y^{-1}\nabla_x(w+u^s-U)\partial_y\nabla_x(w+u^s-U)]_{\sigma}\|_{X^s_{\gamma}}\\&\quad\qquad\qquad\qquad\lesssim_{s,\delta}\|\nabla_x(w+u^s-U)_{\sigma}\|_{X^s_{\gamma}}\|\partial_y(w+u^s)_{\sigma}\|_{X^s_{\gamma}}\notag\\&\qquad\qquad\qquad\qquad+\|(w+u^s-U)_{\sigma}\|_{X^s_{\gamma}}\|\partial_y\nabla_x(w+u^s-U)_{\sigma}\|_{X^s_{\gamma}}.
    \end{align*}
    This yields 
    \begin{align}\label{AP63}
        J_1\lesssim_{s,\gamma_0,\delta}&\|\langle\nabla_x\rangle^{\frac{3}{2}} w_{\sigma}\|_{X^s_{\gamma}}(\|\langle\nabla_x\rangle^{\frac{3}{2}} w_{\sigma}\|_{X^s_{\gamma}}+\|U_{\sigma}\|_{H^{s+2}_x})(\|w_{\sigma}\|_{X^s_{\gamma}}+\|U_{\sigma}\|_{H^{s+2}_x})\notag\\+&\|\langle\nabla_x\rangle^{\frac{3}{2}} w_{\sigma}\|_{X^s_{\gamma}}\|\partial_y w_{\sigma}\|_{X^s_{\gamma}}(\|\nabla_x w_{\sigma}\|_{X^s_{\gamma}}+\|U_{\sigma}\|_{H^{s+2}_x})\notag\\+&\|\nabla_x w_{\sigma}\|_{X^s_{\gamma}}\|\partial_y\langle\nabla_x \rangle w_{\sigma}\|_{X^s_{\gamma}}(\|w_{\sigma}\|_{X^s_{\gamma}}+\|U_{\sigma}\|_{H^{s+2}_x}).
    \end{align}
     For $J_2$, by repeating the derivation of \eqref{AP17}, we get
    \begin{align}
    \label{AP64}
    J_2\lesssim_{s,\gamma_0}\|U_{\sigma}\|_{H^{s+2}_x}(\|\nabla_x w_{\sigma}\|_{X^s_{\gamma}}+\|\nabla_x w_{\sigma}\|_{\tilde{X}^s_{\gamma}})(\|\partial_y \langle \nabla_x\rangle w_{\sigma}\|_{X^s_{\gamma}}+\|U_{\sigma}\|_{H^{s+2}_x}).
    \end{align}
    Combining the estimates \eqref{AP63} and \eqref{AP64}, we obtain \eqref{AP62_1}.
    \end{proof}

        \begin{lemma}\label{APlemma5}
        {\rm(Higher-order estimate for the force terms).} Under the assumption of Proposition \ref{AP_proposition2}, we have
        \begin{align}
        \label{AP65}
        \langle \nabla_xw_{\sigma}, \nabla_x\mathcal{B}_1(U-u^s,u^s)_{\sigma}-(1-\psi)\nabla_x^2 P_{\sigma} \rangle_{X^s_{\gamma}}\lesssim_{s,\gamma_0}\|\nabla_xw_{\sigma}\|_{X^s_{\gamma}}(\|U_{\sigma}\|^2_{H^{s+2}_x}+\|\nabla_x P_{\sigma}\|_{H^{s+1}_{x}}),
        \end{align}
        and
        \begin{align}\label{AP67}
            \|\nabla_x[\F(w+u^s)-\psi\overline{\F}(U)]_{\sigma}\|_{\mathbb{X}^{s}_{\gamma}}&\lesssim_{\gamma_0} \mathcal{K}_s(\|w_{\sigma}\|_{X^s_{\gamma}},\|U_{\sigma}\|_{H^s_{x}})\left(\|\langle\nabla_x\rangle^{\frac{3}{2}}w_{\sigma}\|_{X^s_{\gamma}}+\|U_{\sigma}\|_{H^{s+2}_{x}}\right)
        \end{align}
        for $t\le T_*$, where $\mathcal{K}_s$ is a positive non-decreasing function depending only on $s$.
    \end{lemma}

    \begin{proof}
        The estimate \eqref{AP65} follows from Lemmas \ref{product_estimate1}, \ref{product_estimate2}, and \ref{APlemma1}, while the estimate \eqref{AP67} can be obtained by using Proposition \ref{Proposition_struc1}, as in the proof of \eqref{AP21}.
    \end{proof}

   The next proposition ensures the H\"older continuity of the parabolic energy at $t=0$. This property plays an important role in the convergence of approximate solutions, see \eqref{LWP26} in Lemma \ref{abstract_cauchy_lemma}.

    \begin{prop}\label{AP_proposition3}
        Let $s\ge 4$, $\gamma_0,\sigma_0>0$, and let assumption \eqref{APassumption} be satisfied. Then, there exist positive constants $\lambda,\delta$ such that for any smooth solution $w$ of the equation \eqref{AP1} and stopping time $\tau\le T_*$, the following estimate holds 
        \begin{align}
        \label{AP59}
        \mathbb{E}\sup_{t\le t' \wedge\tau}\left|\|w_{\sigma}\|^2_{X^s_{\gamma}}-\|w_{\sigma}(0)\|^2_{X^s_{\gamma}}\right|&+\mathbb{E}\int_0^{t' \wedge\tau}\left(\||\nabla_x|^{\frac{1}{2}} w_{\sigma}\|^2_{X^s_{\gamma}}+\|\partial_y w_{\sigma}\|^2_{X^s_{\gamma}}\right)\mathrm{d}t\notag\\&\qquad\lesssim_{s,\gamma_0,\MM ,\NN } (\sqrt{t' }+t' )\left(\mathbb{E}\|\langle\nabla_x\rangle w_{0}\|^2_{X^s_{\gamma_0,\sigma_0}}+1\right)
    \end{align}
   for any $t' \ge0$, where the parameters $\sigma=\sigma(t)$ and $\gamma=\gamma(t)$ are given in \eqref{MR1}, and $T_*$ is defined in \eqref{AP1_1}.
    \end{prop}
    \begin{proof}
        Repeating the derivation of \eqref{AP49}, we have
        \begin{align}\label{AP59_1}
        \mathbb{E}\sup_{t\le t'\wedge\tau}\left|\|w_{\sigma}\|^2_{X^s_{\gamma}}-\|w_{\sigma}(0)\|^2_{X^s_{\gamma}}\right|&+2\mathbb{E}\int_0^{t'\wedge\tau}\left(\lambda\||\nabla_x|^{\frac{1}{2}} w_{\sigma}\|^2_{X^s_{\gamma}}+\frac{3}{8}\|\partial_y w_{\sigma}\|^2_{X^s_{\gamma}}\right)\mathrm{d}t\notag\\&\qquad\qquad\qquad\qquad\lesssim_{s,\gamma_0,\NN } \Lambda_1(t'\wedge\tau),
    \end{align}
    where $\Lambda_1$ is given in \eqref{AP50}. It remains to bound $\Lambda_1$ in a proper way. First, we have
        \begin{align}
        \label{AP54}
        \mathbb{E}\int_0^{t'\wedge\tau}\chi^2(\mathcal{E}_w(t))&\|\langle\nabla_x\rangle^{\frac{1}{2}}w_{\sigma}\|^2_{X^s_{\gamma}}\left(\|w_{\sigma}\|_{X^s_{\gamma}}+\|U_{\sigma}\|_{H^{s+1}_x}\right)\mathrm{d}t\notag\\&+\mathbb{E}\int_0^{t'\wedge\tau}\chi^2(\mathcal{E}_w(t))\|\langle\nabla_x\rangle^{\frac{1}{2}}w_{\sigma}\|_{X^s_{\gamma}}\|\partial_yw_{\sigma}\|_{X^s_{\gamma}}\left(\|w_{\sigma}\|_{X^s_{\gamma}}+\|U_{\sigma}\|_{H^{s+1}_x}\right)\mathrm{d}t\notag\\\le& C_{s,\MM ,\NN }\mathbb{E}\int_0^{t'\wedge\tau}\|\langle\nabla_x\rangle^{\frac{1}{2}} w_{\sigma}\|^2_{X^s_{\gamma}}\mathrm{d}t+\frac{1}{4}\mathbb{E}\int_0^{t'\wedge\tau}\|\partial_yw_{\sigma}\|^2_{X^s_{\gamma}}\mathrm{d}t,
    \end{align}
    and
    \begin{align}
        \label{AP56}
        \mathbb{E}\int_0^{t' \wedge\tau}\left(\|\langle\nabla_x\rangle^{\frac{1}{2}}w_{\sigma}\|_{X^s_{\gamma}}+1\right)\|U_{\sigma}\|^2_{H^{s+1}_x}&\mathrm{d}t+\mathbb{E}\int_0^{t' \wedge\tau}(\|w_{\sigma}\|^2_{X^s_{\gamma}}+\|\nabla_x P_{\sigma}\|^2_{H^{s}_x})\mathrm{d}t\notag\\&\lesssim_{s,\MM ,\NN }\mathbb{E}\int_0^{t' \wedge\tau}\left(\|\langle\nabla_x\rangle^{\frac{1}{2}}w_{\sigma}\|^2_{X^s_{\gamma}}+1\right)\mathrm{d}t.
    \end{align}
    Applying Lemma \ref{APlemma3}, we obtain
    \begin{align}\label{AP56_1}
        \mathbb{E}\int_0^{t'\wedge\tau}\chi^4(\mathcal{E}_w(t))\|[\F(w+u^s)-\psi\overline{\F}(U)]_{\sigma}\|^2_{\mathbb{X}^s_{\gamma}}\mathrm{d} t\lesssim_{s,\gamma_0,\MM ,\NN }\mathbb{E}\int_0^{t' \wedge\tau}\left(\|\langle\nabla_x\rangle^{\frac{1}{2}}w_{\sigma}\|^2_{X^s_{\gamma}}+1\right)\mathrm{d}t,
    \end{align}
    and 
    \begin{align}\label{AP58}
        &\notag\mathbb{E}\left(\int_0^{t' \wedge\tau}\chi^4(\mathcal{E}_w(t))\|w_{\sigma}\|^2_{X^s_{\gamma}}\|[\F(w+u^s)-\psi\overline{\F}(U)]_{\sigma}\|^2_{\mathbb{X}^s_{\gamma}}\mathrm{d} t\right)^{\frac{1}{2}}\notag\\&\qquad\qquad\qquad\qquad\lesssim_{s,\gamma_0,\MM ,\NN }\left(\mathbb{E}\int_0^{t' \wedge\tau}\left(\|\langle\nabla_x\rangle^{\frac{1}{2}} w_{\sigma}\|_{X^s_{\gamma}}^2+1\right)\mathrm{d}t\right)^{\frac{1}{2}}.
    \end{align}
    Combining \eqref{AP54}--\eqref{AP58} with Propositions \ref{AP_proposition1} and \ref{AP_proposition2}, we conclude that 
    \begin{align*}
        &\mathbb{E}\sup_{t\le t' \wedge\tau}\left|\|w_{\sigma}\|^2_{X^s_{\gamma}}-\|w_{\sigma}(0)\|^2_{X^s_{\gamma}}\right|+\mathbb{E}\int_0^{t' \wedge\tau}\left(\||\nabla_x|^{\frac{1}{2}} w_{\sigma}\|^2_{X^s_{\gamma}}+\|\partial_y w_{\sigma}\|^2_{X^s_{\gamma}}\right)\mathrm{d}t\notag\\&\qquad\lesssim_{s,\gamma_0,\MM ,\NN } \left(\mathbb{E}\int_0^{t' \wedge\tau}\left(\|\langle\nabla_x\rangle^{\frac{1}{2}} w_{\sigma}\|_{X^s_{\gamma}}^2+1\right)\mathrm{d}t\right)^{\frac{1}{2}}+\mathbb{E}\int_0^{t' \wedge\tau}\left(\|\langle\nabla_x\rangle^{\frac{1}{2}} w_{\sigma}\|_{X^s_{\gamma}}^2+1\right)\mathrm{d}t\notag\\&\qquad\lesssim_{s,\gamma_0,\MM ,\NN } \sqrt{t'} \left(\mathbb{E}\sup_{t\le T_*}\|\langle\nabla_x\rangle w_{\sigma}\|_{X^s_{\gamma}}^2+1\right)^{\frac{1}{2}}+t'\left(\mathbb{E}\sup_{t\le T_*}\|\langle\nabla_x\rangle w_{\sigma}\|_{X^s_{\gamma}}^2+1\right)\\&\qquad\lesssim_{s,\gamma_0,\MM ,\NN } (\sqrt{t' }+t' )\left(\mathbb{E}\|\langle\nabla_x\rangle w_{0}\|^2_{X^s_{\gamma_0,\sigma_0}}+1\right),
    \end{align*}
    This completes the proof.
    \end{proof}

    \section{Local existence of analytic solutions for the problem with general multiplicative noises}\label{LWP}
    In this section, we study the local wellposedness of the stochastic Prandtl equation \eqref{Intro1} in a tangentially analytic and normally Sobolev-type space. Inspired by \cite{KWX2023}, we propose three approximation schemes \eqref{LWP1}--\eqref{LWP3} and combine the a priori estimates obtained in Propositions \ref{AP_proposition1}, \ref{AP_proposition2}, and \ref{AP_proposition3} to prove the convergence of approximate solutions.

    \subsection{The approximation schemes}\label{schemes}
    \noindent {\textbf{Approximation scheme I.}} First, we choose the following approximation scheme for the homogenized and truncated stochastic Prandtl equation \eqref{AP1}:
    \begin{align}\label{LWP1}
    \begin{cases}
        \mathrm{d}w_n-\partial_y^2w_n\mathrm{d}t+\chi^2(\mathcal{E}_{w_n}(t))\left[\mathcal{B}_1(w_n,\mathcal{R}_{n}w_n)+\mathcal{B}_1(w_n,u^s)+\mathcal{B}_1(u^s,\mathcal{R}_{n}w_n)\right]\mathrm{d} t\\\qquad +\chi^2(\mathcal{E}_{w_n}(t))\mathcal{B}_2(\mathcal{R}_{n}w_n+u^s,w_n+u^s)\mathrm{d} t\\
        \quad
        = \left[\mathcal{B}_1(U-u^s,u^s)-(1-\psi)\nabla_x P\right]\mathrm{d} t+\chi^2(\mathcal{E}_{w_n}(t))[\F(\mathcal{R}_{n}w_n+u^s)-\psi \overline{\F}(U)]\mathrm{d}W,\\
        w_n|_{y=0}=\lim_{y\to\infty} w_n=0,\qquad
        w_n(0)=w_0:=u_0-U_0\erf(\sqrt{2\gamma_0}y),
    \end{cases}
    \end{align}
    where $\mathcal{R}_{n}f:=\mathcal{F}_x^{-1}(1_{|\xi|\le n}\mathcal{F}_xf)$ denotes the tangential regularizing operator with $1_{|\xi|\le n}(\xi)$ being the characteristic function, and $\mathcal{B}_1,\mathcal{B}_2$ are defined in \eqref{bilinear_term}. 
    
    \noindent {\textbf{Approximation scheme I\!I.}} To study the wellposedness of the above equation for each fixed $n$, we introduce the following iteration scheme:
    \begin{align}\label{LWP2}
    \begin{cases}
        \mathrm{d}w^{m+1}-\partial_y^2w^{m+1}\mathrm{d}t\\\qquad+\chi(\mathcal{E}_{w^{m+1}}(t))\chi(\mathcal{E}_{w^{m}}(t))\left[\mathcal{B}_1(w^{m},\mathcal{R}w^{m})+\mathcal{B}_1(w^{m},u^s)+\mathcal{B}_1(u^s,\mathcal{R}w^{m})\right]\mathrm{d} t\\\qquad+\chi(\mathcal{E}_{w^{m+1}}(t))\chi(\mathcal{E}_{w^{m}}(t))\mathcal{B}_2(\mathcal{R}w^{m}+u^s,w^{m}+u^s)\mathrm{d} t\\
        \quad= \left[\mathcal{B}_1(U-u^s,u^s)-(1-\psi)\nabla_x P\right]\mathrm{d} t\\\qquad+\chi(\mathcal{E}_{w^{m+1}}(t))\chi(\mathcal{E}_{w^{m}}(t))[\F(\mathcal{R}w^{m}+u^s)-\psi \overline{\F}(U)]\mathrm{d}W,\\
        w^{m+1}|_{y=0}=\lim_{y\to\infty} w^{m+1}=0,\qquad
        w^{m+1}(0)=w_0:=u_0-U_0\erf\left(\sqrt{2\gamma_0}y\right),
    \end{cases}
    \end{align}
    where $\mathcal{R}:=\mathcal{R}_n$ denotes the tangential regularizing operator introduced in \eqref{LWP1} for simplicity, and the initial state $w^0$ is the solution of the following degenerate heat equation:
     \begin{align}\label{LWP3}
    \begin{cases}
        \partial_t w^0=\partial_y^2w^0,\\
        w^{0}|_{y=0}=\lim_{y\to\infty} w^{0}=0,\\
        w^{0}(0)=w_0(x,y).
    \end{cases}
    \end{align}
    
    \noindent {\textbf{Approximation scheme I\!I\!I.}} Given the process $w^{m}$, the equation \eqref{LWP2} is still a semi-linear stochastic PDE for $w^{m+1}$. To solve this equation, we introduce the following iteration scheme:
    \begin{align}\label{LWP4}
    \begin{cases}
        \mathrm{d}w^{(m+1)}-\partial_y^2w^{(m+1)}\mathrm{d}t\\\qquad+\chi(\mathcal{E}_{w^{(m)}}(t))\chi(\mathcal{E}_{w}(t))\left[\mathcal{B}_1(w,\mathcal{R}w)+\mathcal{B}_1(w,u^s)+\mathcal{B}_1(u^s,\mathcal{R}w)\right]\mathrm{d} t\\\qquad+\chi(\mathcal{E}_{w^{(m)}}(t))\chi(\mathcal{E}_{w}(t))\left[\mathcal{B}_2(\mathcal{R}w+u^s,w+u^s)\right]\mathrm{d} t\\\quad= 
        \left[\mathcal{B}_1(U-u^s,u^s)-(1-\psi)\nabla_x P\right]\mathrm{d} t\\\qquad+\chi(\mathcal{E}_{w^{(m)}}(t))\chi(\mathcal{E}_{w}(t))[\F(\mathcal{R}w+u^s)-\psi \overline{\F}(U)]\mathrm{d}W,\\
        w^{(m+1)}|_{y=0}=\lim_{y\to\infty} w^{(m+1)}=0,\quad
        w^{(m+1)}(0)=w_0:=u_0-U_0\erf\left(\sqrt{2\gamma_0}y\right),
    \end{cases}
    \end{align}
    where $w$ is a given progressively measurable process valued in $H^1_yH^2_x$ that satisfies
    \begin{align}
    \label{LWP5}w|_{y=0}=0,\qquad\mathcal{E}_w(T_*)<\infty,\qquad\mbox{$\PP$-a.s.}
    \end{align}
    with $T_*$ defined in \eqref{AP1_1}. Here, we set $w^{(0)}$ as the solution of \eqref{LWP3}.

    Given the processes $w^{(m)}$ and $w$, the equation \eqref{LWP4} is a linear stochastic PDE with additive noise. The wellposedness theory of this equation, which can be established directly through the variational approach, see \cite{LR2015}, is stated as follows.
    
    \begin{lemma}Let $s\ge 4$, $\gamma_0,\sigma_0>0$, and let assumption \eqref{APassumption} be satisfied. Suppose that $w^{(m)}, w$ are progressively measurable processes valued in $H^1_yH^2_x$ which satisfy \eqref{LWP5}. Then, there exists a unique solution $w^{(m+1)}$ of the equation \eqref{LWP4} satisfying \eqref{LWP5}.
    \end{lemma}

    \subsection{Convergence of the approximation schemes I-\texorpdfstring{I\!I\!I}{III}}
    In this subsection, we study the convergence of the approximation schemes I-I\!I\!I. 

    \subsubsection{Convergence of the approximation scheme \texorpdfstring{I\!I\!I}{III}}
    First, we address the convergence of the approximation scheme I\!I\!I, as $m\to\infty$, which ensures the wellposedness of \eqref{LWP2}.
    
    \begin{prop}\label{LWP_proposition1} Let $s\ge 4$, $\gamma_0,\sigma_0>0$, and let assumption \eqref{APassumption} be satisfied. Suppose that $w^m$ is a progressively measurable process valued in $H^1_yH^2_x$ satisfying \eqref{LWP5}, and 
    \[\mathbb{E}\|w_0\|^2_{X^s_{\gamma_0,\sigma_0}}<\infty.\]
    Then, there exists a unique solution $w^{m+1}$ of \eqref{LWP2} satisfying \eqref{LWP5}.
    \end{prop}
    
    \begin{proof}Let $z^{(m)}:=w^{(m+1)}-w^{(m)}$, where for each $m$, $w^{(m+1)}$ is the solution of \eqref{LWP4} with given $w^{(m)}$ and $w$. Then, $z^{(m)}$ satisfies the following equation:
    \begin{align}\label{LWP6}
    \begin{cases}
        \mathrm{d}z^{(m)}-\partial_y^2z^{(m)}\mathrm{d}t+(I^1+I^2)\ddd t
        = J\ddd W,\\
        z^{(m)}(0)=0,\qquad z^{(m)}|_{y=0}=\lim_{y\to\infty} z^{(m)}=0,
    \end{cases}
    \end{align}
    where
    \begin{align*}
        I^1&:=[\chi(\mathcal{E}_{w^{(m)}}(t))-\chi(\mathcal{E}_{w^{(m-1)}}(t))]\chi(\mathcal{E}_{w}(t))\left[\mathcal{B}_1(w,\mathcal{R}w)+\mathcal{B}_1(w,u^s)+\mathcal{B}_1(u^s,\mathcal{R}w)\right],\\
        I^2&:=[\chi(\mathcal{E}_{w^{(m)}}(t))-\chi(\mathcal{E}_{w^{(m-1)}}(t))]\chi(\mathcal{E}_{w}(t))\mathcal{B}_2(\mathcal{R}w+u^s,w+u^s),\\
        J&:=[\chi(\mathcal{E}_{w^{(m)}}(t))-\chi(\mathcal{E}_{w^{(m-1)}}(t))]\chi(\mathcal{E}_{w}(t))[\F(\mathcal{R}w+u^s)-\psi \overline{\F}(U)].
    \end{align*}
    As shown in the proof of Proposition \ref{AP_proposition1}, the estimate of $z^{(m)}$ can be obtained through an induction-based argument, if the estimates for the convection terms and the force terms are available. To this end, we apply Lemma \ref{APlemma2} combined with the property of $\mathcal{R}$ to get
   \begin{align}\label{LWP7}
       \mathbb{E}\int_0^{\tau}\langle I^1_{\sigma},\notag& z^{(m)}_{\sigma}\rangle_{X^s_{\gamma}}\mathrm{d}t\\&\lesssim\mathbb{E}\int_0^{\tau\wedge\tau_{3\MM }}\mathcal{E}_{z^{(m-1)}}(t)|\langle\mathcal{B}_1(w,\mathcal{R}w)_{\sigma}+\mathcal{B}_1(w,u^s)_{\sigma}+\mathcal{B}_1(u^s,\mathcal{R}w)_{\sigma},z^{(m)}_{\sigma} \rangle_{X^s_{\gamma}}|\mathrm{d}t\notag\\&\lesssim_{s,\gamma_0}\mathbb{E}\int_0^{\tau\wedge\tau_{3\MM }}\mathcal{E}_{z^{(m-1)}}(t)\|z^{(m)}_{\sigma}\|_{X^s_{\gamma}}\|w_{\sigma}\|_{X^s_{\gamma}}(\|w_{\sigma}\|_{X^s_{\gamma}}+1)\mathrm{d}t\notag\\&\lesssim_{s,\gamma_0,\MM,\NN }\mathbb{E}\left(\tau\mathcal{E}_{z^{(m-1)}}^2(\tau)\right)+\mathbb{E}\left(\tau\mathcal{E}_{z^{(m)}}^2(\tau)\right),
    \end{align} 
    where $\tau_{3\MM }$ is defined in \eqref{AP50_1} and $\tau\le T_*$. To estimate $I^2$, we notice that 
    \begin{align*}
        &\langle\mathcal{B}_2(\mathcal{R}w+u^s,w+u^s)_{\sigma},z^{(m)}_{\sigma} \rangle_{X^s_{\gamma}}\\&\qquad\qquad\lesssim_{s,\gamma_0} C_{\delta}\|z^{(m)}_{\sigma} \|_{X^s_{\gamma}}(\|w_{\sigma} \|_{X^s_{\gamma}}+\|U_{\sigma}\|_{H^{s+1}_x})(\|w_{\sigma} \|_{X^s_{\gamma}}+\|\partial_yw_{\sigma} \|_{X^s_{\gamma}}+\|U_{\sigma}\|_{H^{s+1}_x})\\&\qquad\qquad\qquad\quad+\|z^{(m)}_{\sigma} \|_{\widetilde{X}^s_{\gamma}}\|U_{\sigma}\|_{H^{s+1}_x}(\|\partial_yw_{\sigma} \|_{X^s_{\gamma}}+\|U_{\sigma}\|_{H^{s+1}_x}),
    \end{align*}
    due to the property of $\mathcal{R}$, which implies
    \begin{align}
        \label{LWP8}
        &\mathbb{E}\int_0^{\tau}\langle I^2_{\sigma},z^{(m)}_{\sigma} \rangle_{X^s_{\gamma}}\mathrm{d}t\lesssim \mathbb{E}\int_0^{\tau\wedge\tau_{3\MM }}\mathcal{E}_{z^{(m-1)}}(t)|\langle\mathcal{B}_2(\mathcal{R}w+u^s,w+u^s)_{\sigma},z^{(m)}_{\sigma} \rangle_{X^s_{\gamma}}|\mathrm{d}t\notag\\&\qquad\qquad\lesssim_{s,\gamma_0,\MM ,\NN }C_{\delta}\mathbb{E}\int_0^{\tau\wedge\tau_{3\MM }}\mathcal{E}_{z^{(m-1)}}(t)\|z^{(m)}_{\sigma} \|_{X^s_{\gamma}}(1+\|\partial_yw_{\sigma} \|_{X^s_{\gamma}})\mathrm{d}t\notag\\&\qquad\qquad\qquad+\mathbb{E}\int_0^{\tau\wedge\tau_{3\MM }}\mathcal{E}_{z^{(m-1)}}(t)\frac{1}{\langle t\rangle^{\delta}}\|z^{(m)}_{\sigma} \|_{\widetilde{X}^s_{\gamma}}(1+\|\partial_yw_{\sigma} \|_{X^s_{\gamma}})\mathrm{d}t\notag\\&\qquad\qquad\lesssim_{s,\gamma_0,\MM ,\NN } C_{\delta}\mathbb{E}\left(\tau\mathcal{E}_{z^{(m-1)}}^2(\tau)\right)+C_{\delta}\mathbb{E}\left(\tau\mathcal{E}_{z^{(m)}}^2(\tau)\right)\notag\\&\qquad\qquad\qquad+\mathbb{E}\int_0^{\tau}\frac{1}{\langle t\rangle^{2\delta}}\|z^{(m)}_{\sigma} \|_{\widetilde{X}^s_{\gamma}}^2\mathrm{d}t+\epsilon\mathbb{E}\int_0^{\tau\wedge\tau_{3\MM }}\mathcal{E}_{z^{(m-1)}}^2(t)\|\partial_yw_{\sigma} \|^2_{X^s_{\gamma}}\mathrm{d}t.
    \end{align}
    Let us choose $\epsilon>0$ sufficiently small so that 
    \begin{align*}
        \epsilon C_{s,\gamma_0,\MM ,\NN }\mathbb{E}&\int_0^{\tau\wedge\tau_{3\MM }}\mathcal{E}_{z^{(m-1)}}^2(t)\|\partial_yw_{\sigma} \|^2_{X^s_{\gamma}}\mathrm{d}t \notag\\&\qquad\qquad\qquad\qquad\le \epsilon C_{s,\gamma_0,\MM ,\NN }\mathbb{E}\left(\mathcal{E}_{z^{(m-1)}}^2(\tau)\int_0^{\tau\wedge\tau_{3\MM }}\|\partial_yw_{\sigma} \|^2_{X^s_{\gamma}}\mathrm{d}t\right)\notag\\&\qquad\qquad\qquad\qquad\le \epsilon C_{s,\gamma_0,\MM ,\NN }\mathbb{E}\mathcal{E}_{z^{(m-1)}}^2(\tau)\le \frac{1}{16}\mathbb{E}\mathcal{E}_{z^{(m-1)}}^2(\tau).
    \end{align*}
    Combining this with \eqref{LWP8}, we obtain the following estimate for $I^2$:
    \begin{align}
        \label{LWP10}
        &\mathbb{E}\int_0^{\tau}\langle I^2_{\sigma},z^{(m)}_{\sigma} \rangle_{X^s_{\gamma}}\mathrm{d}t\notag\\&\qquad\qquad\le C_{s,\gamma_0,\MM ,\NN }\left[C_{\delta}\mathbb{E}\left(\tau\mathcal{E}_{z^{(m-1)}}^2(\tau)\right)+C_{\delta}\mathbb{E}\left(\tau\mathcal{E}_{z^{(m)}}^2(\tau)\right)+\mathbb{E}\int_0^{\tau}\frac{1}{\langle t\rangle^{2\delta}}\|z^{(m)}_{\sigma} \|_{\tilde{X}^s_{\gamma}}^2\mathrm{d}t\right]\notag\\&\qquad\qquad\quad+\frac{1}{16}\mathbb{E}\mathcal{E}_{z^{(m-1)}}^2(\tau).
    \end{align}
    We turn to the random force term $J$ given in \eqref{LWP6}. Applying the Burkholder--Davis--Gundy inequality \eqref{pre6} combined with Lemma \ref{APlemma3}, we infer
    \begin{align}
        \label{LWP11}
        &\mathbb{E}\int_0^{\tau}\langle z^{(m)}_{\sigma},J_{\sigma}\mathrm{d}W\rangle_{X^s_{\gamma}}\lesssim\mathbb{E}\left(\int_0^{\tau\wedge\tau_{3\MM }}\mathcal{E}_{z^{(m-1)}}^2(t)\|z^{(m)}_{\sigma}\|_{X^s_{\gamma}}^2\|[\F(\mathcal{R}w+u^s)-\psi \overline{\F}(U)]_{\sigma}\|^2_{\mathbb{X}^s_{\gamma}}\mathrm{d}t\right)^{\frac{1}{2}}\notag
        \\&\qquad\qquad\qquad\lesssim_{s,\gamma_0,\MM ,\NN }\mathbb{E}\left(\int_0^{\tau}\mathcal{E}_{z^{(m-1)}}^2(t)\|z^{(m)}_{\sigma}\|_{X^s_{\gamma}}^2\mathrm{d}t\right)^{\frac{1}{2}}\notag
        \\&\qquad\qquad\qquad\le \frac{1}{16}\mathbb{E}\mathcal{E}_{z^{(m-1)}}^2(\tau)+C_{s,\gamma_0,\MM ,\NN }\mathbb{E}\left(\tau\mathcal{E}_{z^{(m)}}^2(\tau)\right).
    \end{align}
     Similarly, we have the following estimate for the quadratic variation term:
    \begin{align}
        \label{LWP12}
         \mathbb{E}\int_0^{\tau}\|J_{\sigma}\|^2_{\mathbb{X}^s_{\gamma}}\mathrm{d}t\lesssim_{s,\gamma_0,\MM ,\NN }\mathbb{E}\left(\tau\mathcal{E}_{z^{(m-1)}}^2(\tau)\right).
    \end{align}
    Combining \eqref{LWP7} with \eqref{LWP10}--\eqref{LWP12} and choosing 
    \begin{align}\label{LWP12_1}
        2\lambda\gg1,\qquad \delta>1\vee\frac{C_{s,\gamma_0,\MM ,\NN }}{2\gamma_0},\qquad \tau\le \frac{1}{16C_{s,\gamma_0,\MM ,\NN }}\wedge T_*,
    \end{align}
    we conclude that
    \begin{align}
        \label{LWP13}
        \mathbb{E}\mathcal{E}_{z^{(m)}}^2(\tau)\le \frac{5}{13}\mathbb{E}\mathcal{E}_{z^{(m-1)}}^2(\tau).
    \end{align}
    Thus, there exists a process $\tilde{w}$ satisfying 
    \[\lim_{m\to\infty}\mathbb{E}\mathcal{E}_{w^{(m)}-\tilde{w}}^2(\tau)=0.\]
    This convergence ensures that $\tilde{w}$ is a solution of \eqref{LWP2} with given $w^m$. 
    
    It is standard to adapt the above contraction argument to prove the pathwise uniqueness, and we omit the details. The solution exists up to $T_*$, as the lifespan does not depend on the initial data. This completes the proof of Proposition \ref{LWP_proposition1}.
    \end{proof}

    \subsubsection{Convergence of the approximation scheme \texorpdfstring{I\!I}{II}}
    
    Next, we prove the convergence of the approximation scheme I\!I, as $m\to \infty$. This implies the wellposedness of \eqref{LWP1}. 
    
    \begin{prop}\label{LWP_proposition2}
    Let $s\ge 4$, $\gamma_0,\sigma_0>0$, and let assumption \eqref{APassumption} be satisfied. Suppose that 
    \[\mathbb{E}\|w_0\|^2_{X^s_{\gamma_0,\sigma_0}}<\infty.\]
    Then, there exists a unique solution $w_n$ of \eqref{LWP1} satisfying \eqref{LWP5}.
    \end{prop}
    
    \begin{proof}
       Let $z^m:=w^{m+1}-w^m$, where for each $m$, $w^{m+1}$ is the solution of \eqref{LWP2} with given $w^m$. Then, $z^m$ satisfies the following equation:
    \begin{align}\label{LWP15}
    \begin{cases}
        \mathrm{d}z^{m}-\partial_y^2z^{m}\mathrm{d}t+(\mathcal{I}_1+\mathcal{I}_2+\mathcal{I}_3+\mathcal{I}_4+\mathcal{I}_5+\mathcal{I}_6)\mathrm{d}t=(\JJ_1+\JJ_2+\JJ_3)\mathrm{d}W,\\
        z^{m}(0)=0,\qquad z^{m}|_{y=0}=\lim_{y\to\infty} z^{m}=0,
    \end{cases}
    \end{align}
    where
    \begin{align*}
        &\mathcal{I}_1:=[\chi(\mathcal{E}_{w^{m+1}}(t))-\chi(\mathcal{E}_{w^{m}}(t))]\chi(\mathcal{E}_{w^m}(t))\left[\mathcal{B}_2(\mathcal{R}w^m+u^s,w^m+u^s)\right],\notag\\
        &\mathcal{I}_2:=[\chi(\mathcal{E}_{w^{m+1}}(t))-\chi(\mathcal{E}_{w^{m}}(t))]\chi(\mathcal{E}_{w^m}(t))\left[\mathcal{B}_1(w^m,\mathcal{R}w^m)+\mathcal{B}_1(w^m,u^s)+\mathcal{B}_1(u^s,\mathcal{R}w^m)\right],\notag\\&\mathcal{I}_3:=[\chi(\mathcal{E}_{w^{m}}(t))-\chi(\mathcal{E}_{w^{m-1}}(t))]\chi(\mathcal{E}_{w^m}(t))\left[\mathcal{B}_2(\mathcal{R}w^m+u^s,w^m+u^s)\right],\notag\\
        &\mathcal{I}_4:=[\chi(\mathcal{E}_{w^{m}}(t))-\chi(\mathcal{E}_{w^{m-1}}(t))]\chi(\mathcal{E}_{w^m}(t))\left[\mathcal{B}_1(w^m,\mathcal{R}w^m)+\mathcal{B}_1(w^m,u^s)+\mathcal{B}_1(u^s,\mathcal{R}w^m)\right],\notag\\
        &\mathcal{I}_5:=\chi(\mathcal{E}_{w^{m}}(t))\chi(\mathcal{E}_{w^{m-1}}(t))\left[\mathcal{B}_2(\mathcal{R}z^{m-1},w^m+u^s)+\mathcal{B}_2(\mathcal{R}w^{m-1}+u^s,z^{m-1})\right],\notag\\
        &\mathcal{I}_6:=\chi(\mathcal{E}_{w^{m}}(t))\chi(\mathcal{E}_{w^{m-1}}(t))\left[\mathcal{B}_1(z^{m-1},\mathcal{R}w^m)+\mathcal{B}_1(w^{m-1},\mathcal{R}z^{m-1})\right]\notag\\& \qquad+\chi(\mathcal{E}_{w^{m}}(t))\chi(\mathcal{E}_{w^{m-1}}(t))\left[\mathcal{B}_1(z^{m-1},u^s)+\mathcal{B}_1(u^s,\mathcal{R}z^{m-1})\right],
    \end{align*}
    and 
    \begin{align*}
        \JJ_1&:=[\chi(\mathcal{E}_{w^{m+1}}(t))-\chi(\mathcal{E}_{w^{m}}(t))]\chi(\mathcal{E}_{w^m}(t))[\F(\mathcal{R}w^m+u^s)-\psi \overline{\F}(U)],\notag\\
        \JJ_2&:=[\chi(\mathcal{E}_{w^{m}}(t))-\chi(\mathcal{E}_{w^{m-1}}(t))]\chi(\mathcal{E}_{w^m}(t))[\F(\mathcal{R}w^m+u^s)-\psi \overline{\F}(U)],\notag\\
        \JJ_3&:=\chi(\mathcal{E}_{w^{m}}(t))\chi(\mathcal{E}_{w^{m-1}}(t))[\F(\mathcal{R}w^m+u^s)-\F(\mathcal{R}w^{m-1}+u^s)].\notag
    \end{align*}
    Similar to the proofs of \eqref{LWP7} and \eqref{LWP10}, we have
    \begin{align}
        \label{LWP16}
        \mathbb{E}\int_0^{\tau} \langle (\mathcal{I}_1+\mathcal{I}_2+\mathcal{I}_3+\mathcal{I}_4)_{\sigma}, z^m_{\sigma}\rangle_{X^s_{\gamma}}\mathrm{d}t&\le C_{s,\gamma_0,\MM ,\NN ,\delta}\left[\mathbb{E}\left(\tau\mathcal{E}_{z^{m-1}}^2(\tau)\right)+\mathbb{E}\left(\tau\mathcal{E}_{z^{m}}^2(\tau)\right)\right]\notag\\& \quad+C_{s,\gamma_0,\MM ,\NN }\mathbb{E}\int_0^{\tau}\frac{1}{\langle t\rangle^{2\delta}}\|z^{m}_{\sigma} \|_{\widetilde{X}^s_{\gamma}}^2\mathrm{d}t\notag\\&\quad+\frac{1}{16}\mathbb{E}\mathcal{E}_{z^{m}}^2(\tau)+\frac{1}{16}\mathbb{E}\mathcal{E}_{z^{m-1}}^2(\tau).
    \end{align}
    For $\mathcal{I}_5$ and $\mathcal{I}_6$, we apply Lemmas \ref{product_estimate1}, \ref{product_estimate2}, and \ref{APlemma1} to get 
    \begin{align*}
        &\langle\left[\mathcal{B}_2(\mathcal{R}z^{m-1},w^m+u^s)+\mathcal{B}_2(\mathcal{R}w^{m-1}+u^s,z^{m-1})\right]_{\sigma},z^m_{\sigma}\rangle_{X^s_{\gamma}}\notag\\&\qquad\lesssim_{s,\gamma_0} C_{\delta} \|z^m_{\sigma}\|_{X^s_{\gamma}}\left(\|z^{m-1}_{\sigma}\|_{X^s_{\gamma}}\|\partial_yw^{m}_{\sigma}\|_{X^s_{\gamma}}+\|\partial_yz^{m-1}_{\sigma}\|_{X^s_{\gamma}}\|w^{m-1}_{\sigma}\|_{X^s_{\gamma}}\right)\notag\\
        &\qquad\quad+ C_{\delta} \|z^m_{\sigma}\|_{X^s_{\gamma}}\|U_{\sigma}\|_{H^{s+1}_x}(\|z^{m-1}_{\sigma}\|_{X^s_{\gamma}}+\|\partial_yz^{m-1}_{\sigma}\|_{X^s_{\gamma}})+\|U_{\sigma}\|_{H^{s+1}_x}\|\partial_yz^{m-1}_{\sigma}\|_{X^s_{\gamma}}\|z^{m}_{\sigma}\|_{\widetilde{X}^s_{\gamma}},
    \end{align*}
    and
    \begin{align*}
        &\langle\left[\mathcal{B}_1(z^{m-1},\mathcal{R}w^m)+\mathcal{B}_1(w^{m-1},\mathcal{R}z^{m-1})+\mathcal{B}_1(z^{m-1},u^s)+\mathcal{B}_1(u^s,\mathcal{R}z^{m-1})\right]_{\sigma},z^m_{\sigma}\rangle_{X^s_{\gamma}}\notag\\&\qquad\qquad\qquad\qquad\lesssim_{s,\gamma_0}\|z^m_{\sigma}\|_{X^s_{\gamma}}\|z^{m-1}_{\sigma}\|_{X^s_{\gamma}}\left(\|w^m_{\sigma}\|_{X^s_{\gamma}}+\|w^{m-1}_{\sigma}\|_{X^s_{\gamma}}+\|U_{\sigma}\|_{H^{s+1}_x}\right).
    \end{align*}
    Therefore, we have
    \begin{align}
        \label{LWP19}
        \mathbb{E}\int_0^{\tau} \langle (\mathcal{I}_5+\mathcal{I}_6)_{\sigma}, z^m_{\sigma}\rangle_{X^s_{\gamma}}\mathrm{d}t&\le C_{s,\gamma_0,\MM ,\NN ,\delta}\left[\mathbb{E}\left(\tau\mathcal{E}_{z^{m-1}}^2(\tau)\right)+\mathbb{E}\left(\tau\mathcal{E}_{z^{m}}^2(\tau)\right)\right]\notag\\&\quad +C_{s,\gamma_0,\MM ,\NN }\mathbb{E}\int_0^{\tau}\frac{1}{\langle t\rangle^{2\delta}}\|z^{m}_{\sigma} \|_{\widetilde{X}^s_{\gamma}}^2\mathrm{d}t\notag\\&\quad+\frac{1}{16}\mathbb{E}\mathcal{E}_{z^{m}}^2(\tau)+\frac{1}{16}\mathbb{E}\mathcal{E}_{z^{m-1}}^2(\tau).
    \end{align}
    It remains to bound the random force terms given in \eqref{LWP15}. To this end, we apply Lemma \ref{APlemma3} to obtain
    \begin{align}
        \label{LWP20}
        \mathbb{E}\int_0^{\tau}\|(\JJ_1)_{\sigma}\|^2_{\mathbb{X}^s_{\gamma}}\mathrm{d}t&\lesssim \mathbb{E}\int_0^{\tau}\mathcal{E}_{z^m}^2(t)\chi^2(\mathcal{E}_{w^m}(t))\|[\F(\mathcal{R}w^m+u^s)-\psi \overline{\F}(U)]_{\sigma}\|^2_{\mathbb{X}^s_{\gamma}}\mathrm{d}t\notag\\&\lesssim_{s,\gamma_0,\MM ,\NN }\mathbb{E}\left(\tau\mathcal{E}_{z^m}^2(\tau)\right),
    \end{align}
    and 
    \begin{align}
        \label{LWP21}
        \mathbb{E}\int_0^{\tau}\|(\JJ_2)_{\sigma}\|^2_{\mathbb{X}^s_{\gamma}}\mathrm{d}t&\lesssim \mathbb{E}\int_0^{\tau}\mathcal{E}_{z^{m-1}}^2(t)\chi^2(\mathcal{E}_{w^m}(t))\|[\F(\mathcal{R}w^m+u^s)-\psi \overline{\F}(U)]_{\sigma}\|^2_{\mathbb{X}^s_{\gamma}}\mathrm{d}t\notag\\&\lesssim_{s,\gamma_0,\MM ,\NN }\mathbb{E}\left(\tau\mathcal{E}_{z^{m-1}}^2(\tau)\right).
    \end{align}
    For $\JJ_3$, we use Proposition \ref{Proposition_struc1} to get
    \begin{align}
        \label{LWP22}
        &\mathbb{E}\int_0^{\tau}\|(\JJ_3)_{\sigma}\|^2_{\mathbb{X}^s_{\gamma}}\mathrm{d}t\notag\\&\qquad\lesssim \mathbb{E}\int_0^{\tau}\chi^2(\mathcal{E}_{w^m}(t))\chi^2(\mathcal{E}_{w^{m-1}}(t))\|[\F(\mathcal{R}w^m+u^s)-\F(\mathcal{R}w^{m-1}+u^s)]_{\sigma}\|^2_{\mathbb{X}^s_{\gamma}}\mathrm{d}t\notag\\&\qquad\lesssim_{s,\MM ,\NN }\mathbb{E}\int_0^{\tau}\|z_{\sigma}^{m-1}\|^2_{X^s_{\gamma}}\mathrm{d}t\lesssim_{s,\MM ,\NN }\mathbb{E}\left(\tau\mathcal{E}_{z^{m-1}}^2(\tau)\right).
    \end{align}
    Combining the estimates \eqref{LWP20}--\eqref{LWP22} with the Burkholder--Davis--Gundy inequality \eqref{pre6}, we obtain
    \begin{align}\label{LWP23}
        &\mathbb{E}\int_0^{\tau}\langle z^m_{\sigma},(\JJ_1+\JJ_2+\JJ_3)_{\sigma}\mathrm{d}W\rangle_{X^s_{\gamma}}\notag\\&\qquad\qquad\le C\mathbb{E}\left(\int_0^{\tau}\|z^m_{\sigma}\|^2_{X^s_{\gamma}}\|(\JJ_1+\JJ_2+\JJ_3)_{\sigma}\|^2_{\mathbb{X}^s_{\gamma}}\mathrm{d}t\right)^{\frac{1}{2}}\notag\\&\qquad\qquad\le \frac{1}{16}\mathbb{E}\sup_{t\le\tau}\|z^m_{\sigma}\|^2_{X^s_{\gamma}}+ C\mathbb{E}\int_0^{\tau}\|(\JJ_1+\JJ_2+\JJ_3)_{\sigma}\|^2_{\mathbb{X}^s_{\gamma}}\mathrm{d}t\notag\\
        &\qquad\qquad\le \frac{1}{16}\mathbb{E}\mathcal{E}_{z^m}^2(\tau)+C_{s,\gamma_0,\MM ,\NN }\left[\mathbb{E}\left(\tau\mathcal{E}_{z^{m-1}}^2(\tau)\right)+\mathbb{E}\left(\tau\mathcal{E}_{z^{m}}^2(\tau)\right)\right].
    \end{align}
    Let the parameters $\lambda,\delta,\tau$ satisfy \eqref{LWP12_1}. Utilizing the estimates \eqref{LWP16}, \eqref{LWP19}, and \eqref{LWP23}, together with the induction-based argument presented in Section \ref{AP}, we conclude
    \begin{align}
        \label{LWP24}
        \mathbb{E}\mathcal{E}_{z^{m}}^2(\tau)\le \frac{4}{11}\mathbb{E}\mathcal{E}_{z^{m-1}}^2(\tau).
    \end{align}
    This implies the convergence of $w^m$, and thereby completes the proof.
    \end{proof}

    \subsubsection{Convergence of the approximation scheme I}\label{convergence}
    Finally, we extract a convergent subsequence from the approximate solutions $\{w_n\}_{n=1}^\infty$, thereby obtaining a solution of the stochastic Prandtl equation. Instead of a fix-point argument, we apply the following abstract Cauchy lemma, see \cite{GZ2009} for a proof. 
    \begin{lemma}
        \label{abstract_cauchy_lemma} 
        Let $\{w_n\}_{n=1}^{\infty}$ be a sequence of progressively measurable processes valued in $H^1_yH_x^2$ which satisfy
        \[\mathcal{E}_{w_n}(T_*)<\infty,\qquad \forall n\in\mathbb{N},\qquad \mbox{$\PP$-a.s..}\]
        Suppose that 
        \begin{align}
            \label{LWP25}
            \lim_{n\to\infty}\sup_{m\ge n}\mathbb{E}\mathcal{E}_{w_n-w_m}(T_*\wedge \Lambda^{n,m}_{\MM })=0,
        \end{align}
        and 
        \begin{align}
            \label{LWP26}
            \lim_{t\to0}\sup_{n\ge 1}\PP \left\{\mathcal{E}_{w_n}(t\wedge T_*\wedge \Lambda^{n}_{\MM })\ge \mathcal{E}_{w_n}(0)+\MM \right\}=0,
        \end{align}
        where 
        \begin{align}\label{deflambda}
            \Lambda^n_{\MM }:=\left\{\mathcal{E}_{w_n}(t)\ge \mathcal{E}_{w_n}(0)+2\mathcal{\MM }\right\},\qquad\Lambda^{n,m}_{\MM }:=\Lambda^n_{\MM }\wedge \Lambda^m_{\MM }.
        \end{align}        
        Then, there exists a positive stopping time $\tau$, a subsequence $\{w_{n_l}\}_{l=1}^\infty$, and a process $w$ satisfying
        \[\lim_{l\to\infty}\mathcal{E}_{w_{n_l}-w}(\tau)=0\qquad \mbox{$\PP$-a.s..}\]
    \end{lemma}
    The main result of this subsection is the following one.
    
    \begin{prop}
            \label{LWP_proposition3}
            Let $s\ge 4$, $\gamma_0,\sigma_0>0$, and let assumption \eqref{APassumption} be satisfied. Suppose that 
            \begin{align}\label{LWP50}
                \|\langle\nabla_x \rangle w_0\|_{X^s_{\gamma_0,\sigma_0}}\le \MM,\qquad \mbox{$\PP$-a.s..}
            \end{align}
            Then, there exist parameters $\lambda,\delta>0$ and a unique local pathwise solution $(w,\tau)$ of the homogenized and truncated stochastic Prandtl equation \eqref{AP1} such that
        \[w\in L^{\infty}\left(0,\tau; X^s_{\gamma(\cdot),\sigma(\cdot)}\right),\qquad\mbox{$\PP$-a.s.,}\]
        where $\gamma(\cdot),\sigma(\cdot)$ are given in \eqref{MR1}.
        \end{prop}
        
        \begin{proof} It suffices to verify the conditions \eqref{LWP25} and \eqref{LWP26} given in Lemma \ref{abstract_cauchy_lemma}. To this end, let us combine Propositions \ref{AP_proposition1} and \ref{AP_proposition2} to get
    \begin{align}\label{LWP27}
        \mathbb{E}\sup_{t\le T_*}\|\langle\nabla_x\rangle w_{\sigma}\|^2_{X^s_{\gamma}}+\mathbb{E}\int_0^{T_*}&\left(\|\langle\nabla_x\rangle^{\frac{3}{2}} w_{\sigma}\|^2_{X^s_{\gamma}}+\|\partial_y\langle\nabla_x\rangle w_{\sigma}\|^2_{X^s_{\gamma}}\right)\mathrm{d}t\le C_{s,\gamma_0,\MM ,\NN }
    \end{align}
    for any smooth solution $w$ of \eqref{AP1} satisfying \eqref{LWP50}.
    
    \noindent {\textbf{Step I. Verification of \eqref{LWP25}}}. Let $z^{n,m}:=w_n-w_m$, where for each $n$, $w_n$ is the solution of \eqref{LWP1}. Then, we deduce that $z^{n,m}$ satisfies the following equation: 
    \begin{align}\label{LWP28}
    \begin{cases}
        \mathrm{d}z^{n,m}-\partial_y^2z^{n,m}\mathrm{d}t+(\mathcal{I}^1+\mathcal{I}^2+\mathcal{I}^3+\mathcal{I}^4)\mathrm{d}t=(\JJ^1+\JJ^2)\mathrm{d}W\\
        z^{n,m}(0)=0,\qquad z^{n,m}|_{y=0}=\lim_{y\to\infty} z^{n,m}=0,
    \end{cases}
    \end{align}
    where 
    \begin{align*}
    &\mathcal{I}^1:=[\chi^2(\mathcal{E}_{w_n}(t))-\chi^2(\mathcal{E}_{w_m}(t))]\left[\mathcal{B}_1(w_n,\mathcal{R}_{n}w_n)+\mathcal{B}_1(w_n,u^s)+\mathcal{B}_1(u^s,\mathcal{R}_{n}w_n)\right], \\&
    \mathcal{I}^2:=[\chi^2(\mathcal{E}_{w_n}(t))-\chi^2(\mathcal{E}_{w_m}(t))]\mathcal{B}_2(\mathcal{R}_nw_n+u^s
    ,w_n+u^s),\\&\mathcal{I}^3:=\chi^2(\mathcal{E}_{w_m}(t))[\mathcal{B}_1(z^{n,m},\mathcal{R}_nw_n)+\mathcal{B}_1(w_{m},\mathcal{R}_nw_n-\mathcal{R}_mw_m)\\&\qquad\qquad+\mathcal{B}_1(z^{n,m},u^s)+\mathcal{B}_1(u^s,\mathcal{R}_nw_n-\mathcal{R}_mw_m)],\\&\mathcal{I}^4:=\chi^2(\mathcal{E}_{w_m}(t))[\mathcal{B}_2(\mathcal{R}_nw_n+u^s
    ,w_n+u^s)-\mathcal{B}_2(\mathcal{R}_mw_m+u^s
    ,w_m+u^s)],
    \end{align*}
    and 
    \begin{align}
        \JJ^1&:=[\chi^2(\mathcal{E}_{w_{n}}(t))-\chi^2(\mathcal{E}_{w_{m}}(t))][\F(\mathcal{R}_nw_n+u^s)-\psi \overline{\F}(U)],\notag\\
        \JJ^2&:=\chi^2(\mathcal{E}_{w_{m}}(t))[\F(\mathcal{R}_nw_n+u^s)-\F(\mathcal{R}_mw_{m}+u^s)].\notag
    \end{align}
    Applying Lemmas \ref{product_estimate1}, \ref{product_estimate2}, and \ref{APlemma1}, we get
    \begin{align}
        \label{LWP29}
        |\langle z^{n,m}_{\sigma},\mathcal{I}^1_{\sigma}\rangle_{X^s_{\gamma}}|&\lesssim_{s,\gamma_0}\mathcal{E}_{z^{n,m}}(t)\|\langle\nabla_x\rangle^{\frac{1}{2}}z_{\sigma}^{n,m}\|_{X^s_{\gamma}}\|\langle\nabla_x\rangle^{\frac{1}{2}}(w_n)_{\sigma}\|_{X^s_{\gamma}}\left(\|(w_n)_{\sigma}\|_{X^s_{\gamma}}+\|U_{\sigma}\|_{H^{s+1}_x}\right).
    \end{align}
    Similar to the proof of \eqref{AP13_1}, we have
    \begin{align}
        \label{LWP30}
        |\langle z^{n,m}_{\sigma},\mathcal{I}^2_{\sigma}\rangle_{X^s_{\gamma}}|&\lesssim_{s,\gamma_0} C_{\delta}\mathcal{E}_{z^{n,m}}(t)\|\langle\nabla_x\rangle^{\frac{1}{2}}z_{\sigma}^{n,m}\|_{X^s_{\gamma}}\|\langle\nabla_x\rangle^{\frac{1}{2}}(w_n)_{\sigma}\|_{X^s_{\gamma}}\|(w_n+u^s-U)_{\sigma}\|_{X^s_{\gamma}}\notag
        \\&\quad+C_{\delta}\mathcal{E}_{z^{n,m}}(t)\|\langle\nabla_x\rangle^{\frac{1}{2}}z_{\sigma}^{n,m}\|_{X^s_{\gamma}}\|(w_n)_{\sigma}\|_{X^s_{\gamma}}\|\partial_y(w_n+u^s)_{\sigma}\|_{X^s_{\gamma}}\notag\\&\quad+C_{\delta}\mathcal{E}_{z^{n,m}}(t)\|z_{\sigma}^{n,m}\|_{X^s_{\gamma}}\|U_{\sigma}\|_{H^{s+1}_x}\|\partial_y(w_n+u^s)_{\sigma}\|_{X^s_{\gamma}}\notag\\&\quad+\mathcal{E}_{z^{n,m}}(t)\|z_{\sigma}^{n,m}\|_{\tilde{X}^s_{\gamma}}\|U_{\sigma}\|_{H^{s+1}_{x}}\|\partial_y(w_n+u^s)_{\sigma}\|_{X^s_{\gamma}}.
    \end{align}
    We decompose $\mathcal{I}^3$ into:
    \begin{align*}    
    \mathcal{I}^3&=\chi^2(\mathcal{E}_{w_m}(t))[\mathcal{B}_1(z^{n,m},\mathcal{R}_nw_n)+\mathcal{B}_1(w_{m},\mathcal{R}_mz^{n,m})+\mathcal{B}_1(z^{n,m},u^s)+\mathcal{B}_1(u^s,\mathcal{R}_mz^{n,m})]\\
    &\quad+\chi^2(\mathcal{E}_{w_m}(t))[\mathcal{B}_1(w_{m},(\mathcal{R}_n-\mathcal{R}_m) w_{n})+\mathcal{B}_1(u^s ,(\mathcal{R}_n-\mathcal{R}_m) w_{n})]=:\mathcal{I}^{3,1}+\mathcal{I}^{3,2}.
    \end{align*}
    For $\mathcal{I}^{3,1}$, we have
    \begin{align}
        \label{LWP31}
        &|\langle z^{n,m}_{\sigma},\mathcal{I}^{3,1}_{\sigma}\rangle_{X^s_{\gamma}}|\notag\\&\qquad\lesssim_{s,\gamma_0} \|\langle\nabla_x\rangle^{\frac{1}{2}}z_{\sigma}^{n,m}\|_{X^s_{\gamma}}\|z^{n,m}_{\sigma}\|_{X^s_{\gamma}}\|\langle\nabla_x\rangle^{\frac{1}{2}}(w_n)_{\sigma}\|_{X^s_{\gamma}}+\|\langle\nabla_x\rangle^{\frac{1}{2}}z_{\sigma}^{n,m}\|_{X^s_{\gamma}}^2\|(w_m)_{\sigma}\|_{X^s_{\gamma}}\notag\\&\qquad\quad+\|z_{\sigma}^{n,m}\|_{X^s_{\gamma}}^2\|U_{\sigma}\|_{H^{s+1}_{x}}+\|\langle\nabla_x\rangle^{\frac{1}{2}}z_{\sigma}^{n,m}\|_{X^s_{\gamma}}^2\|U_{\sigma}\|_{H^{s+1}_{x}},
    \end{align}
    while for $\mathcal{I}^{3,2}$, we derive
    \begin{align}
        \label{LWP32}
        &|\langle z^{n,m}_{\sigma},\mathcal{I}^{3,2}_{\sigma}\rangle_{X^s_{\gamma}}|\notag\\&\qquad\lesssim_{s,\gamma_0} \|\langle\nabla_x\rangle^{\frac{1}{2}}z_{\sigma}^{n,m}\|_{X^s_{\gamma}}\|\langle\nabla_x\rangle^{\frac{1}{2}}(\mathcal{R}_n-\mathcal{R}_m) (w_{n})_{\sigma}\|_{X^s_{\gamma}}\left(\|(w_m)_{\sigma}\|_{X^s_{\gamma}}+\|U_{\sigma}\|_{H^{s+1}_{x}}\right).
    \end{align}
    To estimate $\mathcal{I}^4$, we rewrite it as
    \begin{align*}
       \mathcal{I}^4&=\chi^2(\mathcal{E}_{w_m}(t))[\mathcal{B}_2(\mathcal{R}_nw_n+u^s
    ,w_n+u^s)-\mathcal{B}_2(\mathcal{R}_mw_m+u^s
    ,w_m+u^s)]\\&=\chi^2(\mathcal{E}_{w_m}(t))\mathcal{B}_2((\mathcal{R}_n-\mathcal{R}_m)w_n,w_n+u^s)+\chi^2(\mathcal{E}_{w_m}(t))\mathcal{B}_2(\mathcal{R}_mz^{n,m}
    ,w_n+u^s)\\&\quad+\chi^2(\mathcal{E}_{w_m}(t))\mathcal{B}_2(\mathcal{R}_mw_m+u^s
    ,z^{n,m})\\&=:\mathcal{I}^{4,1}+\mathcal{I}^{4,2}+\mathcal{I}^{4,3}.
    \end{align*}
    On the other hand, we have
    \begin{align}
        \label{LWP33}
        |\langle z^{n,m}_{\sigma},\mathcal{I}^{4,1}_{\sigma}\rangle_{X^s_{\gamma}}|&\lesssim_{s ,\delta}\|\langle\nabla_x\rangle^{\frac{1}{2}}z_{\sigma}^{n,m}\|_{X^s_{\gamma}}\|\langle\nabla_x\rangle^{\frac{1}{2}}(\mathcal{R}_n-\mathcal{R}_m) (w_{n})_{\sigma}\|_{X^s_{\gamma}}\|(w_n+u^s-U)_{\sigma}\|_{X^s_{\gamma}}\notag\\&\quad +\|\langle\nabla_x\rangle^{\frac{1}{2}}z_{\sigma}^{n,m}\|_{X^s_{\gamma}}\|(\mathcal{R}_n-\mathcal{R}_m) (w_{n})_{\sigma}\|_{X^s_{\gamma}}\|\partial_y(w_n+u^s)_{\sigma}\|_{X^s_{\gamma}},
    \end{align}
    \begin{align}
        \label{LWP34}
        |\langle z^{n,m}_{\sigma},\mathcal{I}^{4,2}_{\sigma}\rangle_{X^s_{\gamma}}|&\lesssim_{s ,\delta}\|\langle\nabla_x\rangle^{\frac{1}{2}}z_{\sigma}^{n,m}\|_{X^s_{\gamma}}\|\langle\nabla_x\rangle^{\frac{1}{2}}\mathcal{R}_m z^{n,m}_{\sigma}\|_{X^s_{\gamma}}\|(w_n+u^s-U)_{\sigma}\|_{X^s_{\gamma}}\notag\\&\quad +\|\langle\nabla_x\rangle^{\frac{1}{2}}z_{\sigma}^{n,m}\|_{X^s_{\gamma}}\|\mathcal{R}_m z^{n,m}_{\sigma}\|_{X^s_{\gamma}}\|\partial_y(w_n+u^s)_{\sigma}\|_{X^s_{\gamma}},
    \end{align}
    and 
    \begin{align}
        \label{LWP35}
        |\langle z^{n,m}_{\sigma},\mathcal{I}^{4,3}_{\sigma}\rangle_{X^s_{\gamma}}|& \lesssim_{s} C_{\delta}\|\langle\nabla_x\rangle^{\frac{1}{2}}z_{\sigma}^{n,m}\|_{X^s_{\gamma}}\|\langle\nabla_x\rangle^{\frac{1}{2}}\mathcal{R}_m(w_m)_{\sigma}\|_{X^s_{\gamma}}\|z^{n,m}_{\sigma}\|_{X^s_{\gamma}}\notag
        \\&\quad+C_{\delta}\|\langle\nabla_x\rangle^{\frac{1}{2}}z_{\sigma}^{n,m}\|_{X^s_{\gamma}}\|\mathcal{R}_m(w_m)_{\sigma}\|_{X^s_{\gamma}}\|\partial_yz^{n,m}_{\sigma}\|_{X^s_{\gamma}}\notag\\&\quad+C_{\delta}\|z_{\sigma}^{n,m}\|_{X^s_{\gamma}}\|U_{\sigma}\|_{H^{s+1}_x}\|\partial_yz_{\sigma}^{n,m}\|_{X^s_{\gamma}}\notag\\&\quad+\|z_{\sigma}^{n,m}\|_{\widetilde{X}^s_{\gamma}}\|U_{\sigma}\|_{H^{s+1}_{x}}\|\partial_yz_{\sigma}^{n,m}\|_{X^s_{\gamma}}.
    \end{align}
    Combining the estimates \eqref{LWP29}--\eqref{LWP35}, we get
    \begin{align}\label{LWP36}
    &\notag\mathbb{E}\int_{\tau_1}^{\tau_2}\langle z^{n,m}_{\sigma}, (\mathcal{I}^{1}+\mathcal{I}^{2}+\mathcal{I}^{3}+\mathcal{I}^{4})_{\sigma}\rangle_{X^{s}_{\gamma}}\mathrm{d}t\\&\quad\lesssim_{s,\gamma_0,\MM ,\NN } C_{\delta}\mathbb{E}\int_{\tau_1}^{\tau_2}\mathcal{E}^2_{z^{n,m}}(t)\mathrm{d}t+C_{\delta}\mathbb{E}\int_{\tau_1}^{\tau_2}\|\langle\nabla_x\rangle^{\frac{1}{2}}z_{\sigma}^{n,m}\|^2_{X^s_{\gamma}}\mathrm{d}t\notag\\&\qquad+C_{\delta}\mathbb{E}\int_{\tau_1}^{\tau_2}\|\langle\nabla_x\rangle^{\frac{1}{2}}z_{\sigma}^{n,m}\|_{X^s_{\gamma}}\|\partial_yz_{\sigma}^{n,m}\|_{X^s_{\gamma}}\mathrm{d}t\notag\\&\qquad+C_{\delta}\mathbb{E}\int_{\tau_1}^{\tau_2}\mathcal{E}_{z^{n,m}}(t)\|\langle\nabla_x\rangle^{\frac{1}{2}}z_{\sigma}^{n,m}\|_{X^s_{\gamma}}\left(\|\langle\nabla_x\rangle^{\frac{1}{2}}(w_n)_{\sigma}\|_{X^s_{\gamma}}+\|\langle\nabla_x\rangle^{\frac{1}{2}}(w_m)_{\sigma}\|_{X^s_{\gamma}}\right)\mathrm{d}t\notag\\&\qquad+C_{\delta}\mathbb{E}\int_{\tau_1}^{\tau_2}\mathcal{E}_{z^{n,m}}(t)\|\langle\nabla_x\rangle^{\frac{1}{2}}z_{\sigma}^{n,m}\|_{X^s_{\gamma}}\left(\|\partial_y(w_n)_{\sigma}\|_{X^s_{\gamma}}+1\right)\mathrm{d}t\notag\\\notag&\qquad +C_{\delta}\mathbb{E}\int_{\tau_1}^{\tau_2}\|\langle\nabla_x\rangle^{\frac{1}{2}}z_{\sigma}^{n,m}\|_{X^s_{\gamma}}\|\langle\nabla_x\rangle^{\frac{1}{2}}(\mathcal{R}_n-\mathcal{R}_m)(w_n)_{\sigma}\|_{X^s_{\gamma}}\mathrm{d}t\\\notag&\qquad +C_{\delta}\mathbb{E}\int_{\tau_1}^{\tau_2}\|\langle\nabla_x\rangle^{\frac{1}{2}}z_{\sigma}^{n,m}\|_{X^s_{\gamma}}\|(\mathcal{R}_n-\mathcal{R}_m)(w_n)_{\sigma}\|_{X^s_{\gamma}}\|\partial_y(w_n)_{\sigma}\|_{X^s_{\gamma}}\mathrm{d}t\\&\qquad+\mathbb{E}\int_{\tau_1}^{\tau_2}\frac{1}{\langle t\rangle^{\delta}}\|z_{\sigma}^{n,m}\|_{\widetilde{X}^s_{\gamma}}\left(\|\partial_yz_{\sigma}^{n,m}\|_{X^s_{\gamma}}+\left(1+\|\partial_y(w_n)_{\sigma}\|_{X^s_{\gamma}}\right)\mathcal{E}_{z^{n,m}}(t)\right)\mathrm{d}t
    \end{align}
    for any stopping times $\tau_1\le \tau_2\le T_*\wedge \Lambda^{n,m}_{\MM }$. We turn to the estimate of the force terms given in \eqref{LWP28}. To this end, we use Lemma \ref{APlemma3} to get
    \begin{align}
        \label{LWP37}
        \mathbb{E}\int_{\tau_1}^{\tau_2}\|\JJ^1_{\sigma}\|^2_{\mathbb{X}^{s}_{\gamma}}\mathrm{d}t\lesssim_{s,\gamma_0,\MM ,\NN } \mathbb{E}\int_{\tau_1}^{\tau_2} \mathcal{E}^2_{z^{n,m}}(t)\left(1+\||\nabla_x|^{\frac{1}{2}}(w_n)_{\sigma}\|^2_{X^s_{\gamma}}\right)\mathrm{d}t.
    \end{align}
    This, together with the Burkholder--Davis--Gundy inequality \eqref{pre6}, yields
    \begin{align}
        \label{LWP38}
        &\mathbb{E}\int_{\tau_1}^{\tau_2}\langle z^{n,m}_{\sigma},\JJ^1_{\sigma} \mathrm{d}W\rangle_{X^{s}_{\gamma}}\notag\\&\qquad\le \epsilon \mathbb{E}\mathcal{E}^2_{z^{n,m}}(\tau_2)+C_{s,\gamma_0,\MM ,\NN }\mathbb{E}\int_{\tau_1}^{\tau_2} \mathcal{E}^2_{z^{n,m}}(t)\left(1+\||\nabla_x|^{\frac{1}{2}}(w_n)_{\sigma}\|^2_{X^s_{\gamma}}\right)\mathrm{d}t,
    \end{align}
    where $\epsilon>0$ is a parameter to be determined later in \eqref{LWP44_1}. Similarly, by Proposition \ref{Proposition_struc1}, 
    \begin{align}
        \label{LWP39}
        &\mathbb{E}\int_{\tau_1}^{\tau_2}\|\JJ^2_{\sigma}\|^2_{\mathbb{X}^{s}_{\gamma}}\mathrm{d}t\notag\\&\qquad\lesssim_{s,\gamma_0,\MM ,\NN } \mathbb{E}\int_{\tau_1}^{\tau_2} \left(\|\langle\nabla_x\rangle^{\frac{1}{2}}(\mathcal{R}_n-\mathcal{R}_m)(w_n)_{\sigma}\|_{X^s_\gamma}^2+\|\langle\nabla_x\rangle^{\frac{1}{2}}z^{n,m}_{\sigma}\|_{X^s_\gamma}^2\right)\mathrm{d}t,
    \end{align}
    which, together with the Burkholder--Davis--Gundy inequality \eqref{pre6}, implies
    \begin{align}
        \label{LWP40}
        &\mathbb{E}\int_{\tau_1}^{\tau_2}\langle z^{n,m}_{\sigma},\JJ^2_{\sigma} \mathrm{d}W\rangle_{X^{s}_{\gamma}}\notag\\&\qquad \le \epsilon \mathbb{E}\mathcal{E}^2_{z^{n,m}}(\tau_2)+C_{s,\gamma_0,\MM ,\NN }\mathbb{E}\int_{\tau_1}^{\tau_2} \left(\|\langle\nabla_x\rangle^{\frac{1}{2}}(\mathcal{R}_n-\mathcal{R}_m)(w_n)_{\sigma}\|_{X^s_\gamma}^2+\|\langle\nabla_x\rangle^{\frac{1}{2}}z^{n,m}_{\sigma}\|_{X^s_\gamma}^2\right)\mathrm{d}t.
    \end{align}
    For the terms involving the operator $\mathcal{R}_n-\mathcal{R}_m$ given in \eqref{LWP36} and \eqref{LWP39}, we have:
    \begin{align}
        \label{LWP41}
        \|\langle\nabla_x\rangle^{\frac{1}{2}}(\mathcal{R}_n-\mathcal{R}_m)(w_n)_{\sigma}\|_{X^s_\gamma}\le \frac{1}{n}\|\langle\nabla_x\rangle^{\frac{3}{2}}(w_n)_{\sigma}\|_{X^s_\gamma},
    \end{align}
    and 
    \begin{align}
        \label{LWP42}
        \|(\mathcal{R}_n-\mathcal{R}_m)(w_n)_{\sigma}\|_{X^s_\gamma}\le \frac{1}{n}\|\langle\nabla_x\rangle(w_n)_{\sigma}\|_{X^s_\gamma}
    \end{align}
    for any $m\ge n\ge N$. Combining the estimates \eqref{LWP36}--\eqref{LWP42} with the induction-based argument presented in Section \ref{AP}, we deduce
    \begin{align}
        \label{LWP43}
        &\mathbb{E}\sup_{\tau_1\le t\le \tau_2}\|z^{n,m}_{\sigma}\|^2_{X^{s}_{\gamma}}+2\mathbb{E}\int_{\tau_1}^{\tau_2}\left(\lambda\||\nabla_x|^{\frac{1}{2}}z^{n,m}_{\sigma}\|^2_{X^s_{\gamma}}+\frac{1}{2}\|\partial_yz^{n,m}_{\sigma}\|^2_{X^s_{\gamma}}+\frac{\gamma_0\delta}{\langle t\rangle^{\delta+1}}\|z^{n,m}_{\sigma}\|^2_{\widetilde{X}^s_{\gamma}}\right)\mathrm{d}t\notag
        \\&\quad \le C\mathbb{E}\|z^{n,m}_{\sigma}(\tau_1)\|^2_{X^s_{\gamma}}+C_{s,\gamma_0,\MM ,\NN ,\delta}\mathbb{E}\int_{\tau_1}^{\tau_2}\left(\mathcal{E}^2_{z^{n,m}}(t)+\||\nabla_x|^{\frac{1}{2}}z^{n,m}_{\sigma}\|^2_{X^s_{\gamma}}\right)\mathrm{d}t\notag\\&\qquad+\frac{1}{16}\mathbb{E}\int_{\tau_1}^{\tau_2}\|\partial_yz^{n,m}_{\sigma}\|^2_{X^s_{\gamma}}\mathrm{d}t+\frac{C_{s,\gamma_0,\MM ,\NN ,\delta}}{n^2}\mathbb{E}\int_{\tau_1}^{\tau_2}\|\langle\nabla_x\rangle^{\frac{3}{2}}(w_n)_{\sigma}\|^2_{X^s_\gamma}\mathrm{d}t\notag
        \\&\qquad+\frac{C_{s,\gamma_0,\MM ,\NN ,\delta}}{n^2}\mathbb{E}\int_{\tau_1}^{\tau_2}\|\langle\nabla_x\rangle(w_n)_{\sigma}\|_{X^s_\gamma}^2\|\partial_y(w_n)_{\sigma} \|^2_{X^s_{\gamma}}\mathrm{d}t\notag\\&\qquad+\epsilon C_{s,\gamma_0,\MM ,\NN ,\delta}\mathbb{E}\int_{\tau_1}^{\tau_2}\mathcal{E}^2_{z^{n,m}}(t)\left(\||\nabla_x|^{\frac{1}{2}} (w_n)_{\sigma}\|^2_{X^s_{\gamma}}+\||\nabla_x|^{\frac{1}{2}} (w_m)_{\sigma}\|^2_{X^s_{\gamma}}\right)\mathrm{d}t\notag
        \\&\qquad+\epsilon C_{s,\gamma_0,\MM ,\NN ,\delta}\mathbb{E}\int_{\tau_1}^{\tau_2}\mathcal{E}^2_{z^{n,m}}(t)\left(\|\partial_y (w_n)_{\sigma}\|^2_{X^s_{\gamma}}+1\right)\mathrm{d}t\notag
        \\&\qquad+C_{s,\gamma_0,\MM ,\NN }\mathbb{E}\int_{\tau_1}^{\tau_2}\frac{1}{\langle t\rangle^{2\delta}}\| z^{n,m}_{\sigma}\|^2_{\widetilde{X}^s_{\gamma}}\mathrm{d}t+\frac{1}{72\MM ^2}\mathbb{E}\int_{\tau_1}^{\tau_2}\mathcal{E}^2_{z^{n,m}}(t)\|\partial_y(w_n)_{\sigma} \|^2_{X^s_{\gamma}}\mathrm{d}t\notag\\&\qquad +2\epsilon\mathbb{E}\mathcal{E}^2_{z^{n,m}}(\tau_2)+C_{s,\gamma_0,\MM ,\NN }\mathbb{E}\int_{\tau_1}^{\tau_2}\mathcal{E}^2_{z^{n,m}}(t)\||\nabla_x|^{\frac{1}{2}}(w_n)_{\sigma}\|^2_{X^s_{\gamma}}\mathrm{d}t.
    \end{align}
    Let us take
    \begin{align}
        \label{LWP44}
        \delta>1\vee\frac{C_{s,\gamma_0,\MM ,\NN }}{2\gamma_0},\qquad 2\lambda\ge C_{s,\gamma_0,\MM ,\NN ,\delta}+1,
    \end{align}
    and choose $\epsilon$ sufficiently small so that
    \begin{align}
        \label{LWP44_1}
        \epsilon C_{s,\gamma_0,\MM ,\NN ,\delta}\mathbb{E}&\int_{\tau_1}^{\tau_2}\mathcal{E}^2_{z^{n,m}}(t)\left(\||\nabla_x|^{\frac{1}{2}} (w_n)_{\sigma}\|^2_{X^s_{\gamma}}+\||\nabla_x|^{\frac{1}{2}} (w_m)_{\sigma}\|^2_{X^s_{\gamma}}+\|\partial_y (w_n)_{\sigma}\|^2_{X^s_{\gamma}}+1\right)\notag\\&\qquad\qquad\quad\le \epsilon C_{s,\gamma_0,\MM ,\NN ,\delta}\mathbb{E}\mathcal{E}^2_{z^{n,m}}(\tau_2)\le \frac{1}{16}\mathbb{E}\mathcal{E}^2_{z^{n,m}}(\tau_2).
    \end{align}
    This, together with \eqref{LWP27} and \eqref{LWP43}, implies 
    \begin{align}
        \label{LWP45}
        \mathbb{E}&\sup_{\tau_1\le t\le \tau_2}\|z^{n,m}_{\sigma}\|^2_{X^{s}_{\gamma}}+\mathbb{E}\int_{\tau_1}^{\tau_2}\left(\||\nabla_x|^{\frac{1}{2}}z^{n,m}_{\sigma}\|^2_{X^s_{\gamma}}+\|\partial_yz^{n,m}_{\sigma}\|^2_{X^s_{\gamma}}\right)\mathrm{d}t\notag\\&\qquad\qquad\qquad\quad\le C_{s,\gamma_0,\MM ,\NN }\mathbb{E}\|z^{n,m}_{\sigma}(\tau_1)\|^2_{X^s_{\gamma}}+\frac{1}{4}\mathbb{E}\mathcal{E}_{z^{n,m}}^2(\tau_2)+\frac{C_{s,\gamma_0,\MM ,\NN }}{n^2}\notag\\&\qquad\qquad\qquad\qquad +C_{s,\gamma_0,\MM ,\NN }\mathbb{E}\int_{\tau_1}^{\tau_2}\mathcal{E}^2_{z^{n,m}}(t)\left(1+\||\nabla_x|^{\frac{1}{2}}(w_n)_{\sigma}\|^2_{X^s_{\gamma}}\right)\mathrm{d}t.
        \end{align}
        To complete the estimate for $z^{n,m}$, we use a stochastic Gronwall-type argument. First, by definition \eqref{deflambda}, we have
        \[\int_0^{T_*\wedge\Lambda^{n,m}_{\MM }}\left(1+\||\nabla_x|^{\frac{1}{2}}(w_n)_{\sigma}\|^2_{X^s_{\gamma}}\right)\mathrm{d}t\le 9\MM ^2+1.\]
        Then, there exist a deterministic parameter $N_0$ and a sequence of stopping times $0=t_0<t_1<\cdots<t_{N_0}=T_*\wedge\Lambda^{n,m}_{\MM }$ satisfying 
        \begin{align}\label{LWP46}
            \int_{t_j}^{t_{j+1}}\left(1+\||\nabla_x|^{\frac{1}{2}}(w_n)_{\sigma}\|^2_{X^s_{\gamma}}\right)\mathrm{d}t\le \frac{1}{4C_{s,\gamma_0,\MM ,\NN }},\qquad \forall  j\le {N_0-1}.
        \end{align}
        Combining this with the estimate \eqref{LWP45} restricted on $[t_0,t_1]$, we obtain
        \begin{align*}
        \mathbb{E}&\sup_{0\le t\le t_1}\|z^{n,m}_{\sigma}\|^2_{X^{s}_{\gamma}}+\mathbb{E}\int_{0}^{t_1}\left(\||\nabla_x|^{\frac{1}{2}}z^{n,m}_{\sigma}\|^2_{X^s_{\gamma}}+\|\partial_yz^{n,m}_{\sigma}\|^2_{X^s_{\gamma}}\right)\mathrm{d}t\le \frac{C_{s,\gamma_0,\MM ,\NN }}{n^2}.
        \end{align*} 
        By induction, we assume that
        \begin{align}
        \label{LWP48}
        \mathbb{E}&\sup_{0\le t\le t_j}\|z^{n,m}_{\sigma}\|^2_{X^{s}_{\gamma}}+\mathbb{E}\int_{0}^{t_j}\left(\||\nabla_x|^{\frac{1}{2}}z^{n,m}_{\sigma}\|^2_{X^s_{\gamma}}+\|\partial_yz^{n,m}_{\sigma}\|^2_{X^s_{\gamma}}\right)\mathrm{d}t\le \frac{C_{s,\gamma_0,\MM ,\NN ,j}}{n^2}
        \end{align}
        for some $j>1$ and $C_{s,\gamma_0,\MM ,\NN ,j}>0$. Then, for $j+1$, by applying \eqref{LWP46} and \eqref{LWP48}, we get
        \begin{align*}
        \mathbb{E}&\sup_{0\le t\le t_{j+1}}\|z^{n,m}_{\sigma}\|^2_{X^{s}_{\gamma}}+\mathbb{E}\int_{0}^{t_{j+1}}\left(\||\nabla_x|^{\frac{1}{2}}z^{n,m}_{\sigma}\|^2_{X^s_{\gamma}}+\|\partial_yz^{n,m}_{\sigma}\|^2_{X^s_{\gamma}}\right)\mathrm{d}t\\&\qquad\le \frac{C_{s,\gamma_0,\MM ,\NN ,j}}{n^2}+\mathbb{E}\sup_{t_j\le t\le t_{j+1}}\|z^{n,m}_{\sigma}\|^2_{X^{s}_{\gamma}}+\mathbb{E}\int_{t_j}^{t_{j+1}}\left(\||\nabla_x|^{\frac{1}{2}}z^{n,m}_{\sigma}\|^2_{X^s_{\gamma}}+\|\partial_yz^{n,m}_{\sigma}\|^2_{X^s_{\gamma}}\right)\mathrm{d}t\\&\qquad\le \frac{C_{s,\gamma_0,\MM ,\NN ,j}}{n^2}+C_{s,\gamma_0,\MM ,\NN }\mathbb{E}\|z^{n,m}_{\sigma}(t_j)\|^2_{X^s_{\gamma}}+\frac{1}{2}\mathbb{E}\mathcal{E}_{z^{n,m}}^2(t_{j+1})+\frac{C_{s,\gamma_0,\MM ,\NN }}{n^2}\\&\qquad\le \frac{(C_{s,\gamma_0,\MM ,\NN ,j}+1)(C_{s,\gamma_0,\MM ,\NN }+1)}{n^2}+\frac{1}{2}\mathbb{E}\mathcal{E}_{z^{n,m}}^2(t_{j+1}).
        \end{align*}
        Therefore, by setting 
        \[C_{s,\gamma_0,\MM ,\NN ,j+1}:=2(C_{s,\gamma_0,\MM ,\NN ,j}+1)(C_{s,\gamma_0,\MM ,\NN }+1)\]
        and $C_{s,\gamma_0,\MM ,\NN ,1}:=C_{s,\gamma_0,\MM ,\NN }$, we see that \eqref{LWP48} holds for any $j\le N_0$. In particular, 
        \begin{align}\label{LWP49}
        \mathbb{E}&\sup_{0\le t\le T_*\wedge \Lambda^{n,m}_{\MM }}\|z^{n,m}_{\sigma}\|^2_{X^{s}_{\gamma}}+\mathbb{E}\int_{0}^{T_*\wedge \Lambda^{n,m}_{\MM }}\left(\||\nabla_x|^{\frac{1}{2}}z^{n,m}_{\sigma}\|^2_{X^s_{\gamma}}+\|\partial_yz^{n,m}_{\sigma}\|^2_{X^s_{\gamma}}\right)\mathrm{d}t\lesssim_{s,\gamma_0,\MM ,\NN } \frac{1}{n^2}.
        \end{align}
        This implies 
        \begin{align*}
            \sup_{m\ge n}\mathbb{E}\mathcal{E}_{w_n-w_m}(T_*\wedge \Lambda^{n,m}_{\MM })\lesssim_{s,\gamma_0,\MM ,\NN } \frac{1}{n}\to0,\qquad n\to\infty.
        \end{align*}
        
        \noindent {\textbf{Step I\!I. Verification of} \eqref{LWP26}}. As $w_n$ satisfies the equation \eqref{LWP1}, by Proposition \ref{AP_proposition3}, we infer
        \begin{align}
            \PP \left\{\mathcal{E}_{w_n}(t\wedge T_*\wedge\Lambda^{n}_{\MM })\ge \mathcal{E}_{w_n}(0)+\MM \right\}&\le \PP \left\{\mathcal{E}^2_{w_n}(t\wedge T_*\wedge \Lambda^{n}_{\MM })\ge \mathcal{E}_{w_n}^2(0)+\MM ^2\right\}\notag\\
            &\le \frac{1}{\MM ^2}\mathbb{E}\left(\mathcal{E}^2_{w_n}(t\wedge T_*\wedge \Lambda^{n}_{\MM })- \mathcal{E}_{w_n}^2(0)\right)\notag\\&\lesssim_{s,\gamma_0,\MM ,\NN } \sqrt{t}+t.\notag
        \end{align}
       This ensures the condition \eqref{LWP26}. 
        
        \noindent {\textbf{Step I\!I\!I. Constructing solutions of} \eqref{AP1}}. We apply Lemma \ref{abstract_cauchy_lemma} to extract a convergent subsequence $\{w_{n_l}\}_{l=1}^\infty$ from the approximate solutions $\{w_n\}_{n=1}^{\infty}$ given by \eqref{LWP1}, thereby obtaining a local pathwise solution $w$ of \eqref{AP1}. The pathwise uniqueness can be established similarly, and thus we omit the details. This completes the proof of Proposition \ref{LWP_proposition3}.
    \end{proof}
    
    \subsection{Proof of Theorem \ref{MT1}}\label{proof_LWP}
    It remains to remove the restrictions given in \eqref{APassumption} and \eqref{LWP50} respectively. Assume that  
    \begin{align}\label{LWP50_1}
        \|\langle\nabla_x\rangle w_0\|_{X^s_{\gamma_0,\sigma_0}}<\infty,\qquad\|U_0\|_{H^{s+2}_{x,\sigma_0}}<\infty,\qquad \|\nabla_x P|_{t=0}\|_{H^{s+1}_{x,\sigma_0}}<\infty,\qquad\PP\mbox{-a.s..}
    \end{align}
    Define
    \begin{align*}
        \mathcal{A}_{\MM ,\NN _0,\NN _1}&:=\left\{\|\langle\nabla_x\rangle w_0\|_{X^s_{\gamma_0,\sigma_0}}\in [\MM -1,\MM )\right\}\notag\\&\qquad\bigcap \left\{\|U_0\|_{H^{s+2}_{x,\sigma_0}}\in [\NN _0-1,\NN _0)\right\}\bigcap \left\{\|\nabla_x P|_{t=0}\|_{H^{s+1}_{x,\sigma_0}}\in[\NN _1-1,\NN _1)\right\},
    \end{align*}
    where $\MM,\NN_0,\NN_1$ are positive integers. By Proposition \ref{LWP_proposition3}, there is a local pathwise solution $(u_{\MM ,\NN _0,\NN _1},\tau_{\MM ,\NN _0,\NN _1})$ of the stochastic Prandtl equation \eqref{Intro1} with the data
    \[w_0^{\MM ,\NN _0,\NN _1}:=1_{\mathcal{A}_{\MM ,\NN _0,\NN _1}}w_0,\quad U_0^{\MM ,\NN _0,\NN _1}:=1_{\mathcal{A}_{\MM ,\NN _0,\NN _1}}U_0,\quad \nabla_x P_0^{\MM ,\NN _0,\NN _1}:=1_{\mathcal{A}_{\MM ,\NN _0,\NN _1}}\nabla_x P_0.\]
    Let $\lambda^{\MM ,\NN _0,\NN _1},\delta^{\MM ,\NN _0,\NN _1}>0$ be the parameters indicated by Proposition \ref{LWP_proposition3}. Define
    \[u:=\sum_{\MM ,\NN _0,\NN _1\in\mathbb{N}}1_{\mathcal{A}_{\MM ,\NN _0,\NN _1}}u^{\MM ,\NN _0,\NN _1},\qquad \tau:=\sum_{\MM ,\NN _0,\NN _1\in\mathbb{N}}1_{\mathcal{A}_{\MM ,\NN _0,\NN _1}}\tau^{\MM ,\NN _0,\NN _1}\]
    and
    \[\lambda:=\sum_{\MM ,\NN _0,\NN _1\in\mathbb{N}}1_{\mathcal{A}_{\MM ,\NN _0,\NN _1}}\lambda^{\MM ,\NN _0,\NN _1},\qquad \delta:=\sum_{\MM ,\NN _0,\NN _1\in\mathbb{N}}1_{\mathcal{A}_{\MM ,\NN _0,\NN _1}}\delta^{\MM ,\NN _0,\NN _1}.\]
    As
    \[\PP \left(\bigcup_{\MM ,\NN _0,\NN _1\in\mathbb{N}}\mathcal{A}_{\MM ,\NN _0,\NN _1}\right)=1,\]
    we see that
    \[\PP \{\tau>0\}=1,\]
    and $(u,\tau)$ is a local pathwise solution of \eqref{Intro1}, and $\lambda,\delta$ are $\mathscr{F}_0$-measurable random parameters. Since
    \[\|\langle\nabla_x\rangle w_0\|_{X^s_{\gamma_0,\sigma_0}}\lesssim_{\sigma_0}\|w_0\|_{X^s_{\gamma_0,2\sigma_0}},\]
    we conclude that the stochastic Prandtl equation \eqref{Intro1} is locally solvable for initial data $w_0$ valued in $X^s_{\gamma_0,2\sigma_0}$. This completes the proof of Theorem \ref{MT1}.

    \section{Probabilistic global existence of Gevrey solutions for the problem with tangential random diffusion}\label{GWP}
    In this section, we investigate the probabilistic global existence for the stochastic Prandtl equation \eqref{Intro3} with tangential random diffusion.
    
    \subsection{Probabilistic global existence of Gevrey-2 solutions in 3D}
    Let
    \[w:=u-u^s,\qquad u^s(t,y):=\Psi(y) U(t),\]
    with $\Psi$ defined in \eqref{MR3_1}. Then, from \eqref{Intro3}, we see that $w$ satisfies the following equation:
    \begin{align}
    \label{GWP0}
    \begin{cases}
        \mathrm{d}w-\partial_y^2w\mathrm{d}t+\left[(w+u^s)\nabla_x w-\partial_y^{-1}(\nabla_x\cdot w)\partial_y(w+u^s)\right]\mathrm{d} t=(\alpha_1+\alpha|\nabla_x|^{\frac{1}{2}})w\mathrm{d}B_t,\\
        w|_{y=0}=\lim_{y\to\infty} w=0,\\
        w(0)=w_0:=u_0-U_0\Psi.
    \end{cases}
    \end{align}
    The local wellposedness of the above equation in analytic class can be similarly established as in Section \ref{LWP}. In particular, since there is no linearly increasing term, e.g. $\partial_y^{-1}\nabla_x u^s \partial_yw$ in the equation, the local wellposedness holds in the space $\XX^{s,1}_{\gamma,\sigma}$ with polynomial weights. The main goal of this subsection is to obtain the probabilistic global existence of Gevrey-2 solutions of equation \eqref{GWP0}, as stated in Theorem \ref{MT2}.
    
    \begin{proof}[Proof of Theorem \ref{MT2}] The proof is divided into the following three steps. Without loss of generality, we take $\UU_0=1$. 

    \noindent {\textbf{Step I. Reduction to a PDE with random parameter}}.
    Applying the random transformation $\Gamma_t$ defined in \eqref{Intro3_1}, we infer that $w_{\Gamma}:= \Gamma_t w$ satisfies the following equation:
    \begin{align}\label{GWP1}
    \begin{cases}
    \partial_t w_{\Gamma}+\beta w_{\Gamma}+\left((\alpha\alpha_1-\lambda)|\nabla_{x}|^{\frac{1}{2}}+\frac{\alpha^2}{2}|\nabla_{x}|\right)w_{\Gamma}\\\qquad\qquad\qquad+\left((w+u^s)\cdot\nabla_x w\right)_{\Gamma}-\left(\partial_y^{-1}(\nabla_x\cdot w)\partial_y(w+u^s)\right)_{\Gamma}=\partial_y^2w_{\Gamma},\\
    w_{\Gamma}(0)=e^{\sigma_0|\nabla_{x}|^{\frac{1}{2}}}w_0,\qquad w_{\Gamma}|_{y=0}=0,\qquad \lim_{y\to\infty}w_{\Gamma}=0.
    \end{cases}
    \end{align}
    To obtain global existence for the above Prandtl-type equation, we introduce the stopping time 
    \[\TT_*:=\inf\left\{t\ge0\bigg| \alpha_1B_t\ge R+\left(\frac{\alpha_1^2}{2}-\beta\right)t\right\}\wedge\inf\left\{t\ge0\bigg| \alpha B_t\ge \frac{\sigma_0+\lambda t}{2}\right\},\]
    where $R>0$ is a parameter to be determined later. Notice that
    \[\UU(t)=\exp\left(\alpha_1B_t+\left(\beta-\frac{\alpha_1^2}{2}\right)t\right).\]
    Then, for any $t\le \TT^*$,
   \begin{align}\label{GWP2}
       \UU(t)\le e^R,\qquad \sigma_0+\lambda t-\alpha B_t\ge \frac{\sigma_0+\lambda t}{2}>0.
   \end{align}
   Moreover, the following probabilistic estimate for $\TT_*$ is available:
   \begin{align}
        \label{GWP10}
        \PP \{\TT_*=\infty\}\ge 1-\exp\left(-R\left(1-\frac{2\beta}{\alpha_1^2}\right)\right)-\exp\left(-\frac{\lambda\sigma_0}{2\alpha^2}\right),
    \end{align}
    cf. Chapter 3 in \cite{KS91} for its proof.
    
    \noindent {\textbf{Step I\!I. Energy estimate for \eqref{GWP1} before} $\TT_*$}. Taking $\partial_y$ on both sides of \eqref{GWP1} gives 
    \begin{align}\label{GWP2_1}
    \partial_t \partial_yw_{\Gamma}+\beta \partial_yw_{\Gamma}+&\left((\alpha\alpha_1-\lambda)|\nabla_{x}|^{\frac{1}{2}}+\frac{\alpha^2}{2}|\nabla_{x}|\right)\partial_yw_{\Gamma}\notag\\&+\left((w+u^s)\cdot\nabla_x \partial_y w\right)_{\Gamma}-\left(\partial_y^{-1}(\nabla_x\cdot w)\partial^2_y(w+u^s)\right)_{\Gamma}=\partial_y^3w_{\Gamma}.
    \end{align}
    To close the energy estimate in the space $\XX^{s,2}_\gamma$, it suffices to bound the convection terms in \eqref{GWP1} and \eqref{GWP2_1}. First, we note that 
    \begin{align}\label{GWP3}
        |\langle w_\Gamma,\left((w+u^s)\cdot\nabla_x w\right)_{\Gamma}\rangle_{H^{s,0}_\gamma}|&\le \|w_\Gamma\|_{H^{s+1/2,0}_\gamma}\left\|\|\left(w\cdot\nabla_x w\right)_{\Gamma}\|_{H^{s-1/2}_x}\langle y\rangle^\gamma\right\|_{L^2_y}\notag\\&\quad+\|w_\Gamma\|_{H^{s+1/2,0}_\gamma}\left\|\|\nabla_x w_{\Gamma}\|_{H^{s-1/2}_x}\langle y\rangle^\gamma u^s\right\|_{L^2_y}\notag\\&\lesssim_s \|w_\Gamma\|_{H^{s+1/2,0}_\gamma}^2\left(\left\|\|w_\Gamma\|_{H^{s-1/2}_x}\right\|_{L^\infty_y}+\UU(t)\right)\notag\\&\lesssim_s \|\langle\nabla_x\rangle^{\frac{1}{2}}w_\Gamma\|_{\XX^{s,2}_\gamma}^2\left(\|w_\Gamma\|_{\XX^{s,2}_\gamma}+\UU(t)\right).
    \end{align}
    Similarly, we have 
    \begin{align}\label{GWP4}
     |\langle \partial_yw_\Gamma,\left((w+u^s)\cdot\nabla_x \partial_yw\right)_{\Gamma}\rangle_{H^{s-1/2,0}_\gamma}|\lesssim_s \|\langle\nabla_x\rangle^{\frac{1}{2}}w_\Gamma\|_{\XX^{s,2}_\gamma}^2\left(\|w_\Gamma\|_{\XX^{s,2}_\gamma}+\UU(t)\right).
    \end{align}
    Next, we infer that
    \begin{align}
        \label{GWP5}
        \notag|\langle w_\Gamma,&\left(\partial_y^{-1}(\nabla_x\cdot w)\partial_y(w+u^s)\right)_{\Gamma}\rangle_{H^{s,0}_\gamma}|\\&\qquad\qquad\qquad\qquad\notag\le \|w_\Gamma\|_{H^{s+1/2,0}_\gamma}\left\|\left\|\left(\partial_y^{-1}(\nabla_x \cdot w)\partial_yw\right)_{\Gamma}\right\|_{H^{s-1/2}_x}\langle y\rangle^\gamma\right\|_{L^2_y}\\&\qquad\qquad\qquad\qquad\quad+\||\nabla_x|^{\frac{1}{2}}w_\Gamma\|_{H^{s,0}_\gamma}\left\|\|\partial_y^{-1}(|\nabla_x|^{\frac{1}{2}}w)_{\Gamma}\|_{H^{s}_x}\langle y\rangle^\gamma \partial_yu^s\right\|_{L^2_y}\notag\\&\qquad\qquad\qquad\qquad\lesssim_{s,\gamma}\|w_\Gamma\|_{H^{s+1/2,0}_\gamma}^2\|\partial_y w_\Gamma\|_{H^{s-1/2,0}_\gamma}+\||\nabla_x|^{\frac{1}{2}}w_\Gamma\|_{H^{s,0}_\gamma}^2\|\langle y\rangle^\gamma\partial_yu^s\|_{L^2_y}\notag\\&\qquad\qquad\qquad\qquad\lesssim_{s,\gamma}\|\langle\nabla_x\rangle^{\frac{1}{2}}w_\Gamma\|_{\XX^{s,2}_\gamma}^2\|w_\Gamma\|_{\XX^{s,2}_\gamma}+\||\nabla_x|^{\frac{1}{2}}w_\Gamma\|_{\XX^{s,2}_\gamma}^2\beta^{\frac{1}{4}}\UU(t),
    \end{align}
    and 
    \begin{align}
        \label{GWP6}
        |\langle \partial_y&w_\Gamma,\left(\partial_y^{-1}(\nabla_x\cdot w)\partial^2_y(w+u^s)\right)_{\Gamma}\rangle_{H^{s-1/2,0}_\gamma}|\notag\\&\lesssim_{s,\gamma}\|\partial_yw_\Gamma\|_{H^{s,0}_\gamma}\|w_\Gamma\|_{H^{s,0}_\gamma}\|\partial_y^2w_\Gamma\|_{H^{s-1,0}_\gamma}+\|\partial_y|\nabla_x|^{\frac{1}{2}}w_\Gamma\|_{H^{s-1/2,0}_\gamma}\|w_\Gamma\|_{H^{s,0}_\gamma}\beta^{\frac{3}{4}}\UU(t)\notag\\&\lesssim_{s,\gamma}\|\langle\nabla_x\rangle^{\frac{1}{2}}w_\Gamma\|_{\XX^{s,2}_\gamma}\|w_\Gamma\|_{\XX^{s,2}_\gamma}\|\partial_yw_\Gamma\|_{\XX^{s,2}_\gamma}+\||\nabla_x|^{\frac{1}{2}}w_\Gamma\|_{\XX^{s,2}_\gamma}\|w_\Gamma\|_{\XX^{s,2}_\gamma}\beta^{\frac{3}{4}}\UU(t).
    \end{align}
    Combining the estimates \eqref{GWP3}--\eqref{GWP6}, we obtain
    \begin{align*}
        \frac{\ddd}{\ddd t}\|w_\Gamma\|^2_{\XX^{s,2}_\gamma}+\beta\|w_\Gamma\|^2_{\XX^{s,2}_\gamma}+&2(\alpha\alpha_1-\lambda)\||\nabla_x|^{\frac{1}{4}}w_\Gamma\|^2_{\XX^{s,2}_\gamma}+\alpha^2\||\nabla_x|^{\frac{1}{2}}w_\Gamma\|^2_{\XX^{s,2}_\gamma}+\|\partial_yw_\Gamma\|^2_{\XX^{s,2}_\gamma}\\
        &\lesssim_{s,\gamma}\|\langle\nabla_x\rangle^{\frac{1}{2}}w_\Gamma\|_{\XX^{s,2}_\gamma}^2\left(\|w_\Gamma\|_{\XX^{s,2}_\gamma}+\|w_\Gamma\|_{\XX^{s,2}_\gamma}^2+\UU(t)\right)\\&\quad+\||\nabla_x|^{\frac{1}{2}}w_\Gamma\|_{\XX^{s,2}_\gamma}^2\left(\beta^{\frac{1}{4}}\UU(t)+\beta^{\frac{1}{2}}\UU^2(t)\right).
    \end{align*}
    Let us set
    \[\MM:=\|w_\Gamma(0)\|^2_{L^\infty(\Omega;\XX^{s,2}_{\gamma})}=\sup_{\omega\in\Omega}\|u_0(\omega)-\Psi\vec{e}\|_{\XX^{s,2}_{\gamma,\sigma_0}}^2<\infty\]
    and 
    \[\TT_{2\MM}:=\inf\{t\ge 0| \|w_\Gamma(t)\|_{\XX^{s,2}_\gamma}^2\ge 2\MM\}.\]
    Then, for any $t\le \TT_{2\MM}\wedge\TT_*$, we have 
    \begin{align*}
        \frac{\ddd}{\ddd t}\|w_\Gamma\|^2_{\XX^{s,2}_\gamma}+\beta\|w_\Gamma\|^2_{\XX^{s,2}_\gamma}+&2(\alpha\alpha_1-\lambda)\||\nabla_x|^{\frac{1}{4}}w_\Gamma\|^2_{\XX^{s,2}_\gamma}+\alpha^2\||\nabla_x|^{\frac{1}{2}}w_\Gamma\|^2_{\XX^{s,2}_\gamma}+\|\partial_yw_\Gamma\|^2_{\XX^{s,2}_\gamma}\\
        &\lesssim_{s,\gamma}C_{\MM,R}\|w_\Gamma\|_{\XX^{s,2}_\gamma}^2+(C_{\MM,R}+C_R\beta^{\frac{1}{2}})\||\nabla_x|^{\frac{1}{2}}w_\Gamma\|_{\XX^{s,2}_\gamma}^2.
    \end{align*}
    Taking
    \begin{align}\label{GWP7}
        \beta:=\frac{\alpha_1^2}{4},\qquad\lambda:=\frac{\alpha\alpha_1}{2},\qquad\alpha_1:=|\alpha|^{\frac{5}{4}}\sgn \alpha,\qquad R:=-2\log\left(\frac{\epsilon}{2}\right),
    \end{align}
    we get 
    \begin{align*}
        \frac{\ddd}{\ddd t}\|w_\Gamma\|^2_{\XX^{s,2}_\gamma}+|\alpha|^{\frac{5}{2}}\|w_\Gamma\|^2_{\XX^{s,2}_\gamma}+&\alpha^2\||\nabla_x|^{\frac{1}{2}}w_\Gamma\|^2_{\XX^{s,2}_\gamma}+\|\partial_yw_\Gamma\|^2_{\XX^{s,2}_\gamma}\\&\lesssim_{s,\gamma,\MM,R}\|w_\Gamma\|_{\XX^{s,2}_\gamma}^2+(1+|\alpha|^{\frac{5}{4}})\||\nabla_x|^{\frac{1}{2}}w_\Gamma\|_{\XX^{s,2}_\gamma}^2.
    \end{align*}
    This, together with choosing $|\alpha|$ sufficiently large, leads to 
    \[\|w_\Gamma(\TT_{2\MM}\wedge\TT_*)\|^2\le \MM,\]
    which implies $\TT_{2\MM}\ge \TT_*$ and the following energy estimate
    \begin{align}
        \label{GWP8}
        \|w_\Gamma(\TT_*)\|^2+\int^{\TT_*}_0\left(\||\nabla_x|^{\frac{1}{2}}w_\Gamma\|^2_{\XX^{s,2}_\gamma}+\|\partial_yw_\Gamma\|^2_{\XX^{s,2}_\gamma}\right)\ddd t\le \MM.
    \end{align}
    
    \noindent {\textbf{Step I\!I\!I. Constructing the probabilistic global solution}}. Combining \eqref{GWP7} with \eqref{GWP10} and taking $|\alpha|$ sufficiently large, we obtain
    \begin{align}
        \label{GWP9}\PP \{\TT_*=\infty\}\ge 1-\frac{\epsilon}{2}-\exp(\sigma_0|\alpha|^{\frac{1}{4}})\ge 1-\epsilon.
    \end{align}
    To conclude the proof of Theorem \ref{MT2}, we combine the estimates \eqref{GWP8}, \eqref{GWP9}, and the local existence theory. For simplicity, we outline a sketch instead of giving the details. Let the initial data $u_0$ satisfy \eqref{MR5_1}. For any $n\ge1$, there exists a local pathwise solution $w^n$ of \eqref{GWP0} with initial data $\mathcal{R}_nw_0$, where $\mathcal{R}_{n}f:=\mathcal{F}_x^{-1}1_{|\xi|\le n}\mathcal{F}_xf$ is the tangential regularizing operator. Using the estimate \eqref{GWP8}, one could prove the convergence of $\{w^n\}_{n=1}^\infty$, and the limit is the desired Gevrey-2 solution. 
    \end{proof}
    
    \subsection{Probabilistic global existence of Gevrey-4 solutions in 2D}

    The goal of this subsection is to establish the probabilistic global existence of Gevrey-4 solutions of the two-dimensional linearized problem \eqref{Intro3-1}, as stated in Theorem \ref{MT3}.
    
    As discussed in the previous subsection, we define the random multiplier
    \[\Gamma_t:=e^{(\sigma_0+\lambda t-\alpha B_t)|\partial_x|^{\frac{1}{4}}},\qquad u_\Gamma:=e^{-bt}\Gamma_tu\]
    with the constant $b>0$ to be determined later. Then, from \eqref{Intro3-1}, we see that $u_\Gamma$ satisfies the following equation: 
    \begin{align}\label{2DGWP0}
        \partial_t u_\Gamma+bu_\Gamma+\frac{\alpha^2}{2}|\partial_x|^\frac{1}{2} u_\Gamma-\lambda |\partial_x|^\frac{1}{4}u_\Gamma+u^s\partial_x u_\Gamma+v_\Gamma\partial_y u^s-\partial_y^2u_\Gamma=0
    \end{align}
   with $v_\Gamma=-\partial_y^{-1} \partial_xu_\Gamma:= -\int_0^y \partial_xu_\Gamma dy'$. Due to the loss of tangential derivatives, one cannot control the convection terms in this equation through direct energy method, as there is no monotonic assumption on $u^s$ in $y$. To this end, we use the idea of Dietert and Gérard-Varet \cite{DG2018}, and introduce the following good unknown $H$:
    \begin{align}
        \label{2DGWP1}
        \begin{cases}
        \partial_y^{-1}u_\Gamma=\LL \partial_y^{-1}H:=\left(\partial_t +b+\frac{\alpha^2}{2}|\partial_x|^\frac{1}{2} -\lambda |\partial_x|^\frac{1}{4}+u^s\partial_x -\partial_y^2\right)\partial_y^{-1}H,\\
         H|_{t=0}=0,\qquad\partial_yH|_{y=0}=H|_{y\to\infty}=0,
        \end{cases}
    \end{align}
    together with the corresponding test function $\varphi$:
    \begin{align}
        \label{2DGWP2}
        \begin{cases}
        H=\LL^*\varphi:=\left(-\partial_t +b+\frac{\alpha^2}{2}|\partial_x|^\frac{1}{2} -\lambda |\partial_x|^\frac{1}{4}-u^s\partial_x -\left(\partial_y+\frac{2\gamma}{\langle y\rangle}\right)^2\right)\varphi,\\
        \varphi|_{t=T}=0,\qquad \varphi|_{y=0}=\varphi|_{y\to\infty}=0.
        \end{cases}
    \end{align}
    Through direct calculation, we see that $\LL^*$ is the adjoint operator of $\LL$ in the weighted space 
    \[L^2_tL^2_{\gamma,x,y}:=\left\{u:\R_+\times \R^d_+\to\R\bigg| \|u\|^2_{L^2_tL^2_{\gamma,x,y}}:=\int_{\R_+}\int_{\R_+^d}|u|^2\langle y\rangle^{2\gamma}\ddd x\ddd y\ddd t<\infty\right\}.\]
    
    Now we are in a position to prove Theorem \ref{MT3}.

    \begin{proof}[Proof of Theorem \ref{MT3}]The proof is divided into the following steps.
    
    \noindent {\textbf{Step I. Estimate for $u$.}}
    Taking the inner product in $H^{s,0}_{\gamma}$ on both sides of \eqref{2DGWP0}, we get 
    \begin{align*}
        \frac{1}{2}\frac{\ddd}{\ddd t}\|u_\Gamma\|^2_{H^{s,0}_\gamma}+b\|u_\Gamma\|^2_{H^{s,0}_\gamma}+\|\partial_y u_\Gamma\|^2_{H^{s,0}_\gamma}+&\frac{\alpha^2}{2}\||\partial_x|^{\frac{1}{4}}u_\Gamma\|^2_{H^{s,0}_\gamma}-\lambda\||\partial_x|^{\frac{1}{8}}u_\Gamma\|^2_{H^{s,0}_\gamma}\\&=-2\gamma\langle \langle y\rangle^{-1}u_\Gamma ,\partial_yu_\Gamma\rangle_{H^{s,0}_\gamma}-\langle u_\Gamma,v_\Gamma\partial_y u^s\rangle_{H^{s,0}_\gamma},
    \end{align*}
    where 
    \[-2\gamma\langle \langle y\rangle^{-1}u_\Gamma ,\partial_yu_\Gamma\rangle_{H^{s,0}_\gamma}\lesssim_\gamma\|u_\Gamma\|_{H^{s,0}_\gamma}\|\partial_y u_\Gamma\|_{H^{s,0}_\gamma},\]
    and 
    \[-\langle u_\Gamma,v_\Gamma\partial_y u^s\rangle_{H^{s,0}_\gamma}\lesssim_\gamma\||\partial_x|^{\frac{1}{2}}u_\Gamma\|^2_{H^{s,0}_\gamma}\|\langle y\rangle^\gamma\partial_yu^s\|_{L^2_y}.\]
    This implies
    \begin{align}
        \label{2DGWP3}
        &\|u_\Gamma(T)\|^2_{H^{s,0}_\gamma}+\int_0^T\left(\|\partial_yu_\Gamma\|^2_{H^{s,0}_\gamma}+\alpha^2\||\partial_x|^{\frac{1}{4}}u_\Gamma\|^2_{H^{s,0}_\gamma}\right)\ddd t\notag
        \\&\qquad+\left(4b-C_\gamma-C\frac{\lambda^2}{\alpha^2}\right)\int_0^T\|u_\Gamma\|^2_{H^{s,0}_\gamma}\ddd t\lesssim_{u^s,\gamma} \|u_\Gamma(0)\|^2_{H^{s,0}_\gamma}+\int_0^T\||\partial_x|^{\frac{1}{2}}|u_\Gamma\|^2_{H^{s,0}_\gamma}\ddd t.
    \end{align}
    Similarly, we can deduce
    \begin{align*}
        \|\partial_yu_\Gamma(T)\|^2_{H^{s-1/4,0}_\gamma}&+\int_0^T\left(\|\partial^2_yu_\Gamma\|^2_{H^{s-1/4,0}_\gamma}+\alpha^2\||\partial_x|^{\frac{1}{4}}\partial_yu_\Gamma\|^2_{H^{s-1/4,0}_\gamma}\right)\ddd t\notag
        \\&\quad+\left(4b-C_\gamma-C\frac{\lambda^2}{\alpha^2}\right)\int_0^T\|\partial_yu_\Gamma\|^2_{H^{s-1/4,0}_\gamma}\ddd t\notag\\&\lesssim_{u^s,\gamma}\|\partial_y u_\Gamma(0)\|^2_{H^{s-1/4,0}_\gamma}+\frac{1}{\alpha^2}\int_0^T\||\partial_x|^{\frac{1}{2}}|u_\Gamma\|^2_{H^{s,0}_\gamma}\ddd t.
    \end{align*}
    This, together with \eqref{2DGWP3}, implies
    \begin{align}
        \label{2DGWP4}
        \|u_\Gamma(T)\|^2_{\XX^{s,4}_\gamma}+\int_0^T&\left(\|\partial_yu_\Gamma\|^2_{\XX^{s,4}_\gamma}+\alpha^2\||\partial_x|^{\frac{1}{4}}u_\Gamma\|^2_{\XX^{s,4}_\gamma}\right)\ddd t+\left(4b-C_\gamma-C\frac{\lambda^2}{\alpha^2}\right)\int_0^T\|u_\Gamma\|^2_{\XX^{s,4}_\gamma}\ddd t\notag
        \\&\lesssim_{u^s,\gamma} \|u_\Gamma(0)\|^2_{\XX^{s,4}_\gamma}+\int_0^T\||\partial_x|^{\frac{1}{2}}|u_\Gamma\|^2_{H^{s,0}_\gamma}\ddd t.
    \end{align}
    To control the last term on the right hand side of \eqref{2DGWP4}, we use the good unknown $H$ introduced in the problem \eqref{2DGWP1}, from which we get
    \begin{align}\label{2DGWP5}
        \int_0^T\||\partial_x|^{\frac{1}{2}}u_\Gamma\|^2_{H^{s,0}_\gamma}\ddd t=\int_0^T\langle \LL H+\partial_y u^s\partial_x\partial^{-1}_yH,|\partial_x|u_\Gamma\rangle_{H^{s,0}_\gamma}\ddd t=:J^1+J^2,
    \end{align}
    where 
    \begin{align}\label{2DGWP6}
        J^2\lesssim_{u^s,\gamma} \int_0^T \|H\|_{H^{s+3/2,0}_\gamma}\||\partial_x|^{\frac{1}{2}}u_\Gamma\|_{H^{s,0}_\gamma}\ddd t.
    \end{align}
    For $J^1$, we take integration by parts to derive 
    \begin{align}\label{2DGWP7}
        J^1= \langle H(T),|\partial_x|u_\Gamma(T)\rangle_{H^{s,0}_\gamma}+\int_0^T\langle |\partial_x|H,\LL^*u_\Gamma\rangle_{H^{s,0}_\gamma}\ddd t,
    \end{align}
    where by equation \eqref{2DGWP0} and the definition of $\LL^*$ in \eqref{2DGWP2}, we have 
    \begin{align*}
        \LL^*u_\Gamma=2bu_\Gamma+\alpha^2|\partial_x|^{\frac{1}{2}}u_\Gamma-2\lambda|\partial_x|^{\frac{1}{4}}u_\Gamma-2\partial_y^2u_\Gamma+v_\Gamma\partial_y u^s-\frac{2\gamma(2\gamma-1)}{\langle y\rangle^2}u_\Gamma-\frac{4\gamma}{\langle y\rangle}\partial_yu_\Gamma.
    \end{align*}
    Combining this with \eqref{2DGWP7}, we see that
    \begin{align}\label{2DGWP8}\notag
        J^1&= \langle H(T),|\partial_x|u_\Gamma(T)\rangle_{H^{s,0}_\gamma}+2b\int_0^T\langle |\partial_x|H,u_\Gamma\rangle_{H^{s,0}_\gamma}\ddd t+\alpha^2\int_0^T\langle |\partial_x|H,|\partial_x|^{\frac{1}{2}}u_\Gamma\rangle_{H^{s,0}_\gamma}\ddd t\\\notag&\quad-2\lambda\int_0^T\langle |\partial_x|H,|\partial_x|^{\frac{1}{4}}u_\Gamma\rangle_{H^{s,0}_\gamma}\ddd t-2\int_0^T\langle |\partial_x|H,\partial_y^2u_\Gamma\rangle_{H^{s,0}_\gamma}\ddd t\\\notag&\quad+\int_0^T\langle |\partial_x|H,v_\Gamma\partial_yu^s\rangle_{H^{s,0}_\gamma}\ddd t-2\gamma(2\gamma-1)\int_0^T\langle |\partial_x|H,u_\Gamma\langle y\rangle^{-2}\rangle_{H^{s,0}_\gamma}\ddd t\\&\quad-4\gamma\int_0^T\langle |\partial_x|H,\partial_yu_\Gamma\langle y\rangle^{-1}\rangle_{H^{s,0}_\gamma}\ddd t=:J^{1,1}+\ldots+J^{1,8}.
    \end{align}
    We estimate the terms $J^{1,1},\ldots,J^{1,8}$ in the following way:
    \begin{align*}
        J^{1,1}&\le \|H(T)\|_{H^{s+1,0}_\gamma}\|u_\Gamma(T)\|_{H^{s,0}_{\gamma}},\\
        J^{1,2}+J^{1,3}+J^{1,4}+J^{1,7}&\lesssim_{b,\alpha,\lambda,\gamma} \int_0^T\|H\|_{H^{s+1,0}_\gamma}\||\partial_x|^\frac{1}{2} u_\Gamma\|_{H^{s,0}_\gamma}\ddd t,\\
        J^{1,5}&\le 2\int_0^T\|H\|_{H^{s+5/4
        ,0}_\gamma}\|\partial_y^2u_\Gamma\|_{H^{s-1/4,0}_\gamma}\ddd t,\\
        J^{1,6}&\lesssim_{u^s,\gamma} \int_0^T\|H\|_{H^{s+3/2,0}_\gamma}\||\partial_x|^\frac{1}{2} u_\Gamma\|_{H^{s,0}_\gamma}\ddd t,\\
        J^{1,8}&\le 4\gamma\int_0^T\|H\|_{H^{s+1,0}_\gamma}\|\partial_y u_\Gamma\|_{H^{s,0}_\gamma}\ddd t.
    \end{align*}
    This, together with \eqref{2DGWP6} and \eqref{2DGWP8}, implies 
    \begin{align*}
        \int_0^T\||\partial_x|^{\frac{1}{2}}u_\Gamma\|^2_{H^{s,0}_\gamma}\ddd t&\lesssim\|H(T)\|_{H^{s+1,0}_\gamma}\|u_\Gamma(T)\|_{H^{s,0}_{\gamma}}+C_{u^s,b,\alpha,\lambda,\gamma} \int_0^T\|H\|_{H^{s+3/2,0}_\gamma}^2\ddd t\notag\\&\quad+\int_0^T\|H\|_{H^{s+5/4
        ,0}_\gamma}\|\partial_y^2u_\Gamma\|_{H^{s-1/4,0}_\gamma}\ddd t+C_\gamma\int_0^T\|H\|_{H^{s+1,0}_\gamma}\|\partial_y u_\Gamma\|_{H^{s,0}_\gamma}\ddd t.
    \end{align*}
    Plugging the above estimate into \eqref{2DGWP4} and using the Cauchy--Schwartz inequality, we obtain
    \begin{align}
        \label{2DGWP9}
        \|u_\Gamma(T)\|^2_{\XX^{s,4}_\gamma}+\int_0^T&\left(\|\partial_yu_\Gamma\|^2_{\XX^{s,4}_\gamma}+\alpha^2\||\partial_x|^{\frac{1}{4}}u_\Gamma\|^2_{\XX^{s,4}_\gamma}\right)\ddd t+\left(4b-C_\gamma-C\frac{\lambda^2}{\alpha^2}\right)\int_0^T\|u_\Gamma\|^2_{\XX^{s,4}_\gamma}\ddd t\notag
        \\&\lesssim_{u^s,b,\alpha,\lambda,\gamma} \|u_\Gamma(0)\|^2_{\XX^{s,4}_\gamma}+\|H(T)\|_{H^{s+1,0}_\gamma}^2+\int_0^T\|H\|_{H^{s+3/2,0}_\gamma}^2\ddd t.
    \end{align}

    \noindent {\textbf{Step I\!I. Estimate for $H $.}} Using equation \eqref{2DGWP1} and the divergence-free condition, we have 
    \begin{align}
        \label{2DGWP10}
        u_\Gamma=\LL H+\partial_y u^s\partial_x\partial_y^{-1}H,\\
        \label{2DGWP11} -v_\Gamma=\LL\partial_x\partial_y^{-1}H.
    \end{align}
    Rewrite equation \eqref{2DGWP0} in the form:
    \begin{align*}
        \LL u_\Gamma+v_\Gamma\partial_y u^s=0.
    \end{align*}
    Testing the above equation against the test function $\varphi$ given by \eqref{2DGWP2}, we find
    \begin{align}\label{2DGWP12}
        0&=\int_0^T\langle v_\Gamma\partial_yu^s,\varphi\rangle_{H^{r,0}_\gamma}\ddd t+\int_0^T\langle\LL u_\Gamma,\varphi\rangle_{H^{r,0}_\gamma}\ddd t\notag\\
        &=\int_0^T\langle v_\Gamma\partial_yu^s,\varphi\rangle_{H^{r,0}_\gamma}\ddd t-\langle u_\Gamma(0),\varphi(0)\rangle_{H^{r,0}_\gamma}+\int_0^T\langle u_\Gamma,\LL^*\varphi\rangle_{H^{r,0}_\gamma}\ddd t\notag\\
        &=-\int_0^T\langle \partial_yu^s\LL\partial_x\partial_y^{-1}H,\varphi\rangle_{H^{r,0}_\gamma}\ddd t-\langle u_\Gamma(0),\varphi(0)\rangle_{H^{r,0}_\gamma}\notag\\&\quad+\int_0^T\langle \LL H,H\rangle_{H^{r,0}_\gamma}\ddd t+\int_0^T\langle \partial_y u^s\partial_x\partial_y^{-1}H,H\rangle_{H^{r,0}_\gamma}\ddd t\notag\\&=2\int_0^T\langle \partial_y^2u^s\partial_x\partial_y^{-1}H,\varphi\rangle_{H^{r,0}_\gamma}\ddd t+\int_0^T\langle \LL H,H\rangle_{H^{r,0}_\gamma}\ddd t-\langle u_\Gamma(0),\varphi(0)\rangle_{H^{r,0}_\gamma}\notag\\&=:J^3+J^4-\langle u_\Gamma(0),\varphi(0)\rangle_{H^{r,0}_\gamma},
    \end{align}
    where we used the following cancellation property:
    \begin{align*}
    \int_0^T\langle \partial_yu^s\LL\partial_x\partial_y^{-1}H,\varphi\rangle_{H^{r,0}_\gamma}\ddd t&=\int_0^T\langle \LL(\partial_yu^s\partial_x\partial_y^{-1}H),\varphi\rangle_{H^{r,0}_\gamma}\ddd t\\&\quad-2\int_0^T\langle \partial_y^2u^s\partial_x\partial_y^{-1}H,\varphi\rangle_{H^{r,0}_\gamma}\ddd t\notag\\&=\int_0^T\langle \partial_y u^s\partial_x\partial_y^{-1}H,H\rangle_{H^{r,0}_\gamma}\ddd t-2\int_0^T\langle \partial_y^2u^s\partial_x\partial_y^{-1}H,\varphi\rangle_{H^{r,0}_\gamma}\ddd t.
    \end{align*}
    Notice that 
    \begin{align}
        \label{2DGWP13}
        |J^3|\lesssim_{u^s,\gamma}\int_0^T\||\partial_x|^{\frac{1}{4}}H\|_{H^{r,0}_\gamma}\|\varphi\|_{H^{r+3/4,0}_\gamma}\ddd t,
    \end{align}
    and 
    \begin{align}
        \label{2DGWP14}
        J^4&\ge \frac{1}{2}\|H(T)\|^2_{H^{r,0}_\gamma}+\int_0^T\left(b\|H\|_{H^{r,0}_\gamma}^2+\|\partial_y H\|^2_{H^{r,0}_\gamma}+\frac{\alpha^2}{2}\||\partial_x|^{\frac{1}{4}} H\|^2_{H^{r,0}_\gamma}\right)\ddd t\notag\\&\quad -\lambda\int_0^T\||\partial_x|^{\frac{1}{8}} H\|^2_{H^{r,0}_\gamma}\ddd t-2\gamma\int_0^T\|\partial_y H\|_{H^{r,0}_\gamma}\|H\|_{H^{r,0}_\gamma}\ddd t.
    \end{align}
    Combining the estimates \eqref{2DGWP12}--\eqref{2DGWP14}, we conclude 
    \begin{align}
        \label{2DGWP15}
        \|H(T)\|^2_{H^{r,0}_\gamma}+\int_0^T&\left(\|\partial_y H\|^2_{H^{r,0}_\gamma}+\alpha^2\||\partial_x|^{\frac{1}{4}} H\|^2_{H^{r,0}_\gamma}\right)\ddd t+\left(4b-C_\gamma
        -C\frac{\lambda^2}{\alpha^2}\right)\int_0^T\|H\|_{H^{r,0}_\gamma}^2\ddd t\notag\\&\lesssim_{u^s,\gamma}\|u_\Gamma(0)\|_{H^{r-1/2,0}_\gamma}\|\varphi(0)\|_{H^{r+1/2,0}_\gamma}+\frac{1}{\alpha^2}\int_0^T \|\varphi\|_{H^{r+3/4,0}_\gamma}^2\ddd t.
    \end{align}
    It remains to estimate the test function $\varphi$. Applying the standard energy estimate on the linear degenerate parabolic equation \eqref{2DGWP2}, we see that 
    \begin{align}
        \label{2DGWP16}
        \|\varphi(0)\|^2_{H^{r+1/2,0}_\gamma}&+\int_0^T\left(\|\partial_y \varphi\|^2_{H^{r+1/2,0}_\gamma}+\alpha^2\||\partial_x|^{\frac{1}{4}} \varphi\|^2_{H^{r+1/2,0}_\gamma}\right)\ddd t\notag\\&\quad+\left(4b-C_\gamma-C\frac{\lambda^2}{\alpha^2}\right)\int_0^T\|\varphi\|_{H^{r+1/2,0}_\gamma}^2\ddd t\lesssim\frac{1}{\alpha^2} \int_0^T\|H\|_{H^{r+1/4,0}_\gamma}^2\ddd t.
    \end{align}

    \noindent {\textbf{Step I\!I\!I. Existence of a probabilistic global solution.}} 
    Let us combine the estimates \eqref{2DGWP9}, \eqref{2DGWP15}, and \eqref{2DGWP16} to close the energy estimate. Taking $r=s+\frac{5}{4}$ in \eqref{2DGWP15} and \eqref{2DGWP16}, and choosing 
    \[\alpha^2\gg C_{u^s,\gamma},\qquad \lambda=-\frac{2\alpha^2}{\sigma_0}\log\epsilon,\qquad 4b\ge C_\gamma-C\frac{\lambda^2}{\alpha^2}+1, \]
    we find 
    \begin{align}
        \label{2DGWP17}\|H(T)\|^2_{H^{s+5/4,0}_\gamma}+\int_0^T&\left(\|\partial_y H\|^2_{H^{s+5/4,0}_\gamma}+\|H\|^2_{H^{s+3/2,0}_\gamma}\right)\ddd t
        \lesssim_{u^s,\gamma}\|u_\Gamma(0)\|_{H^{s+\frac{3}{4},0}_\gamma}^2.
    \end{align}
    Plugging the above estimate into \eqref{2DGWP9}, we obtain
    \begin{align*}
    \|u_\Gamma(T)\|^2_{\XX^{s,4}_\gamma}+\int_0^T&\left(\|\partial_yu_\Gamma\|^2_{\XX^{s,4}_\gamma}+\||\partial_x|^{\frac{1}{4}}u_\Gamma\|^2_{\XX^{s,4}_\gamma}+\|u_\Gamma\|^2_{\XX^{s,4}_\gamma}\right)\ddd t\\&
        \qquad\qquad\qquad\qquad\qquad\lesssim_{u^s,b,\alpha,\lambda,\gamma} \|u_\Gamma(0)\|_{H^{s+\frac{3}{4},0}_\gamma}^2+\|u_\Gamma(0)\|^2_{\XX^{s,4}_\gamma}.
    \end{align*}
    This, together with the probabilistic estimate 
    \[\PP\left\{\alpha B_t\ge \frac{\sigma_0+\lambda t}{2},\ \forall t\ge 0\right\}\ge 1-\exp\left(-\frac{\lambda\sigma_0}{2\alpha^2}\right)=1-\epsilon,\]
    ensures the existence of a global Gevrey-4 solution of \eqref{Intro3-1} with high probability.
    \end{proof}

    \appendix
    \section{Formal derivation of the stochastic Prandtl equation}\label{derivation}
     In this appendix, we sketch a formal derivation of the stochastic Prandtl equation in the inviscid limit of the stochastic NS equation with no-slip boundary conditions. Consider the following stochastic NS equation:
    \begin{align}\label{pre7}
        \begin{cases}
    \mathrm{d} u^{\epsilon}(t)+((u^{\epsilon}\cdot\nabla) u^{\epsilon}+\nabla p^{\epsilon}-\epsilon\Delta u^{\epsilon})\mathrm{d}t=\F^\epsilon(x,y,u^{\epsilon})\mathrm{d}W,\qquad(x,y)\in\R^{d}_+,\\
    \dive u^{\epsilon}=0,\\
    u^{\epsilon}(0)=u_0,\qquad u^{\epsilon}|_{y=0}=0,
    \end{cases}
    \end{align}
    where $d=2$ or $3$, $\epsilon$ is the kinematic viscosity, 
    \[u^\epsilon:=\begin{cases}
        (u_h^{\epsilon},u_v^{\epsilon}),\qquad &d=2,\\
        (u_h^{\epsilon},u_v^{\epsilon})=(u_{h,1}^{\epsilon},u_{h,2}^{\epsilon},u_v^{\epsilon}),\qquad &d=3
    \end{cases}\]
    is the velocity field, $p^{\epsilon}$ denotes the pressure, and $\F^\epsilon(\cdot)\mathrm{d}W$ represents the random force.
    Due to the Prandtl theory given in \cite{P1904} for the corresponding deterministic problem, as $\epsilon\to 0$, there is a boundary layer near $\{y=0\}$ with the thickness being $\sqrt{\epsilon}$, so we take the following Prandtl ansatz for $u^{\epsilon}$, $p^{\epsilon}$, and $\F^\epsilon$:
    \begin{align}\begin{cases}\label{pre8}
    u^{\epsilon}(t,x,y)=u^{I,0}(t,x,y)+u^{B,0}\left(t,x,\frac{y}{\sqrt{\epsilon}}\right)+\sqrt{\epsilon}(u^{I,1}(t,x,y)+u^{B,1}\left(t,x,\frac{y}{\sqrt{\epsilon}}\right)+o(\sqrt{\epsilon}),\\
    p^{\epsilon}(t,x,y)=p^{I,0}(t,x,y)+p^{B,0}\left(t,x,\frac{y}{\sqrt{\epsilon}}\right)+O(\sqrt{\epsilon}),\\
    \F^{\epsilon}(x,y,\cdot)=\F^{I,0}(x,y,\cdot)+\F^{B,0}\left(x,\frac{y}{\sqrt{\epsilon}},\cdot\right)+O(\sqrt{\epsilon}),
    \end{cases}
    \end{align}
    where $
    u^{B,0}(t,x,Y), u^{B,1}(t,x,Y), p^{B,0}(t,x,Y),\F^{B,0}(x,Y,\cdot)$ together with their derivatives decay rapidly to zero, as the fast variable $Y:=y/\sqrt{\epsilon}\to \infty$. 
    
    Plugging the ansatz \eqref{pre8}$_1$ into the divergence-free condition \eqref{pre7}$_2$, and using the fast decay property of $u^{B,0}(t,x,Y)$ and $u^{B,1}(t,x,Y)$ in $Y$, we get 
    \begin{equation}\label{A.4}
    \partial_Y u_v^{B,0}=0,\qquad u_v^{B,0}(t,x,Y)\equiv 0,
    \end{equation}
    which implies
    \begin{equation}\label{A.5}
    \nabla\cdot u^{I,0}=0, \qquad 
    \nabla_x\cdot u_h^{B,0}+\partial_Yu_v^{B,1}=0,
    \end{equation}
    where $\nabla_x$ is the gradient operator in the tangential variable $x\in \mathbb{R}^{d-1}$. 

   Similarly, combining the vertical component of the momentum equation \eqref{pre7}$_1$ with the ansatz \eqref{pre8}$_2$ and the relation \eqref{A.4}, we have $\partial_Y p^{B,0}=0$,
   which implies 
   \begin{equation}\label{A.6}
   p^{B,0}(t,x,Y)\equiv 0.
   \end{equation}
    
   To derive the equation for the leading outflow profile $(u^{I,0}, p^{I,0})$, we use the momentum equation \eqref{pre7}$_1$ combined with the ansatz \eqref{pre8} and the relations \eqref{A.4} and \eqref{A.5}, and take $\epsilon \to 0$ to get the following stochastic Euler equation
   \begin{equation}\label{A.7}
   \begin{cases}    
   \ddd u^{I,0}+((u^{I,0}\cdot\nabla) u^{I,0}+\nabla p^{I,0})\mathrm{d}t=\F^{I,0}(x,y,u^{I,0})\mathrm{d}W,\\
   \nabla\cdot u^{I,0}=0,\\
   u_v^{I,0}|_{y=0}=0.
   \end{cases}
   \end{equation}

   On the other hand, plugging the ansatz \eqref{pre8} into \eqref{pre7}$_1$, we get
    \begin{align}\label{A.8}
    &\ddd(u_h^{I,0}+u_h^{B,0})+\left(
    (u_h^{I,0}+u_h^{B,0})\cdot\nabla_x(u_h^{I,0}+u_h^{B,0})\right)\ddd t\notag\\&\quad+\left(u_v^{I,0}(\partial_Y u_h^{B,1}+\partial_y u_h^{I,0})\ddd t 
    +\left(\frac{1}{\sqrt{\epsilon}}u_v^{I,0}+u_v^{I,1}+u_v^{B,1}\right)\partial_Yu_h^{B,0}+\nabla_xp^{I,0}-\partial_Y^2 u_h^{B,0}\right)\ddd t\notag\\&=\left(\F^{I,0}(x,y, u^{I,0}+u^{B,0})+\F^{B,0}\left(x,\frac{y}{\sqrt{\epsilon}}, u^{I,0}+u^{B,0}\right)\right)\ddd W+O(\sqrt{\epsilon}).
   \end{align}
   Let us write
  \begin{equation}\label{A.9}
 \begin{cases}    
 u^P(t,x,Y):=u_h^{I,0}(t,x,0)+u_h^{B,0}(t,x,Y),\\
 v^P(t,x,Y):=Y\partial_yu_v^{I,0}(t,x,0)+
 u_v^{I,1}(t,x,0)+u_v^{B,1}(t,x,Y),\\
 \F^P(x,Y,\cdot):=\F^{I,0}(x,0,\cdot)+\F^{B,0}(x,Y,\cdot).
 \end{cases}
   \end{equation}
Taking the change of variable $y=\sqrt{\epsilon}Y$, using the Taylor expansion of $u^{I,0}$, $u^{I,1}$, $p^{I,0}$ and $\F^{I,0}$ in \eqref{A.8}, and combining the relations \eqref{A.5} and \eqref{A.7} with the no-slip conditions \eqref{pre7}$_3$, we conclude that $(u^P, v^P)$ satisfies the following stochastic Prandtl equation:
 \begin{align*}
    \begin{cases}
    \mathrm{d} u^P+(u^P\cdot\nabla_{x}  u^P+v^P\partial_Y u^P-\partial_Y^2 u^P+\nabla_x \overline{p^{I,0}})\mathrm{d}t=\F^P(x,Y,u^P)\mathrm{d}W,\quad (x,Y)\in\R^{d}_+,\\
    \nabla_x \cdot u+\partial_Y v=0,\\
u|_{Y=0}=v|_{Y=0}=0,\qquad \lim\limits_{Y\to\infty}u=\overline{u_h^{I,0}},
    \end{cases}
    \end{align*}
    where $(\overline{u_h^{I,0}}, \overline{p^{I,0}}):=(u_h^{I,0}(t,x,0), p^{I,0}(t,x,0))$ satisfies the stochastic Bernoulli law 
    \[\ddd \overline{u^{I,0}_h}+ (\overline{u^{I,0}_h}\cdot\nabla_x  \overline{u^{I,0}_h}+\nabla_x \overline{p^{I,0}})\ddd t=\F^{I,0}(x,0, \overline{u^{I,0}_h})\ddd W,\]
    and
    \[\lim_{Y\to\infty}\F^P(x,Y,\cdot)=\F^{I,0}(x,0,\cdot).\]

    \section{Proof of auxiliary estimates}

    In this appendix, we prove Lemma \ref{product_estimate1} and Proposition \ref{Proposition_struc1}, which are used to control the convection term and the nonlinear multiplicative noise.

    \subsection{Proof of Lemma \ref{product_estimate1}}\label{proof_product_estimate}
    Let us recall some basic results of the Littlewood-Paley theory, cf. \cite{BCD2011}. Define
    \[\Delta_j f:=\mathcal{F}_x^{-1}(\varphi(2^{-j}|\xi|)\mathcal{F}_xf),\qquad S_jf:=\mathcal{F}_x^{-1}(\chi(2^{-j}|\xi|)\mathcal{F}_xf),\qquad j\ge 0,\]
    and 
    \[\Delta_{-1} f:=S_0f,\qquad \Delta_{j}f:=0,\qquad j\le -2,\]
    where $\varphi,\chi$ are smooth radial functions valued in $[0,1]$ such that 
    \[\supp \varphi\subset \left\{\xi\in \R^{d-1}\bigg| \frac{3}{4}\le |\xi|\le \frac{8}{3}\right\},\qquad \supp \chi\subset \left\{\xi\in \R^{d-1}\bigg| |\xi|\le \frac{4}{3}\right\},\]
    and $\chi(\cdot)+\sum_{j\ge 0}\varphi(2^{-j}\cdot)\equiv 1.$ Let $T_f g:=\sum_{j}S_{j-1}f\Delta_jg$, then we have the following Bony decomposition 
    \[fg=T_f g+T_gf+\mathcal{R}(f,g),\]
    where $\mathcal{R}(f,g):=\sum_{j}\tilde{\Delta}_jf\Delta_jg$ with $\tilde{\Delta}_jf:=\sum_{|j^{'}-j|\le 1}\Delta_{j^{'}}f$. We need the following auxiliary estimate. 
    
    \begin{lemma}\label{lemmaA1} Let $f,g\in \mathcal{S}(\R^{d})$ and $|a|\le d$. Then,
    \[\|\langle \nabla_x\rangle^a (fg)\|_{L^2}\le C\|\langle \nabla_x\rangle^a f\|_{L^2}\|g\|_{H^{d}}.\]
    \end{lemma}
    \begin{proof}
    It suffices to consider the case $0\le a\le d$, since the other follows from a duality argument. By the Bony decomposition,
        \[\|\langle \nabla_x\rangle^a (fg)\|_{L^2}\le \|T_f g\|_{H^a}+\|T_g f\|_{H^a}+\|\mathcal{R}(f,g)\|_{H^a}.\]
        Using the Bernstein inequality, see Lemma 2.1 in \cite{BCD2011}, we have
        \begin{align*}
            \|\Delta_k T_f g\|_{L^2}&\le \sum_{|k-k^{'}|\le 4}\|\Delta_k (S_{k^{'}-1}f \Delta_{k^{'}}g)\|_{L^2}\\&\le C\sum_{|k-k^{'}|\le 4}\|S_{k^{'}-1}f \Delta_{k^{'}}g\|_{L^2}\le C\sum_{|k-k^{'}|\le 4}\|\Delta_{k^{'}}g\|_{L^2}\sum_{l\le k^{'}-2}2^{\frac{d}{2}l}\|\Delta_l f\|_{L^2}.
        \end{align*}
        Since $a\le d$, we have
        \begin{align*}
            2^{ka}\|\Delta_k T_f g\|_{L^2}&\le C\sum_{|k-k^{'}|\le 4}2^{k^{'}d}\|\Delta_{k^{'}}g\|_{L^2}2^{k^{'}(a-d)}\sum_{l\le k^{'}-2}2^{\frac{d}{2}l}\|\Delta_l f\|_{L^2}
            \\&\le C\sum_{|k-k^{'}|\le 4}2^{k^{'}d}\|\Delta_{k^{'}}g\|_{L^2}\sum_{l\le k^{'}-2}2^{\left(a-\frac{d}{2}\right)l}\|\Delta_l f\|_{L^2}\\&\le C\|f\|_{H^a}\sum_{|k-k^{'}|\le 4}2^{k^{'}d}\|\Delta_{k^{'}}g\|_{L^2},
        \end{align*}
        which gives 
        \begin{align}
            \label{A1}
            \|T_f g\|_{H^a} \le C\|f\|_{H^a}\|g\|_{H^{d}}.
        \end{align}
        On the other hand, applying the boundedness of the para-product operator $T_g$, cf. Theorem 2.82 in \cite{BCD2011}, we get
        \begin{align}
            \label{A2}
            \|T_g f\|_{H^a} \le C\|g\|_{L^{\infty}}\|f\|_{H^a} \le C\|g\|_{H^{d}}\|f\|_{H^a}.
        \end{align}
        For the reminder $\mathcal{R}(f,g)$, we have
        \begin{align*}
            2^{ka}\|\Delta_k \mathcal{R}(f,g)\|_{L^2}&\le C2^{ka}\sum_{k^{'}\ge k-3}\|\tilde{\Delta}_{k^{'}}f\|_{L^2}\|\Delta_{k^{'}}g\|_{L^2}2^{\frac{d}{2}k^{'}}\\&\le C\sum_{k^{'}\ge k-3}2^{(k-k^{'})a}2^{k^{'}a}\|\tilde{\Delta}_{k^{'}}f\|_{L^2}\|\Delta_{k^{'}}g\|_{L^2}2^{\frac{d}{2}k^{'}}.
        \end{align*}
        Using the Young inequality and the H\"older inequality, we obtain
        \begin{align}
            \label{A3}
            \|\mathcal{R}(f,g)\|_{H^a}\le C\|f\|_{H^a} \|g\|_{H^{\frac{d}{2}}}.
        \end{align}
        Combining the estimates \eqref{A1}--\eqref{A3}, we conclude the proof of this lemma.
    \end{proof}

    Let us turn to the proof of Lemma \ref{product_estimate1}.
    \begin{proof}[Proof of Lemma \ref{product_estimate1}] Notice that
    \[\mathcal{F}_x
    (uv)_{\sigma}(\xi)\le \mathcal{F}_x(u^+_{\sigma}v^+_{\sigma})(\xi),\]
    where $u^+:=\mathcal{F}_x^{-1}|\mathcal{F}_xu|$, and the map $u\mapsto u^+$ preserves the $L^2_x$-norm. Therefore, it suffices to prove the product estimates in the weighted conormal Sobolev space $X^s_{\gamma}$. 

    The first product estimate given in Lemma \ref{product_estimate1} follows from a direct application of the classical Sobolev embedding theorem combined with a careful distribution of the conormal derivatives. The second estimate can be established analogously, relying additionally on the estimates:
    \[\|\partial_y^{-1} f\|_{L_y^{\infty}L^2_x}\le \|f\|_{L_y^{1}L^2_x} \lesssim_{\gamma} \|e^{\gamma y^2}f\|_{L_2},\qquad 1+y\lesssim_{\gamma} e^{\gamma y^2}.\]
    The third estimate can be derived similarly by incorporating Lemma \ref{lemmaA1}. Therefore, we omit their proofs and give a detailed proof of the last product estimate.
    
    By definition, we have
    \begin{align*}
        &\|\langle \nabla_x\rangle^{-\frac{1}{2}}[(\partial_y^{-1}\nabla_x \cdot u) \partial_y v]\|_{H^s_{\gamma,\co}}\\&\qquad\qquad\lesssim_s\sum_{|k|+j\le s}\sum_{\substack{k_1+k_2=k\\j_1+j_2=j}}\|e^{\gamma y^2}\langle \nabla_x\rangle^{-\frac{1}{2}}[Z^{j_1}\partial_x^{k_1}(\partial_y^{-1}\nabla_x \cdot u) Z^{j_2}\partial_x^{k_2}\partial_y v]\|_{L^2}.
    \end{align*}
    For $j_1=0$, we derive
    \begin{align}
        \label{A9}\|Z^{j_1}\partial_x^{k_1}(\partial_y^{-1}\nabla_x \cdot u) \|_{L^{\infty}_yL^2_x}=\|\partial_x^{k_1}(\partial_y^{-1}\nabla_x \cdot u) \|_{L^{\infty}_yL^2_x}\lesssim_{\gamma} \|e^{\gamma y^2}\partial_x^{k_1}\nabla_x u\|_{L^2},
    \end{align}
    while for $j_1>0$, by the Sobolev embedding and the relation $y Z^{j-1}\partial_y=Z^j$, we get
    \begin{align}
        \label{A10}\|Z^{j_1}\partial_x^{k_1}(\partial_y^{-1}\nabla_x \cdot u) \|_{L^{\infty}_yL^2_x}&= \|y Z^{j_1-1}\partial_x^{k_1}\nabla_x u\|_{L^{\infty}_yL^2_x}\notag\\&\lesssim \|y Z^{j_1-1}\partial_x^{k_1}\nabla_x u\|_{L^2}+j_1\| Z^{j_1-1}\partial_x^{k_1}\nabla_x u\|_{L^2}+\| Z^{j_1}\partial_x^{k_1}\nabla_x u\|_{L^2}\notag\\&\lesssim_{s,\gamma} \|e^{\gamma y^2} Z^{j_1-1}\partial_x^{k_1}\nabla_x u\|_{L^2}+\| e^{\gamma y^2}Z^{j_1}\partial_x^{k_1}\nabla_x u\|_{L^2}.
    \end{align}
    Hence, for $j_1+|k_1|\le s-3$, we have
    \begin{align}\label{A10_1}
        \|e^{\gamma y^2}\langle \nabla_x\rangle^{-\frac{1}{2}}[Z^{j_1}\partial_x^{k_1}(\partial_y^{-1}\nabla_x \cdot u) Z^{j_2}\partial_x^{k_2}\partial_y v]\|_{L^2}&\lesssim \|Z^{j_1}\partial_x^{k_1}(\partial_y^{-1}\nabla_x \cdot u)\|_{L^{\infty}}\|e^{\gamma y^2}Z^{j_2}\partial_x^{k_2}\partial_y v\|_{L^2}\notag\\&\lesssim_{s,\gamma}\|u\|_{X^s_{\gamma}}\|\partial_yv\|_{X^s_{\gamma}}.
    \end{align}
    Similarly, for $j_1+|k_1|= s-2$,
    \begin{align}\label{A10_2}
         &\|e^{\gamma y^2}\langle \nabla_x\rangle^{-\frac{1}{2}}[Z^{j_1}\partial_x^{k_1}(\partial_y^{-1}\nabla_x \cdot u) Z^{j_2}\partial_x^{k_2}\partial_y v]\|_{L^2}\notag\\&\qquad\qquad\lesssim \|Z^{j_1}\partial_x^{k_1}(\partial_y^{-1}\nabla_x\cdot u)\|_{L^{\infty}_yL^4_x}\|e^{\gamma y^2}Z^{j_2}\partial_x^{k_2}\partial_y v\|_{L^2_yL^4_x}\notag\\&\qquad\qquad\lesssim_{s,\gamma} \|u\|_{X^s_{\gamma}}\|\partial_y v\|_{H^3_{\gamma}}\lesssim_{s,\gamma}\|u\|_{X^s_{\gamma}}\|\partial_yv\|_{X^s_{\gamma}}.
    \end{align}
    For $j_1+|k_1|\ge s-1$, we use Lemma \ref{lemmaA1} combined with \eqref{A9} and \eqref{A10} to obtain
    \begin{align}\label{A10_3}
        &\|e^{\gamma y^2}\langle \nabla_x\rangle^{-\frac{1}{2}}[Z^{j_1}\partial_x^{k_1}(\partial_y^{-1}\nabla_x \cdot u) Z^{j_2}\partial_x^{k_2}\partial_y v]\|_{L^2}\notag\\&\qquad\qquad\lesssim \|Z^{j_1}\partial_x^{k_1}(\partial_y^{-1}\langle\nabla_x \rangle^{\frac{1}{2}}u)\|_{L^{\infty}_yL^2_x}\|e^{\gamma y^2}Z^{j_2}\partial_x^{k_2}\partial_y v\|_{L^2_yH^2_x}\notag\\&\qquad\qquad\lesssim_{s,\gamma}\|\langle\nabla_x \rangle^{\frac{1}{2}}u\|_{X^s_{\gamma}}\|v\|_{X^4_{\gamma}}\lesssim_{s,\gamma}\|\langle\nabla_x \rangle^{\frac{1}{2}}u\|_{X^s_{\gamma}}\|v\|_{X^s_{\gamma}}.
    \end{align}
    Combining \eqref{A10_1}--\eqref{A10_3}, we arrive at
    \begin{align}
        \label{A11}
        \|\langle \nabla_x\rangle^{-\frac{1}{2}}[(\partial_y^{-1}\nabla_x \cdot u) \partial_y v]\|_{H^s_{\gamma,\co}}\lesssim_{s,\gamma} \|u\|_{X^s_{\gamma}}\|\partial_y v\|_{X^s_{\gamma}}+\| \langle\nabla_x \rangle^{\frac{1}{2}}u\|_{X^s_{\gamma}}\|v\|_{X^{s}_{\gamma}}.
    \end{align}
    Next, we note that 
    \[\partial_y[(\partial_y^{-1}\nabla_x \cdot u) \partial_y v]=(\nabla_x \cdot u)\partial_y v+(\partial_y^{-1}\nabla_x \cdot u) \partial^2_y v,\]
    where the first term can be estimated as discussed previously. As for the second term, we derive
    \begin{align*}
        &\|\langle \nabla_x\rangle^{-\frac{1}{2}}[(\partial_y^{-1}\nabla_x \cdot u) \partial^2_y v]\|_{H^{s-1}_{\gamma,\co}}\\&\qquad\qquad\lesssim_s \sum_{|k|+j\le s-1}\sum_{\substack{k_1+k_2=k\\j_1+j_2=j}}\|e^{\gamma y^2}\langle \nabla_x\rangle^{-\frac{1}{2}}[Z^{j_1}\partial_x^{k_1}(\partial_y^{-1}\nabla_x \cdot u) Z^{j_2}\partial_x^{k_2}\partial^2_y v]\|_{L^2},
    \end{align*} 
    where for $j_1+|k_1|\le s-3$, 
    \begin{align*}
        \|e^{\gamma y^2}\langle \nabla_x\rangle^{-\frac{1}{2}}[Z^{j_1}\partial_x^{k_1}(\partial_y^{-1}\nabla_x \cdot u) Z^{j_2}\partial_x^{k_2}\partial^2_y v]\|_{L^2}&\lesssim \|Z^{j_1}\partial_x^{k_1}(\partial_y^{-1}\nabla_x \cdot u)\|_{L^{\infty}}\|e^{\gamma y^2}Z^{j_2}\partial_x^{k_2}\partial^2_y v\|_{L^2}\\&\lesssim_{s,\gamma} \|u\|_{X^s_{\gamma}} \|\partial_y v\|_{X^s_{\gamma}},
    \end{align*}
    and for $j_1+|k_1|\ge s-2$, 
    \begin{align*}
        &\|e^{\gamma y^2}\langle \nabla_x\rangle^{-\frac{1}{2}}[Z^{j_1}\partial_x^{k_1}(\partial_y^{-1}\nabla_x \cdot u) Z^{j_2}\partial_x^{k_2}\partial^2_y v]\|_{L^2}\\&\qquad\qquad\lesssim \|Z^{j_1}\partial_x^{k_1}(\partial_y^{-1}\langle\nabla_x\rangle^{\frac{1}{2}}u)\|_{L^{\infty}_yL^2_x}\|e^{\gamma y^2}Z^{j_2}\partial_x^{k_2}\partial^2_y v\|_{L^2_yH^2_x}\\&\qquad\qquad\lesssim_{s,\gamma}\|\langle \nabla_x\rangle^{\frac{1}{2}}u\|_{H^{s-1}_{\gamma,\co}}\|\partial^2_y v\|_{H^{3}_{\gamma}}\lesssim_{s,\gamma}\|u\|_{X^{s}_{\gamma}}\|\partial_y v\|_{X^{s}_{\gamma}}.
    \end{align*}
    Therefore,
    \begin{align}
        \label{A12}
        \|\langle \nabla_x\rangle^{-\frac{1}{2}}[(\partial_y^{-1}\nabla_x \cdot u) \partial^2_y v]\|_{H^{s-1}_{\gamma,\co}}\lesssim_{s,\gamma}\|u\|_{X^s_{\gamma}}\|\partial_y v\|_{X^{s}_{\gamma}}.
    \end{align}
    Combining \eqref{A11} and \eqref{A12}, we conclude the proof of the last product estimate. 
    \end{proof}

    \subsection{Proof of Proposition \ref{Proposition_struc1}}
    \label{proof_nonlinear_estimate}
    We shall confine ourselves to the case when $\F(x,y,u)$ satisfies \eqref{struc0_1}--\eqref{struc0_3_1}, as the estimates for the generalized linear multiplicative term $\mathbb{G}|\nabla_x|^a u$ follow similarly.

    \noindent {\textbf{$\bullet$ Proof of} \eqref{estimate_noise1}}. By definition \eqref{struc0_1}, 
    \begin{equation}\label{struc1}
        \F(u)-\F(\tilde{u})=\sum_{n\in\mathbb{N}^{d-1}\backslash\{\mathbf{0}\}}a_n(x,y)\sum_{l=1}^{d-1}\left(\prod_{i=1}^{l-1}u_i^{n_i}\right)(u_l^{n_l}-\tilde{u_l}^{n_l})\left(\prod_{i=l+1}^{d-1}\tilde{u_i}^{n_i}\right).
    \end{equation}
     Taking $\mathbb{X}^{s}_{\gamma,\sigma}$-norm on both sides and using Lemmas \ref{product_estimate1} and \ref{product_estimate2}, we have
    \begin{align}
        \label{struc2}
        &\|\F(u)-\F(\tilde{u})\|_{\mathbb{X}^{s}_{\gamma,\sigma}}\notag\\&\quad\le \sum_{n\in\mathbb{N}^{d-1}\backslash\{\mathbf{0}\}}\left(\|a_n^0\|_{\mathbb{X}^s_{\gamma_0,2\sigma_0 }}+\|\overline{a}_n\|_{\mathbb{H}^s_{x,2\sigma_0}}\right)\sum_{l=1}^{d-1}\left\|\left(\prod_{i=1}^{l-1}u_i^{n_i}\right)(u_l^{n_l}-\tilde{u_l}^{n_l})\left(\prod_{i=l+1}^{d-1}\tilde{u_i}^{n_i}\right)\right\|_{X^s_{\gamma,\sigma}}\notag\\&\quad\le \left(\sum_{l=1}^{d-1}\sum_{n\in\mathbb{N}^{d-1}\backslash\{\mathbf{0}\}}n_l\left(\|a_n^0\|_{\mathbb{X}^s_{\gamma_0,2\sigma_0}}+\|\overline{a}_n\|_{\mathbb{H}^s_{x,2\sigma_0}}\right) (C_s M)^{|n|-1}\right)\|u-\tilde{u}\|_{X^s_{\gamma,\sigma}},
    \end{align}
    where $M:=\max\{\|v\|_{X^s_{\gamma,\sigma}},\|\tilde{v}\|_{X^s_{\gamma,\sigma}},\|U\|_{H^s_{x,\sigma}}\}.$ Notice that 
    \begin{align}
        \label{struc3}&\sum_{n\in\mathbb{N}^{d-1}\backslash\{\mathbf{0}\}}n_l\left(\|a_n^0\|_{\mathbb{X}^s_{\gamma_0,2\sigma_0}}+\|\overline{a}_n\|_{\mathbb{H}^s_{x,2\sigma_0}}\right) z^{n-e_l}\notag\\&\qquad\qquad\qquad\qquad=\partial_{z_l}\left(\sum_{n\in\mathbb{N}^{d-1}\backslash\{\mathbf{0}\}}\left(\|a_n^0\|_{\mathbb{X}^s_{\gamma_0,2\sigma_0}}+\|\overline{a}_n\|_{\mathbb{H}^s_{x,2\sigma_0}}\right) z^{n}\right),
    \end{align}
    where $e_l$ denotes the unit vector with the $l$-th component being $1$. As the series on the right hand-side converges absolutely, due to \eqref{struc0_3}, we see that
    \begin{align*}
        &\mathcal{K}_s(\|v\|_{X^s_{\gamma,\sigma}},\|\tilde{v}\|_{X^s_{\gamma,\sigma}},\|U\|_{H^s_{x,\sigma}})\\&\qquad\qquad:=\left(\sum_{l=1}^{d-1}\sum_{n\in\mathbb{N}^{d-1}\backslash\{\mathbf{0}\}}n_l\left(\|a_n^0\|_{\mathbb{X}^s_{\gamma_0,2\sigma_0}}+\|\overline{a}_n\|_{\mathbb{H}^s_{x,2\sigma_0}}\right) (C_s M)^{|n|-1}\right)<\infty,
    \end{align*}
    which implies the estimate \eqref{estimate_noise1}.

    \noindent {\textbf{$\bullet$ Proof of} \eqref{estimate_noise2}}. In view of definition \eqref{struc0_1}, we have
    \begin{align}
        \label{struc4}
        \F(u)-\overline{\F}(U):=\sum_{n\in\mathbb{N}^{d-1}}a_n^0(x,y)u^n+\sum_{n\in\mathbb{N}^{d-1}\backslash\{\mathbf{0}\}}\bar{a}_n(x)(u^n-U^n),
    \end{align}
    where the second term can be handled similarly as presented above:
    \begin{align}
        \label{struc5}
        \left\|\sum_{n\in\mathbb{N}^{d-1}\backslash\{\mathbf{0}\}}\bar{a}_n(x)(u^n-U^n)\right\|_{\mathbb{X}^s_{\gamma,\sigma}}\le \kappa^{'}_s(\|v\|_{X^s_{\gamma,\sigma}},\|U\|_{H^s_{x,\sigma}})\|v\|_{X^s_{\gamma,\sigma}}.
    \end{align}
    Here, $\kappa^{'}_s$ is a positive non-decreasing function. For the first term, we apply Lemmas \ref{product_estimate1} and \ref{product_estimate2} to derive
    \begin{align}
        \label{struc6}
        \left\|\sum_{n\in\mathbb{N}^{d-1}}a_n^0(x,y)u^n\right\|_{\mathbb{X}^s_{\gamma,\sigma}}\le \sum_{n\in\mathbb{N}^{d-1}}\left\|a_n^0(x,y)\right\|_{\mathbb{X}^s_{\gamma_0,2\sigma_0}}(C_s M)^{|n|},
    \end{align}
    where $M:=\|v\|_{X^s_{\gamma,\sigma}}+\|U\|_{H^s_{x,\sigma}}$. By defining
    \[\kappa_s(\|v\|_{X^s_{\gamma,\sigma}},\|U\|_{H^s_{x,\sigma}}):=\kappa^{'}_s(\|v\|_{X^s_{\gamma,\sigma}},\|U\|_{H^s_{x,\sigma}})+\sum_{n\in\mathbb{N}^{d-1}}\left\|a_n^0(x,y)\right\|_{\mathbb{X}^s_{\gamma_0,2\sigma_0}}(C_s M)^{|n|}<\infty,\]
    and combining the estimates \eqref{struc5} and \eqref{struc6}, we get \eqref{estimate_noise2}.

    \noindent {\textbf{$\bullet$ Proof of} \eqref{estimate_noise3}}. Through direct calculation, we see that
    \begin{align}\label{struc7}
       \nabla_x(\F(u)-\overline{\F}(U))&=\sum_{n\in\mathbb{N}^{d-1}\backslash\{\mathbf{0}\}}a_n^0(x,y)\nabla_xu^n+\sum_{n\in\mathbb{N}^{d-1}\backslash\{\mathbf{0}\}}\bar{a}_n(x)\nabla_x(u^n-U^n)\notag\\&\quad+\sum_{n\in\mathbb{N}^{d-1}}\nabla_xa_n^0(x,y)u^n+\sum_{n\in\mathbb{N}^{d-1}\backslash\{\mathbf{0}\}}\nabla_x\bar{a}_n(x)(u^n-U^n).
    \end{align}
    As $\gamma\le\gamma_0$ and $\sigma\le \sigma_0$, we have
    \[\|\nabla_x a_n^0(x,y)\|_{\mathbb{X}^s_{\gamma,\sigma}}\lesssim \|a_n^0(x,y)\|_{\mathbb{X}^s_{\gamma_0,2\sigma_0}},\qquad\|\nabla_x\bar{a}_n(x)\|_{\mathbb{H}^s_{x,\sigma}}\lesssim \|\bar{a}_n(x)\|_{\mathbb{H}^s_{x,2\sigma_0}}.\]
    Thus, repeating the derivation of \eqref{estimate_noise2}, there is a positive non-decreasing function $\kappa^{'}_s$ such that 
    \begin{align}
        \label{struc8}
        &\left\|\sum_{n\in\mathbb{N}^{d-1}}\nabla_xa_n^0(x,y)u^n+\sum_{n\in\mathbb{N}^{d-1}\backslash\{\mathbf{0}\}}\nabla_x\bar{a}_n(x)(u^n-U^n)\right\|_{\mathbb{X}^s_{\gamma,\sigma}}\notag\\&\qquad\qquad\qquad\qquad\qquad\qquad\qquad\le \kappa^{'}_s(\|v\|_{X^s_{\gamma,\sigma}},\|U\|_{H^s_{x,\sigma}})(\|v\|_{X^s_{\gamma,\sigma}}+1).
    \end{align}
    On the other hand, notice that
    \[\sum_{n\in\mathbb{N}^{d-1}\backslash\{\mathbf{0}\}}a_n^0(x,y)\nabla_xu^n+\sum_{n\in\mathbb{N}^{d-1}\backslash\{\mathbf{0}\}}\bar{a}_n(x)\nabla_x(u^n-U^n)=\nabla_u\F(x,y,u)\nabla_x u-\nabla_U\overline{\F}(x,U)\nabla_x U,\]
    where 
    \[\partial_{u_l} \F(x,y,u)=\sum_{n\in\mathbb{N}^{d-1}\backslash\{\mathbf{0}\}}n_la_n(x,y)u^{n-e_l}\]
    with the coefficients $n_la_n(x,y)$ satisfying
    \[\sum_{l=1}^{d-1}\sum_{n\in\mathbb{N}^{d-1}\backslash\{\mathbf{0}\}}n_l(\|a^0_n\|_{\mathbb{X}^s_{\gamma_0,2\sigma_0}}+\|\bar{a}_n\|_{\mathbb{H}^s_{x,2\sigma_0}})\prod_{i=1}^{d-1}|u_i|^{n_i-e^i_l}<\infty.\]
    Then, there exists a positive non-decreasing function $\kappa_s^{''}$ satisfying
    \begin{align*}
        \|\nabla_u\F(x,y,u)-\nabla_U\overline{F}(x,U)\|_{\mathbb{X}^s_{\gamma,\sigma}}+\|\nabla_U\overline{\F}(x,U)\|_{\mathbb{H}^s_{x,\sigma}}\le \kappa_s^{''}(\|v\|_{X^s_{\gamma,\sigma}},\|U\|_{H^s_{x,\sigma}}),
    \end{align*}
    which implies 
    \begin{align}
        \label{struc9}
       \|\nabla_u\F(x,y,u)\nabla_xu-\nabla_U&\overline{\F}(x,U)\nabla_xU \|_{\mathbb{X}^s_{\gamma,\sigma}}\notag\lesssim_s \|\nabla_U\overline{\F}(x,U)\|_{\mathbb{H}^s_{x,\sigma}}\|\nabla_x v\|_{X^s_{\gamma,\sigma}}\\&\quad+\|\nabla_u\F(x,y,u)-\nabla_U\overline{\F}(x,U)\|_{\mathbb{X}^s_{\gamma,\sigma}}(\|\nabla_x v\|_{X^s_{\gamma,\sigma}}+\|\nabla_x U\|_{H^s_{x,\sigma}})\notag\\&\le \kappa_s^{''}(\|v\|_{X^s_{\gamma,\sigma}},\|U\|_{H^s_{x,\sigma}})(\|\nabla_x v\|_{X^s_{\gamma,\sigma}}+\|\nabla_x U\|_{H^s_{x,\sigma}}).
    \end{align}
    Combining the estimates \eqref{struc8} and \eqref{struc9}, we obtain \eqref{estimate_noise3}.

    \section{Existence of analytic solutions of the stochastic Bernoulli law} \label{bernoulli}
    In this appendix, we establish the existence of analytic solutions of the stochastic Bernoulli law \eqref{Intro2}. 
   \begin{prop}
       Let $s\ge 4$, $\sigma_0>0$, and let Assumption \ref{assumption} be satisfied. Suppose that $P$ is a progressively measurable process valued in $H^{s+2}_{x,\sigma_0}$, and the initial data $U_0$ is an $\mathscr{F}_0$-measurable random variable valued in $H^{s+1}_{x,\sigma_0}$. Then, there exist an $\mathscr{F}_0$-measurable random parameter $\lambda>0$ and a unique local pathwise solution $(U,\tau)$ of the stochastic Bernoulli law \eqref{Intro2} such that 
       \[U\in L^{\infty}\left(0,\tau;H^{s}_{x,\sigma(\cdot)}\right),\qquad\PP\mbox{-a.s.,}\]
       where $\sigma(t):=\sigma_0-\lambda t$.
   \end{prop}
   \begin{proof}
   As in the proof of Theorem \ref{MT1}, it suffices to establish a priori estimates for the solution of the following problem for the truncated equation:
    \begin{align}
        \label{bernoulli1}
        \begin{cases}
            \mathrm{d} U+\chi^2(\|U_{\sigma}\|_{H^s_x})(U\cdot\nabla_x U+\nabla_x P)\mathrm{d}t=\chi^2(\|U_{\sigma}\|_{H^s_x})\overline{\F}(U)\mathrm{d}W,\\
            U(0)=U_0,
        \end{cases}
    \end{align}
    where $\chi$ is the cut-off function defined in \eqref{chi}. Assume that
    \begin{align}
        \label{bernoulli1_1}
        \|U_0\|_{H^{s+1}_{x,\sigma_0}}\le \MM ,\qquad \|\nabla_x P|_{t=0}\|_{H^{s+1}_{x,\sigma_0}}\le \NN,
    \end{align}
    for some constants $\MM$ and $\NN$, which can be removed similarly as in Section \ref{proof_LWP}. Define the stopping time 
    \[\mathscr{T}_*:=\inf\left\{t\ge0\big|\|\nabla_x P\|_{H^{s+1}_{x,\sigma_0}}\ge 2\NN \right\}\wedge\frac{\sigma_0}{2\lambda}\wedge1.\]
    Applying the It\^o formula, we get 
    \begin{align}
        \label{bernoulli4}
        &\mathrm{d}\|U_{\sigma}\|^2_{H^s_x}+2\lambda\||\nabla_x|^{\frac{1}{2}}U_{\sigma}\|^2_{H^s_x}\mathrm{d}t+2\chi^2(\|U_{\sigma}\|_{H^s_x})\langle U_{\sigma},(U\cdot\nabla_x U)_{\sigma}+\nabla_x P_{\sigma}\rangle_{H^s_x}\mathrm{d}t\notag\\&\qquad\qquad\qquad=\chi^4(\|U_{\sigma}\|_{H^s_x})\|\overline{\F}_{\sigma}{}(U)\|^2_{\mathbb{H}^s_{x}}\mathrm{d} t+2\chi^2(\|U_{\sigma}\|_{H^s_x})\langle U_{\sigma},\overline{\F}_{\sigma}(U)\mathrm{d}W\rangle_{H^s_{x}}.
    \end{align}
    To bound the convection term, we apply the classical Moser-type estimate, cf. Proposition A.29 in \cite{BV2022}, to get
    \begin{align}
        \label{bernoulli2}
        \|\langle\nabla_x\rangle^{-\frac{1}{2}}(U\nabla_x V)_{\sigma}\|_{H^s_{x}}\lesssim_s \|U_{\sigma}\|_{H^s_{x}}\|V_{\sigma}\|_{H^{s+\frac{1}{2}}_{x}}.
    \end{align}
    In a way similar to the proof of Proposition \ref{Proposition_struc1}, we have
    \begin{align}
        \label{bernoulli3}
        \|\overline{\F}_{\sigma}(U)\|_{\mathbb{H}^s_{x}}\le\mathcal{K}_s(\|U_{\sigma}\|_{H^s_x})\|U_{\sigma}\|_{H^{s+\frac{1}{2}}_{x}},
    \end{align}
    where $\mathcal{K}_s$ is a positive non-decreasing function. This, together with the Burkholder--Davis--Gundy inequality \eqref{pre6}, yields an estimate for the stochastic integral in \eqref{bernoulli4}:
    \begin{align}
        \label{bernoulli6}
        &\mathbb{E}\sup_{t\le \tau}\int_0^t \chi^2(\|U_{\sigma}\|_{H^s_x})\langle U_{\sigma},\overline{\F}_{\sigma}(U)\mathrm{d}W\rangle_{H^s_{x}}\notag\\&\qquad\qquad\qquad\qquad\lesssim \mathbb{E}\left(\int_0^{\tau} \chi^4(\|U_{\sigma}\|_{H^s_x})\|U_{\sigma}\|^2_{H^s_x}\|\overline{\F}_{\sigma}(U)\|^2_{\mathbb{H}^s_{x}}\mathrm{d}t\right)^{\frac{1}{2}}\notag\\&\qquad\qquad\qquad\qquad\le \frac{1}{2}\mathbb{E}\sup_{t\le \tau} \| U_{\sigma}\|^2_{H^s_x}+C\mathbb{E}\int_0^{\tau}\chi^4(\|U_{\sigma}\|_{H^s_x})\|\overline{\F}_{\sigma}(U)\|^2_{\mathbb{H}^s_{x}}\mathrm{d}t.
    \end{align}
    Combining \eqref{bernoulli4}--\eqref{bernoulli6}, we arrive at
    \begin{align}
    \label{bernoulli7}
    &\mathbb{E}\sup_{t\le \tau} \| U_{\sigma}\|^2_{H^s_x}+4\lambda \mathbb{E}\int_0^{\tau}\||\nabla_x|^{\frac{1}{2}}U_{\sigma}\|^2_{H^s_x}\mathrm{d}t\notag\\&\qquad\qquad\qquad\qquad\lesssim_{s,\MM ,\NN }\mathbb{E}\int_0^{\tau}\left( \||\nabla_x|^{\frac{1}{2}}U_{\sigma}\|^2_{H^{s}_x}+\|U_{\sigma}\|^2_{H^s_x}+1\right)\mathrm{d}t+\mathbb{E}\|U_0\|^2_{H^s_{x,\sigma_0}}
    \end{align}
    for $\tau\le \TT_*$. Taking $\lambda$ sufficiently large, we have the following a priori estimate:
    \begin{align}
        \label{bernoulli8}\mathbb{E}\sup_{t\le \tau} \| U_{\sigma}\|^2_{H^s_x}+\mathbb{E}\int_0^{\tau}\||\nabla_x|^{\frac{1}{2}}U_{\sigma}\|^2_{H^s_x}\mathrm{d}t\lesssim_{s,\MM ,\NN }1+\mathbb{E}\|U_0\|^2_{H^s_{x,\sigma_0}}.
    \end{align}
    The following higher-order a priori estimate holds similarly:
    \begin{align}
        \label{bernoulli9}
        \mathbb{E}\sup_{t\le \tau} \| U_{\sigma}\|^2_{H^{s+1}_x}+\mathbb{E}\int_0^{\tau}\||\nabla_x|^{\frac{1}{2}}U_{\sigma}\|^2_{H^{s+1}_x}\mathrm{d}t\lesssim_{s,\MM ,\NN }1+\mathbb{E}\|U_0\|^2_{H^{s+1}_{x,\sigma_0}}
    \end{align}
    These a priori estimates, combined with Lemma \ref{abstract_cauchy_lemma}, ensure the existence of analytic solutions of \eqref{Intro2}. 
    \end{proof}
    The following result ensures the existence of the outflow $(U,P)$ in Theorem \ref{MT1}.
    \begin{coro}\label{bernoullicoro}
       Let $s\ge 4$, $\sigma_0>0$, and let Assumption \ref{assumption} be satisfied. Then, there exists a local pathwise solution of the stochastic Bernoulli law \eqref{Intro2} with spatial regularity $H^{s+2}_{x,\sigma_0}$.
    \end{coro}

    \noindent
    {\bf Acknowledgments:}
    This research was supported by National Key R\&D Program of China under grant 2024YFA1013302, and National Natural Science Foundation of China under Grant Nos. 12331008, 12171317 and 12250710674. The paper was completed during MZ's visit to Universit\'e Paris Cit\'e. He would like to express sincere gratitude to the China Scholarship Council for its financial support, as well as to Professors Sergei Kuksin and Armen Shirikyan for their invitation to Paris, and for the valuable discussions and support throughout the visit.
    
    \vspace{.2in}
    \noindent{\bf Data availability statements:}
    We do not analyse or generate any datasets.

    \vspace{.2in}
    \noindent{\bf Author declarations:} On behalf of all authors, the corresponding author states that there is no conflict of interest.
    \bibliographystyle{alpha}
    \bibliography{reference}
    
    \end{document}